\documentclass[a4paper]{amsart}
\usepackage{amscd,amsmath,amssymb,amsfonts,verbatim}
\usepackage[cmtip, all]{xy}
\usepackage{pgf,tikz}
\usepackage{mathrsfs}
\usetikzlibrary{arrows}

\newtheorem{thm}{Theorem}[subsection]
\newtheorem{prop}[thm]{Proposition}
\newtheorem{lem}[thm]{Lemma}
\newtheorem{cor}[thm]{Corollary}

\newtheorem{ques}[thm]{Question}
\newtheorem{guess}[thm]{Guess}
\theoremstyle{definition}

\newtheorem{defn}[thm]{Definition}
\theoremstyle{remark}
\newtheorem{remk}[thm]{Remark}
\newtheorem{remks}[thm]{Remarks}

\newtheorem{exm}[thm]{Example}
\newtheorem{exms}[thm]{Examples}
\newtheorem{notat}[thm]{Notation}
\numberwithin{equation}{subsection}

{\hfill$\square$\end{defn}}
{\hfill$\square$\end{remk}}
{\hfill$\square$\end{remks}}
{\hfill$\square$\end{exm}}
{\hfill$\square$\end{exms}}
{\hfill$\square$\end{notat}}

\newcommand{\CH}{{\rm CH}}
\newcommand{\TCH}{{\rm TCH}}

\newcommand{\Hom}{{\rm Hom}}

\newcommand{\Spec}{{\rm Spec \,}}

\newcommand{\Tr}{{\rm Tr}}

\newcommand{\ord}{{\rm ord}}

\renewcommand{\max}{{\operatorname{\rm max}}}

\newcommand{\ds}{{/\kern-3pt/}}

\renewcommand{\log}{{\operatorname{log}}}

\newcommand{\Log}{{\operatorname{Log}}}

\newcommand{\TZ}{{\operatorname{Tz}}}
\renewcommand{\TH}{{\operatorname{TCH}}}
\newcommand{\un}{\underline}
\newcommand{\ov}{\overline}

\renewcommand{\dim}{\text{\rm dim}}

\newcommand{\tuborg}{\left\{\begin{array}{ll}}
\newcommand{\sluttuborg}{\end{array}\right.}

\renewcommand{\mod}{ {\rm \ mod \ } }

\begin{document}
\title[de Rham-Witt forms and non-reduced schemes]{The big de Rham-Witt forms over fields and motives of non-reduced schemes}
\author{Jinhyun Park}
\address{Department of Mathematical Sciences, KAIST, 291 Daehak-ro Yuseong-gu, Daejeon, 34141, Republic of Korea (South)}
\email{jinhyun@mathsci.kaist.ac.kr; jinhyun@kaist.edu}




\begin{abstract}

Using algebraic cycles as a medium, we prove that the groups of the big (Hesselholt-Madsen) de Rham-Witt forms over arbitrary fields are isomorphic to the relative improved (Gabber-Kerz) Milnor $K$-groups of Artin local algebras of embedding dimension $1$. This answers an old problem on the relative Milnor $K$-groups studied since 1970s, especially in ${\rm char} (k) = p>0$.

Applications include an interpretation of the big de Rham-Witt forms precisely as the vanishing cycles of the Elmanto-Morrow motivic cohomology of non-reduced schemes, as well as a construction of an extended logarithmic derivative map $d\log$ on the Milnor $K$-theory of some Artin rings to the de Rham-Witt forms.
\end{abstract}

\maketitle

\setcounter{tocdepth}{1}

\tableofcontents

\section{Introduction}

The main objective of this article is to study one old problem on computations of the relative Milnor $K$-theory of some Artin $k$-algebras over a field $k$. The earliest relevant result the author is aware of is the calculation $K_2 ^M (k[\epsilon]/(\epsilon^2), (\epsilon)) \simeq \Omega_{k/\mathbb{Z}} ^1$ in terms of the absolute K\"ahler forms by W. van der Kallen \cite{vdK} in 1971 when ${\rm char} (k) \not = 2$. Since then for about five decades, numerous people including S. Bloch \cite{Bloch 1973}, \cite{Bloch 1975} in the 1970s to B. Dribus \cite{Dribus}, S. Gorchinskiy et. al. \cite{GO}, \cite{GT} in more recent years contributed to the question, under some conditions on the base field $k$.

In this article we furnish a complete answer on all Artin local $k$-algebras of embedding dimension $1$ with the residue field $k$, where $k$ is an arbitrary field of arbitrary characteristic. Specifically,  for integers $m, n \geq 1$, we establish that for $k_{m+1} := k[t]/(t^{m+1})$
$$
\widehat{K}_n ^M (k_{m+1}, (t)) \simeq \mathbb{W}_m \Omega_k ^{n-1},
$$
where the left is the relative improved Milnor $K$-group  defined by O. Gabber and M. Kerz \cite{Kerz finite}, while the right is the group of the big de Rham-Witt forms of Hesselholt-Madsen \cite{HeMa} (see also L. Hesselholt \cite{Hesselholt}).

When ${\rm char} (k) = 0$, or when ${\rm char} (k) \not = 0$ but at least many integers are invertible in $k$, the above isomorphism has been known for several years. Rather unexpectedly, the author learned in 2021 from Matthew Morrow that the case of ${\rm char}(k) = p>0$ had been still open. Morrow asked him if there is a possibility of handling the problem via certain algebraic cycles with which the author has some familiarity. This article grew out of the attempts to answer it. We answer the problem by interpreting it in algebro-geometric terms via algebraic cycles, and our method is characteristic free.

\subsection{Background and some historical accounts}

For a commutative ring $R$ with unity, let $K_n ^M (R)$ denote the $n$-th Milnor $K$-group (recalled in \S \ref{sec:Milnor K}), and let $\Omega_{R/\mathbb{Z}}^n$ denote the $R$-module of the absolute K\"ahler differential $n$-forms. We have the natural $d\log$ homomorphism
$$
d\log: K_n ^M (R) \to \Omega_{R/\mathbb{Z}} ^n
$$
that sends the Milnor symbol $\{ a_1, \cdots, a_n \}$ to $d\log (a_1) \wedge \cdots \wedge d\log (a_n)$, where $d\log (c)$ for $c \in R^{\times}$ means $ \frac{dc}{c}$. This homomorphism is generally far from being an isomorphism, while it was observed that the behavior in the relative situation is often a bit better in the following sense.

For a nilpotent ideal $I \subset R$, we have the relative group $ K_n ^M (R, I) = \ker (K_n ^M (R) \to K_n ^M (R/I))$. In case there is a splitting homomorphism for the surjection $R\to R/I$, by functoriality we obtain the decomposition 
$$
 K_n ^M (R) = K_n ^M (R, I) \oplus K_n ^M (R/I).
 $$ 
 We can similarly define $\Omega_{(R, I)/ \mathbb{Z}} ^n$, and obtain the induced map
$$
d\log: K_n ^M (R, I) \to \Omega_{(R,I)/ \mathbb{Z}} ^n.
$$

Under some extra assumptions such as invertibility of a set of integers in $R$, we have a suitable notion of $``\log" $ for $a \in \ker (R^{\times} \to (R/I)^{\times})$, and by abuse of notations, we may find an ``anti-derivative" of
$$
d \log (a_1) \wedge \cdots \wedge d\log (a_n) = d \left( ``\log" (a_1) d \log (a_2) \wedge \cdots \wedge d \log (a_n) \right),
$$
and this ``anti-derivative" gives a map
$$
B: K_n ^M (R, I) \to \Omega_{ (R,I)/ \mathbb{Z}} ^{n-1} / d \Omega_{(R,I)/\mathbb{Z}} ^{n-2},
$$
to the group modulo exact forms. This kind of idea was sketched in the papers of S. Bloch \cite{Bloch 1973}, \cite{Bloch 1975} when $\mathbb{Q} \subset R$. Thus as an homage to him, we call it \emph{the Bloch map} $B$ whenever such a map exists. This is a $\log$-like map which sends the products of elements in $R^{\times}$ to the addition of forms. Another way to view it is as the Milnor analogue the calculation of the relative Quillen $K$-theory in terms of the relative cyclic homology, see e.g. T. Goodwillie \cite{Goodwillie} (cf. B. Dribus \cite{Dribus}). 

\medskip

Some natural questions arise: 
\begin{enumerate}
\item [(Q1)] When do we have the Bloch map $B$?
\item [(Q2)] If $B$ exists, is it an isomorphism?
\item [(Q3)] If the above kind of map $B$ does not exist, what can we do to improve the situation?
\end{enumerate}

It turns out that for the questions (Q1) and (Q2), there is a sufficiency condition on $R$. Many important results in the literature such as W. van der Kallen \cite{vdK}, \cite{vdK2}, S. Bloch \cite{Bloch 1973}, \cite{Bloch 1975} in the 1970s, to the more recent results of B. Dribus \cite{Dribus}, Gorchinskiy-Osipov \cite{GO} and Gorchinskiy-Tyurin \cite{GT} contributed in this direction.

The latter one \cite{GT} in particular shows that when $R$ is weakly $5$-fold stable, $I^N = 0$, and $\frac{1}{N!} \in R$, then the Bloch map $B$ exists and it is indeed an isomorphism. Recall that a ring $R$ is called \emph{weakly $r$-fold stable} if for any $(r-1)$ members $a_1, \cdots, a_{r-1} \in R$, there exists some $b \in R^{\times}$ such that all $a_i + b \in R^{\times}$. For instance, if $R$ is a field with $|R| \geq 6$, then $R$ is weakly $5$-fold stable. See \cite[Definition 2.2]{GO}, which is equivalent to the one due to M. Morrow \cite[Definition 3.1]{Morrow}, and it traces back to a similar older notion of $r$-fold stability due to W. van der Kallen \cite{vdK2}.

\medskip

While \cite{GT} answered the problem for a large class of cases, the assumption $\frac{1}{N!} \in R$ could be dissatisfactory to some people. For instance, when $k$ is a field of characteristic $p>0$, and $R=k_{m+1}$ with $I= (t)$, the result of \cite{GT} says little about the computation of the relative Milnor $K$-groups $K_n ^M (k_{m+1}, (t))$ for $m \geq p$. Because this is when (Q1) and (Q2) are answered negatively, the situation demands us to think about the plan B, i.e. (Q3).

\medskip

Regardless of these inconveniences, to guess the right question to ask, temporarily review what happened when ${\rm char} (k) = 0$. Here, the Bloch map exists and gives an isomorphism
$$ 
K_n ^M (k_{m+1}, (t)) \simeq \Omega_{(k_{m+1}, (t))/\mathbb{Z}} ^{n-1} /d \Omega_{ (k_{m+1}, (t))/\mathbb{Z}}^{n-2},
$$
where the latter group is isomorphic to the big de Rham-Witt forms 
via
$$
\simeq  \bigoplus_{i=1} ^m t^i \Omega_{k/\mathbb{Z}} ^{n-1} \simeq \bigoplus_{i=1} ^m \Omega_{k/\mathbb{Z}}^{n-1} \simeq  \mathbb{W}_m \Omega_k ^{n-1}.
$$
See, e.g. Gupta-Krishna \cite[Theorem 1.3-(r)]{Gupta-Krishna1}, Iwasa-Kai \cite[Lemma 4.4]{Iwasa-Kai}, Park-\"Unver \cite[Lemma 5.4.3]{PU Milnor}, or Park \cite[Lemma 2.12]{Park AMS}. The intermediate isomorphisms in the above essentially depend on the assumption that ${\rm char} (k) = 0$, under which there exist the formal logarithm and the formal exponential
$$
\Log : ( 1 + t k[[t]])^{\times} \overset{ \sim}{\longrightarrow} (t k[[t]], +),
$$
with ${\rm Exp} = \Log^{-1}$, both of which are isomorphisms of groups (see \eqref{eqn:Log}).

When ${\rm char} (k)= p>0$, we may not have such logarithm or exponential, and the intermediate isomorphisms in the above mostly fail. However, for the two terminal groups, it might be reasonable to ask whether we have an isomorphism via a different path, namely
$$
\widehat{K}_n ^M (k_{m+1}, (t)) \overset{?}{\simeq} \mathbb{W}_m \Omega_k ^{n-1}
$$
for all fields $k$ of arbitrary characteristic. 
In this direction, the best known result so far has been the following (see e.g. Gupta-Krishna \cite[Theorem 1.6]{Gupta-Krishna2} or R\"ulling-Saito \cite{RS}):

\begin{thm}[Gupta-Krishna, R\"ulling-Saito]\label{thm:main pro-iso}
Let $k$ be a field of characteristic $p>0$. Let $n \geq 1$ be an integer.

Then there exists an isomorphism of pro-groups
\begin{equation}\label{eqn:main pro-iso 0}
\left\{ \mathbb{W}_m \Omega_k ^{n-1} \right\}_{m \geq 1} \overset{\simeq}{\longrightarrow} \left\{ \widehat{K}_n ^M (k_{m+1}, (t)) \right\}_{m \geq 1}.
\end{equation}
\end{thm}

As said before, Matthew Morrow enlightened the author that, in case ${\rm char} (k) = p>0$ whether the above pro-group isomorphism could be promoted to a level-wise one had not yet been resolved. 
This question turned out to be deeper than the author initially imagined, and even constructing a level-wise map was challenging in general. For instance, via some general results on cycles with modulus (see Binda-Saito \cite{BS}), Gupta-Krishna \cite{Gupta-Krishna2} constructed homomorphisms $\mathbb{W}_{2m+1} \Omega_k ^{n-1} \to \widehat{K}_n ^M (k_{m+1}, (t))$ for $m \geq 1$ that induce a pro-map, but this is not a level-wise map at all.

\medskip

The central theorem of this article answers this problem, offering an affirmative response to Morrow in particular. We summarize some of them in the following form. All objects in Theorem \ref{thm:intro main vanishing}, except for the one in (5), will be defined or at least sketched in this article:

\begin{thm}\label{thm:intro main vanishing}
Let $k$ be an arbitrary field. Let $m, n \geq 1$ be integers. The following groups are all isomorphic:
\begin{enumerate}
\item $\mathbb{W}_m \Omega_k ^{n-1}:$ the big de Rham-Witt $(n-1)$-forms of modulus $m$ over $k$.
\item $\TH^n (k, n;m):$  the additive higher Chow group of $ \Spec (k)$ of modulus $m$.
\item $\CH_{{\rm v}} ^n (X/ (m+1), n):$ the Chow group  of strict vanishing cycles modulo $I^{m+1}$, where $X=\Spec (A)$ is any integral regular henselian local $k$-scheme of dimension $1$ with the residue field $k=A/I$.
\item $\widehat{K}_n ^M (k_{m+1}, (t)) :$ the relative improved Milnor $K$-group of $k_{m+1}$.
\item $\ker ( {\rm H}_{\mathcal{M}} ^n (\Spec (k_{m+1}), \mathbb{Z} (n)) \to {\rm H}_{\mathcal{M}} ^n (\Spec (k), \mathbb{Z}(n)): $ the vanishing cycles of the Elmanto-Morrow motivic cohomology ${\rm H}_{\mathcal{M}} ^n (\Spec (k_{m+1}), \mathbb{Z} (n))$.

\end{enumerate}
\end{thm}

Here, an isomorphism $(1) \simeq (2)$ is known by K. R\"ulling \cite{R} (see Theorem \ref{thm:Rulling2}), while $(4) \simeq (5)$ immediately follows from Elmanto-Morrow \cite[Theorem 1.8]{EM}. The group in (3) is defined in J. Park \cite{Park presentation}, and an isomorphism $(3) \simeq (4)$ is also established there. Thus the goal of this article can be rephrased as to establish an isomorphism $(2) \simeq (3)$ based on the new definitions and the results of \cite{Park presentation}.

\medskip

One of the nontrivial steps is to construct a level-wise homomorphism
\begin{equation}\label{eqn:intro varphi}
\varphi= \varphi_{m,n, k}: \mathbb{W}_m \Omega_k ^{n-1} \to \widehat{K}_n ^M (k_{m+1}, (t)).
\end{equation}
We call it the \emph{the inverse Bloch map}, and once proven to exist and be invertible, \emph{define} the Bloch map to be $B:=\varphi^{-1}$
$$
B: \widehat{K}_n ^M (k_{m+1}, (t))\overset{\simeq}{ \to} \mathbb{W}_m \Omega_k ^{n-1}.
$$

Later in \S \ref{sec:dlog}, we will interpret $\varphi$ as a kind of the exponential map $\exp_n$ and $B$ as a kind of the logarithm map $\log_n$ for fields $k$ of arbitrary characteristic. It could be of independent interest in the future.

\medskip

Theorem \ref{thm:intro main vanishing} implies the equivalence of two different definitions of the motivic cohomology of the non-reduced scheme $\Spec (k_{m+1})$ given by Krishna-Park \cite{KP Jussieu} and Elmanto-Morrow \cite{EM}:

\begin{cor}\label{cor:main intro motivic}
Let $k$ be an arbitrary field. Let $ m, n \geq 1$ be integers. The following groups are isomorphic to each other:
\begin{enumerate}
\item $\Hom_{\mathcal{DM} (k;m)} ( \underline{\mathbb{Z}}, h (\Spec (k)) (n) [r]):$ the $\Hom$ for the triangulated category $\mathcal{DM} (k,m)$ of the mixed motives over $k[t]/(t^{m+1})$ of Krishna-Park.
\item ${\rm H}_{\mathcal{M}} ^n (\Spec (k_{m+1}), \mathbb{Z} (n)):$ the Elmanto-Morrow motivic cohomology.
\end{enumerate}
\end{cor}

In \cite{KP Jussieu}, it is proven that the $\Hom$ group in (1) is isomorphic to the \emph{total} higher Chow group $\CH^n (\Spec (k), n;m)$ in \emph{ibid.}, which is the direct sum $\CH^n (k, n) \oplus \TH^n (k,n;m)$ of the higher Chow group and the additive higher Chow group (see \S \ref{sec:HC} and \S \ref{sec:ACH}). Combined with Theorem \ref{thm:intro main vanishing}, we deduce that this is isomorphic to the group in (2) because
$$
\CH^n (k, n) \oplus \TH^n (k, n;m) \simeq K_n ^M (k) \oplus \widehat{K}^M_n (k_{m+1}, (t)) = \widehat{K}^M_n (k_{m+1}),
$$
where the first isomorphism follows from Theorem \ref{thm:intro main vanishing} together with the theorem of Nesterenko-Suslin \cite{NS} and Totaro \cite{Totaro} (recalled in Theorem \ref{thm:NST}), and the last group is isomorphic to the Elmanto-Morrow motivic cohomology group ${\rm H}_{\mathcal{M}} ^n (\Spec (k_{m+1})$ in (2) by \cite[Theorem 1.8]{EM}.

\medskip

The big de Rham-Witt complex can also be described in terms of the category of motives with modulus in \cite{KMSY3}. See Koizumi-Miyazaki \cite{KM}. While the algebraic $K$-theory is not representable in the category $\mathcal{DM}(k)$ of motives of Voevodsky, the above suggests that it might be possible that certain relative $K$-theory could be representable in the larger category of motives with modulus. The author would like to understand a clearer picture here in the future.

\medskip

In \S \ref{sec:intro sketch}, we sketch the ideas of the proof of Theorem \ref{thm:intro main vanishing}, and we offer glimpses of some applications in \S \ref{sec:intro applications}. Those who wonder how this article's detour via geometric presentations bypass some technical difficulties for which some past works had to assume invertibility of large number of integers, e.g. as in Gorchinskiy-Tyurin \cite{GT}, can read \S \ref{sec:intro sketch} as well as Remark \ref{remk:GT comparison} for a rapid insight on how our geometric arguments work.

\subsection{Sketch of the main ideas}\label{sec:intro sketch}
First of all, defining a level-wise homomorphism
$$
\varphi: \mathbb{W}_m \Omega_k ^{n-1} \to \widehat{K}_n ^M (k_{m+1}, (t))
$$
 in \eqref{eqn:intro varphi} is itself nontrivial. A typical way to construct a homomorphism from the group $\mathbb{W}_m \Omega_k ^{n-1}$ would be to collect the pro-system of the target groups over $m \geq 1$, and to verify the axioms for a restricted Witt-complex (see Definition \ref{defn:Witt complex}), for which $\{\mathbb{W}_m \Omega_k ^{\bullet}\}_{m \geq 1}$ is an initial object (see \cite[Proposition 1.15]{R})). Unfortunately, this approach requires a rather cumbersome process of proving that we have a few operations such as the Verschiebung $V_r$, the Frobenius $F_r$ for $r \geq 1$, and the differential $d$, etc, and checking all the axioms. This is not at all straightforward to perform on the relative Milnor $K$-groups. (Just in case the reader managed to do it and still want to continue in this direction, one can read \S \ref{sec:addendum} Addendum to see what one should do.)
 
 \medskip

 A key idea of the alternative route we follow is, as briefly mentioned in the paragraph near Theorem \ref{thm:intro main vanishing}, to use good presentations of the groups $\mathbb{W}_m \Omega_k ^{n-1}$ and $\widehat{K}_n ^M (k_{m+1}, (t))$, through which we avoid some exhausting procedures. 
 
\medskip

The first group $\mathbb{W}_m \Omega_k ^{n-1}$ admits the theorem of K. R\"ulling (recalled as Theorem \ref{thm:Rulling2}) which gives a presentation by the additive higher Chow group $\TH^n (k, n;m)$. Here, instead of the notations of the earliest articles, e.g. Bloch-Esnault \cite{BE2}, K. R\"ulling \cite{R} or J. Park \cite{P2}, we use the notations more commonly used these days, e.g. in Krishna-Park \cite{KP crys}.

For the other group $\widehat{K}_n ^M (k_{m+1}, (t))$, we remark that in general finding presentations of certain (relative) $K$-groups has been an important problem since 1970s (see, e.g. W. van der Kallen's 1978 Helsinki ICM lecture \cite{vdK ICM}). We use the recent geometric presentation of the relative Milnor $K$-groups by cycle class groups given in J. Park \cite{Park presentation}, which follows the tradition of the epoch in a sense; the difference is that the generators and relations are free abelian groups given geometrically by algebraic cycles. We use the fact from \cite{Park presentation} (recalled in Theorem \ref{thm:graph final}) that $\widehat{K}_n ^M (k_{m+1}, (t)) \simeq \CH_{{\rm v}} ^n (X/ (m+1), n)$ for the Chow group of \emph{strict vanishing cycles} over an integral regular henselian local $k$-scheme $X$ of dimension $1$ with the residue field $k$. All relevant definitions and results will be recalled in \S \ref{sec:2 recollection}.

\medskip

For the above presentations of $\mathbb{W}_m \Omega_k ^{n-1}$ and $\widehat{K}_n ^M (k_{m+1}, (t))$, the question of defining the algebraic inverse Bloch map $\varphi$ as in \eqref{eqn:intro varphi} is transformed into the question of constructing the geometric (cycle-theoretic) inverse Bloch map on the corresponding cycle class groups (see Proposition \ref{prop:inverse Bloch} and Definition \ref{defn:inverse Bloch})
\begin{equation}\label{eqn:inverse Bloch intro}
\varphi: \TH^n (k, n;m) \to \CH_{{\rm v}} ^n (X/ (m+1), n).
\end{equation}
Eventually we construct this \eqref{eqn:inverse Bloch intro} and prove that it is an isomorphism of groups.

\medskip

The reader may wonder what these strict vanishing cycles in $\CH_{{\rm v}} ^n (X/ (m+1), n)$ defined in Definition \ref{defn:mod t^{m+1} v} are. To give a rough idea, first recall that in the definition of the cubical version of higher Chow groups (see \S \ref{sec:HC}), we use the ambient spaces $\square^n = (\mathbb{P}^1 \setminus \{ 1 \})^n$, so that the ``infinity" of $\square^n$ is given by the union of the divisors $\{y_i = 1\} \subset \overline{\square}^n$ for $1 \leq i \leq n$, where $\overline{\square}= \mathbb{P}^1$. Geometrically speaking, the strict vanishing cycle condition roughly means that the closure of such a cycle intersected with the special fiber over $X$ gives $y_i = 1$ for some $ 1 \leq i \leq n$. This is analogous to the shapes of generators of the relative Milnor $K$-theory, given e.g. by Kato-Saito \cite[Lemma 1.3.1, p.26]{KS} or Gorchinskiy-Tyurin \cite[Lemma 2.2]{GT}:

\begin{lem}[Kato-Saito]\label{lem:KS} Let $R$ be a local ring and let $I \subset R$ be an ideal. Then the relative group $K_n ^M (R, I)$ is generated by the elements of the form $\{ a_1, \cdots, a_n \}$ with $a_i \in R^{\times}$ for all $1 \leq i \leq n$, where for some $1 \leq i_0\leq n$, we have $a_{i_0} \in \ker (R^{\times} \to (R/I)^{\times})$.
\end{lem}

Later in Lemma \ref{lem:KS improved}, when $R= k_{m+1}$, we will have an improvement of Lemma \ref{lem:KS} needed for part of the proof of the main theorem, especially for the surjectivity of the map $\varphi$ in \eqref{eqn:inverse Bloch intro}. The surjectivity is proven by inductive on $n \geq 1$ in \S \ref{sec:surj inverse Bloch}.

\medskip

The injectivity of the map $\varphi$ proven in \S \ref{sec:inj inverse Bloch} requires a nontrivial map called the \emph{de}concatenation 
\begin{equation}\label{eqn:decon intro}
{\rm Dec}: \CH_{{\rm v}} ^n (X, n) \to \CH_{{\rm v}} ^1 (X, 1) \otimes \CH^{n-1} (k, n-1),
\end{equation}
constructed in \S \ref{sec:decon 1}. 

\medskip

 We obtain the eventual {algebraic} inverse Bloch map of \eqref{eqn:intro varphi} between the two groups in the algebraic world via a zigzag of isomorphisms through a journey into the geometric world of algebraic cycles
$$ 
\varphi: \mathbb{W}_m \Omega_k ^{n-1} \overset{gr_k}{\underset{\simeq}{\to}} \TCH^n (k, n;m) \overset{\varphi}{\underset{\simeq}{\to}} \CH_{\rm v} ^n (X/ (m+1), n) \overset{ gr_{{\rm v}}}{\underset{\simeq}{\leftarrow}} \widehat{K}_n ^M (k_{m+1}, (t)),
$$
where $gr_k$ is the isomorphism of K. R\"ulling \cite{R} (see Theorem \ref{thm:Rulling2}), the middle $\varphi$ is the geometric inverse Bloch map in \eqref{eqn:inverse Bloch intro}, and the last $gr_{{\rm v}}$ is an isomorphism of J. Park \cite{Park presentation} (see Theorem \ref{thm:graph final}).

\subsection{Some applications and remarks}\label{sec:intro applications}

 Recall we have the concatenation map
 $$
 {\rm Con}: \mathbb{W}_m (k) \otimes K_{n-1} ^M (k) \to \mathbb{W}_m \Omega_k ^{n-1},
 $$
 defined by sending $\alpha \otimes \{ c_1, \cdots, c_{n-1} \}$ to $\alpha d \log [ c_1] \wedge \cdots \wedge d \log [ c_{n-1}]$, where $\alpha \in \mathbb{W}_m (k)$, $c_i \in k^{\times}$ and $[c_i]$ is the Teichm\"uller lift in $\mathbb{W}_m (k)$ of $c_i$.
 From Theorem \ref{thm:intro main vanishing} and the deconcatenation map \eqref{eqn:decon intro}, we deduce the following map in the opposite direction, though they are not inverse to each other in gneral:
 \begin{cor}
 Let $k$ be a field. Let $m, n \geq 1$ be integers. Then we have
 $$
 {\rm Dec}: \mathbb{W}_m \Omega_k ^{n-1} \to \mathbb{W}_m (k) \otimes K_{n-1} ^M (k).
 $$

 \end{cor}

 \medskip
 
 The following improves the pro-vanishing of the pro-groups $\{ \widehat{K}_n ^M (k_{m+1}, (t)) \}_{m \geq 1}$ and $\{ \widehat{K}_n ^M (k_{m+1})\}_{m \geq 1}$ for a finite field $k$ and $n \geq 2$ proven by Gupta-Krishna \cite[Corollary 4.13]{Gupta-Krishna2} to actual level-wise vanishing and more:

\begin{cor}\label{cor:finite intro}
Let $k$ be a perfect field of characteristic $p>0$, and let $m \geq 1, n \geq 2$ be integers. Then
$$
\widehat{K}_n ^M (k_{m+1}, (t)) = 0, \ \ \mbox{ and } \ \ \ \widehat{K}_n ^M (k_{m+1}) = K_n ^M (k).
$$

In particular, if $k$ is finite, then $\widehat{K}_n ^M (k_{m+1}) =0$.
\end{cor}

For a perfect field $k$ of characteristic $p>0$, we know that $\mathbb{W}_m \Omega_k ^{n-1} = 0$. Thus the main theorem $\mathbb{W}_m \Omega_k ^{n-1} \simeq \widehat{K}_n ^M (k_{m+1}, (t))$ of Theorem \ref{thm:intro main vanishing} and the resulting exact sequence 
$$
0 \to  \mathbb{W}_m \Omega_{k} ^{n-1} \to \widehat{K}_n ^M (k_{m+1}) \to K^M_n (k) \to 0$$
together imply the first assertion of Corollary \ref{cor:finite intro}. The second one follows from the first one because $K_n ^M (k) = 0$ when $k$ is finite and $n \geq 2$. This Corollary shows that $\widehat{K}^M_n (k_{m+1})$ contains more information than $K^M_n (k)$ only when $k$ is imperfect.

\medskip

Recall that we had the map $d\log_k: K_n ^M (k) \to \mathbb{W}_m \Omega_{k} ^n$, see, e.g. Krishna-Park \cite[Corollary 7.5]{KP crys}, sending $\{ c_1, \cdots, c_n \}$ to $d \log [ c_1] \wedge \cdots \wedge d\log [c_n]$, where $[c]$ is the Teichm\"uller lift. We have the following generalization to the Artin $k$-algebra $k_{m+1}$:

\begin{cor}
Let $k$ be a field. Let $m, n \geq 1$ be integers. Then there exist the extended $d\log$ map generalizing $d\log_k : \widehat{K}_n ^M (k) \to \mathbb{W}_m \Omega_k ^n$, denoted by
$$
d\log_t: \widehat{K}_n ^M (k_{m+1}) \to \mathbb{W}_m \Omega_k ^n,
$$
given by sending $(\alpha, \beta) \in K_n ^M (k) \oplus \widehat{K}_n ^M (k_{m+1}, (t))$ to $ d\log_k (\alpha) + d \log_n (\beta) \in \mathbb{W}_m \Omega_k ^n$, where $\log_n$ is the Bloch map
$$
\log_n = B:  \widehat{K}_n ^M (k_{m+1}, (t)) \overset{\sim}{\longrightarrow}  \mathbb{W}_m \Omega_k ^{n-1},
$$
and $d$ in front of $\log_n (\beta)$ is the exterior derivation $d: \mathbb{W}_m \Omega_k ^{n-1} \to \mathbb{W}_m \Omega_k ^{n}$.
\end{cor}

We leave the following remark on cycles with modulus:

\begin{remk}
The additive higher Chow group $\TH^n (k, n;m)$ (see \S \ref{sec:ACH}) is defined using the general position condition, and the modulus condition on cycles on $\mathbb{A}^1 \times \square^{n-1}$ (see Definition \ref{defn:modulus condition}), where the latter condition is rather non-geometric.

On the other hand, the isomorphic group $\CH_{{\rm v}} ^n (X/ (m+1), n)$ uses a different set of conditions, namely the extended general position condition $(GP)_*$, the extended special fiber condition $(SF)_*$, and the vanishing fiber condition (see \S \ref{sec:SF* condition}). All these are entirely geometrically given in terms of intersections. 
\qed
\end{remk}

\medskip

When $R$ is a regular local $k$-algebra essentially of finite type, by Krishna-Park \cite{KP sfs}, \cite{KP crys} we know that $\TH^n (R, n;m) \simeq \mathbb{W}_m \Omega_R ^{n-1}$. Given this, we expect that Theorem \ref{thm:intro main vanishing} may admit the following generalization:

\begin{guess} Let $k$ be a field. Let $m, n \geq 1$ be integers. Let $R$ be a regular local $k$-algebra essentially of finite type and let $R_{m+1} := R[t]/(t^{m+1})$. 

Then there exists an isomorphism, call it the inverse Bloch map $\varphi_R$:
$$
 \varphi_R: \mathbb{W}_m \Omega_R ^{n-1} \simeq \widehat{K}_n ^M (R_{m+1}, (t)).
 $$
 \end{guess}
This is in progress. In the process, Theorem \ref{thm:intro main vanishing} plays a significant role.
Once completed, this could reconnect the two classical articles of S. Bloch \cite{Bloch crys} and L. Illusie \cite{Illusie}, without the necessity of applying the $p$-typicalization.

 \bigskip
 
 \textbf{Conventions:} Unless said otherwise $k$ is an arbitrary field of arbitrary characteristic, a $k$-scheme is a noetherian separated $k$-scheme of finite Krull dimension, but it is not necessarily of finite type over $k$. The fiber product $\times$ means $\times_k$.

\section{Recollections of definitions and results}\label{sec:2 recollection}

 We recall some definitions and results on the objects used in this article.

\subsection{The Milnor $K$-theory}\label{sec:Milnor K}

We recall the functors $K^M_n (-)$ and $\widehat{K}^M_n (-)$ on the category of commutative rings with unity.

\subsubsection{The Milnor $K$-groups}
Recall the classical Milnor $K$-groups of rings, e.g. from J. Milnor \cite{Milnor IM}. Let $R$ be a commutative ring with unity. For the multiplicative abelian group $R^{\times}$, we have the tensor $\mathbb{Z}$-algebra
$$
T_{\mathbb{Z}} R^{\times} := \bigoplus_{n\geq 0} T_n R^{\times},
$$
where $T_0 R^{\times} = \mathbb{Z}$, $T_1 R^{\times} = R^{\times}$, and $T_n R^{\times}= \underset{n}{\underbrace{R^{\times} \otimes \cdots \otimes R^{\times}}}$ for $n \geq 2$. For the two-sided ideal ${\rm St} \subset T_{\mathbb{Z}} R^{\times}$ generated by the elements of the form $a \otimes (1-a)$ over all $a \in R^{\times}$ such that $1-a \in R^{\times}$, the Milnor $K$-ring is the graded ring
$$ 
K_* ^M (R):= T_{\mathbb{Z}} R^{\times} / {\rm St}.
$$
 Its $n$-th graded piece $K_n ^M (R)$ is the $n$-th Milnor $K$-group of $R$. Each $K_n ^M (R)$ is abelian and $K_* ^M (R)$ is a graded-commutative ring. Each $K_n ^M (-)$ is a covariant functor from the category of commutative rings with unity to the category of abelian groups. 

\medskip

For an ideal $I \subset R$, the relative Milnor $K$-group $K_n ^M (R, I)$ is defined to be $ \ker \left( K_n ^M (R) \to K_n ^M (R/I) \right).$ 

\subsubsection{The improved Milnor $K$-groups of Gabber-Kerz}\label{sec:Milnor Kerz}

Recall the improved Milnor $K$-theory of O. Gabber (unpublished) and M. Kerz \cite{Kerz finite}. This is a functor $\widehat{K}_n ^M$ with a surjective transform $K_n ^M ( -) \to \widehat{K}_n ^M (-)$ satisfying a universal property, with respect to the property of having transfers for finite \'etale maps. 
Some essential properties of $\widehat{K}_n ^M$ needed are:

\begin{thm}[{M. Kerz \cite[Proposition 10]{Kerz finite}}]\label{thm:Khat_univ} For local rings $A$, we have:
\begin{enumerate}
\item The natural map $K_n ^M (A) \to \widehat{K}_n ^M (A)$ is surjective.
\item The map in $(1)$ is an isomorphism in the following cases:
\begin{enumerate}
\item When $n=1$, for all local rings $A$.
\item When $A$ is a field, for all $n \geq 1$.
\item There exists a universal integer $M_n>1$ such that when the residue field of $A$ has the cardinality $> M_n$. 
\end{enumerate}
\item For each finite \'etale homomorphism $A \to B$, there exists the norm map ${\rm N}_{B/A} : \widehat{K}^M _n (B) \to \widehat{K}^M_n (A)$, and the norm maps ${\rm N}$ are subject to some properties including the transitivity.
\item (Gersten conjecture) When $A$ is a local domain and $F={\rm Frac} (A)$, the natural map $\widehat{K}^M_n (A) \to \widehat{K}^M_n (F)$ is injective.
\end{enumerate}

\end{thm}

For an ideal $I \subset A$, the relative group $\widehat{K}^M_n (A, I)$ is defined similarly.

\subsection{Big Witt vectors and Big de Rham-Witt forms}

We recall the definitions of the big Witt vectors $\mathbb{W}_m (R)$ and the big de Rham-Witt forms $\mathbb{W}_m \Omega_R ^n$. One may use K. R\"ulling \cite{R} or L. Hesselholt \cite{Hesselholt} as useful references.

\subsubsection{The ring of big Witt vectors}
When $R$ is a commutative ring with unity, let $\mathbb{W}(R):= R^{\mathbb{N}}$ as a set. Define the ghost map
$$ 
gh: \mathbb{W}(R) \to R^{\mathbb{N}},
$$
$$ 
gh ( a_n) = (w_n), \ \ \ \mbox{ where} \ \ \ w_n = \sum_{d | n } d a_d ^{ \frac{n}{d}}.
$$
The target $R^{\mathbb{N}}$ of $gh$ has the coordinate-wise ring structure. A fundamental theorem is that there exists a unique functorial ring structure on $\mathbb{W} (R)$ such that the ghost map $gh$ becomes a functorial homomorphism of rings.

An alternative useful description of the ring structure on $\mathbb{W}(R)$ is the following. Identify $\mathbb{W} (R)$ with $ (1 + t R[[t]])^{\times}$ as a set. Here as groups $(\mathbb{W}(R), +) \simeq (1 + t R[[t]])^{\times}$ while the product on $\mathbb{W}(R)$ is a bit more complicated to describe in terms of $(1 + tR[[t]])^{\times}$. It can be given using:

\begin{prop}[{S. Bloch \cite[I-\S 1-1, p.192]{Bloch crys}}]\label{prop:Witt elements}
Each $x \in (1+ t R[[t]])^{\times}$ can be uniquely expressed as an infinite product
\begin{equation}\label{eqn:Witt infinite}
 x = \prod_{i=1} ^{\infty} (1 - \alpha_i t^i),
 \end{equation}
where $\alpha_i \in R$.
\end{prop}

Under the Proposition \ref{prop:Witt elements}, a product structure $\star$ on $\mathbb{W} (R)$ which is commutative, associative, and distributive over $+$ is determined uniquely by requiring
\begin{equation}\label{eqn:star detail}
(1- a t^m) \star (1- b t^n) = (1- a^{ \frac{n}{r} } b ^{ \frac{m}{r}} t^{\frac{mn}{r}} )^r,
\end{equation}
for all $m, n \geq 1$ and $a, b \in R$, where $r= {\rm gcd} (m,n)$. See \emph{loc.cit.}

In general, finding the sequence of elements $\alpha_1, \alpha_2, \cdots$ in \eqref{eqn:Witt infinite} can be done by elementary but tedious calculations. Even if $x$ is a polynomial in $t$, in \eqref{eqn:Witt infinite} we could possibly have infinitely many non-vanishing entries $\alpha_i$. This can be moderated by taking the quotients by the subgroups $U^m:= (1 + t^{m+1} R[[t]])^{\times}$, namely $\mathbb{W}_m (R):= \mathbb{W} (R) / U^m$. Using \eqref{eqn:star detail}, one checks that the subgroup $U^m \subset \mathbb{W} (R)$ is an ideal. Hence there exists a unique ring structure on $\mathbb{W} _m (R)$ such that the canonical surjection $\mathbb{W} (R) \to \mathbb{W}_m (R)$ is a ring homomorphism.

\subsubsection{The big de Rham-Witt forms}\label{sec:DRW}

Recall the definition of the big de Rham-Witt complexes over a ring $R$. It was originally defined by Hesselholt-Madsen \cite{HeMa}, extending the $p$-typical de Rham-Witt forms of S. Bloch \cite{Bloch crys} and L. Illusie \cite{Illusie}. In this article, we exclusively work with $k$-algebras over a field $k$, so the most general version from L. Hesselholt \cite{Hesselholt} is equivalent to the one from \cite{HeMa}. This description is also found in K. R\"ulling \cite[Definition 1.14]{R}:

\begin{defn}[Hesselholt-Madsen]\label{defn:Witt complex}
Let $k$ be a field and let $R$ be a $k$-algebra. A \emph{restricted Witt complex over $R$} is a pro-differential graded $\mathbb{Z}$-algebra $\{ E_m ^{\bullet} \}_{m \in \mathbb{N}}$, equipped with the restriction homomorphisms $\mathfrak{R}: E_{m+1} ^{\bullet} \to E_m ^{\bullet}$ of the differential graded algebras over $m \in \mathbb{N}$, the graded ring homomorphisms for $r \geq 1$
$$ 
F_r: E_{rm + r -1} ^{\bullet} \to E_m ^{\bullet},
$$
and the graded group homomorphisms for $r \geq 1$
$$ 
V_r: E_m ^{\bullet} \to E_{rm + r -1} ^{\bullet},
$$
satisfying the following requirements:
\begin{enumerate}
\item [(i)] $\mathfrak{R} F_r = F_r \mathfrak{R}_r$, $\mathfrak{R}^r V_r = V_r \mathfrak{R}$, $F_1 = V_1 = {\rm Id}$, $F_r F_s = F_{rs}$, $V_r V_s = V_{rs}$ for all $r, s \geq 1$.
\item [(ii)] $F_r V_r = r$ for all $r \geq 1$ and $F_r V_s = V_s F_r$ for all $r,s \geq 1$ such that $(r,s) = 1$.
\item [(iii)] $V_r (F_r (x) y)  = x V_r (y)$ for $x \in E_{rm + r -1}^{\bullet}$ and $y  \in E_m ^{\bullet}$.
\item [(iv)] $F_r d V_r = d$,
and there exist the homomorphisms of the rings
$$ 
 \{  \lambda : \mathbb{W}_m (R) \to E_m ^0 \}_{ m \geq 1}
 $$
compatible with $F_r$ and $V_r$ for all $r \geq 1$ such that 
\item [(v)] $F_r d \lambda ( [a]) = \lambda ( [a] ^{r-1}) d \lambda ( [a])$ for all $a \in R$ and $r\geq 1$.\qed
\end{enumerate}
\end{defn}

The restricted Witt complexes over $R$ form a category, in which the big de Rham-Witt complexes $\{ \mathbb{W}_m \Omega_R ^{\bullet}\}_{m \geq 1}$ over $R$ are characterized as an initial object (see e.g. \cite[Proposition 1.15]{R}).

While we don't give its full definition, at least we mention that $\mathbb{W}_m \Omega_R ^{\bullet}$ can be expressed as the differential graded algebra $\Omega_{\mathbb{W}_m (R)/ \mathbb{Z}} ^{\bullet}$ modulo a differential graded ideal $\mathcal{N}^{\bullet}$ given by an explicit set of generators, see e.g. \cite[\S 4]{Hesselholt} or \cite[Proposition 1.2]{R}. For any truncation set $S \subset \mathbb{N}$, i.e. a nonempty subset closed under taking all divisors of the members of $S$, we similarly have $\mathbb{W}_S \Omega_R ^{\bullet}$. In case $S= \{ 1, p, p^2, \cdots \}$, we obtain the classical $p$-typical de Rham-Witt complex denoted by $W\Omega_R ^{\bullet}$, while if we have $S= \{ 1, p, \cdots, p^r \}$, then we obtain $W_r \Omega_R ^{\bullet}$. See \cite{Bloch crys} and \cite{Illusie}.

\subsection{Higher Chow groups}\label{sec:HC}

Recall the definition of the cubical version of the higher Chow groups for a scheme $X$ over a field $k$ from the notes of S. Bloch \cite{Bloch notes}. The original simplicial version is from S. Bloch \cite{Bloch HC} and it gives the same cycle class groups. One can read also the introduction of \cite{Park localization} for some general discussions. For simplicity, we suppose that $X$ is equidimensional and of finite Krull dimension.

\medskip

Let $\overline{\square}:= \mathbb{P}_k ^1$ and $\square = \overline{\square}  \setminus \{ 1 \}$. For $n \geq 1$, we let $\square^n$ (resp. $\overline{\square}^n$) be the $n$-fold self-fiber product of $\square$ (resp. $\overline{\square}$) over $k$, while we let $\square^0=\overline{\square}^0:= \Spec (k)$. We use $(y_1, \cdots, y_n) \in \square^n$ (resp. $\in \overline{\square}^n$) for the coordinates.

Write $\square_X ^n:= X \times \square_k ^n$ (resp. $\overline{\square}_X ^n:= X \times \overline{\square} ^n$). If $X= \Spec (A)$, then we often write $\square_A^n$ (resp. $\overline{\square}_A^n$). 

\medskip

Suppose $n \geq 1$. Using the $k$-rational points $\{ 0, \infty \} \subset \square_k $, define the faces (resp. extended faces) as follows: a \emph{face} of $\square_X ^n$ (resp. an \emph{extended face} of $\overline{\square}_X ^n$) is a closed subscheme defined by a finite set of equations of the form $\{ y_{i_1} = \epsilon_1, \cdots, y_{i_s} = \epsilon _s\}$, where $1 \leq i_1 < i_2 < \cdots< i_s \leq n$ with $\epsilon_j \in \{ 0, \infty\}$. Here, we allow $s=0$, i.e. the empty set of equations, in which case the corresponding face (resp. extended face) is the entire space $\square_X^n$ (resp. $\overline{\square}_X ^n$). When $s=1$, we have the codimension $1$ faces (resp. extended faces), and the one given by $\{y_i = \epsilon\}$ is denoted by $F_{i} ^{\epsilon}$ or $F_{i, X} ^{\epsilon}$ (resp. $\overline{F}_i ^{\epsilon}$ or $\overline{F}_{i, X} ^{\epsilon}$). 

\medskip

The extended faces are not needed in \S \ref{sec:HC}, but we use them to improve a result on additive higher Chow groups in \S \ref{sec:ACH}, and they will be essential in \S \ref{sec:our cycle} where we recall the notions of ${\rm d}$-cycles and ${\rm v}$-cycles from \cite{Park presentation}. Especially the conditions $(GF)_*$ and $(SF)_*$, called the extended general position condition, and the extended special fiber condition, respectively, in \S \ref{sec:SF* condition}, are defined in terms of the extended faces.

\begin{defn}[S. Bloch]\label{defn:HC}
Let $ \un{z}^q (X, 0):= z^q (X)$, the group of codimension $q$-cycles on $X$. 

\begin{enumerate}
\item $(GP)$ For $n \geq 1$, let $\un{z}^q (X, n)$ be the subgroup of $z^q (\square_X ^n)$ consisting of the irreducible closed subschemes on $\square_X ^n$ that intersect all faces of $\square_X^n$ properly.

Here, the acronym $(GP)$ stands for \emph{general position}.

\item The subgroup $\un{z}^q (X, n)_{\rm degn} \subset \un{z}^q (X, n)$ of degenerate cycles is generated by the cycles obtained as pull-backs of cycles via various projections $\square_X ^n \to \square_X ^i$, $0 \leq i < n$, defined by omitting some coordinates. We let
$$
 z^q (X, n):= \frac{ \un{z}^q (X, n)}{ \un{z} ^q (X, n)_{\rm degn}}.
$$

\item For each face $F_{i, X} ^{\epsilon}$ of codimension $1$, taking the intersection with the closed immersion $\iota_{i} ^{\epsilon} : \square_X ^{n-1} \hookrightarrow \square_X ^n$ induces a face homomorphism 
$$
 \partial_i ^{\epsilon} : z^q (X, n) \to z^q (X, n-1)
 $$
after identifying $F_{i, X}^{\epsilon}$ with $\square_X ^{n-1}$. The maps $\{ \partial _i ^{\epsilon} \}$ satisfy the cubical identities, so we have a cubical abelian group $\{ \un{n} \mapsto z^q (X, n) \}$. Let $\partial:= \sum_{i=1} ^n (-1)^i (\partial_i ^{\infty} - \partial_i ^0)$. We check that $\partial \circ \partial = 0$ using the usual cubical formalism. The complex $(z^q (X, \bullet), \partial)$ is the higher Chow complex of $X$, and its $n$-th homology is denoted by $\CH^q (X, n)$ and called the $n$-th higher Chow group of $X$ of codimension $q$.\qed
\end{enumerate}
\end{defn}

Recall the following from Nesterenko-Suslin \cite{NS} and B. Totaro \cite{Totaro}:
\begin{thm}[Nesterenko-Suslin, Totaro]\label{thm:NST}
Let $k$ be a field. Then we have an isomorphism $K^M_n (k) \simeq \CH^n (k, n)$ such that the Milnor symbol $\{ a_1, \cdots, a_n \}$ with $a_i \in k^{\times}$ is mapped to the closed point on $\square^n_k$ given by the system $\{ y_1 = a_1, \cdots, y_n = a_n \}$. 
\end{thm}

\subsection{Additive higher Chow group}\label{sec:ACH}

We recall the definition of additive higher Chow groups. For simplicity, in Definition \ref{defn:ACH} we consider only the special cases relevant to this article, namely the additive $0$-cycles and $1$-cycles over $\Spec (k)$.

They were first defined for $0$-cycles in the Milnor range by Bloch-Esnault \cite{BE2} and K. R\"ulling \cite{R}, and were extended to the general range by J. Park \cite{P2}. They were expected to play roles in constructing motivic cohomology over $k_{m+1}$ (cf. Krishna-Park \cite{KP Jussieu}). Some minor extensions to various modulus conditions were first considered in Krishna-Park \cite{KP moving}. Instead of the original notations, we follow the notations  that conform to the present day usages. These are largely consistent with more recent usages in e.g. Krishna-Park \cite[\S 7]{KP DM} and \cite{KP crys}. 

On the other hand, in Definition \ref{defn:ACH2}, we will present a minor modification of the group, which will be proven to be isomorphic to the group in Definition \ref{defn:ACH}.

\medskip

For $n \geq 1$, let $B_n:= \mathbb{A}^1 \times \square^{n-1}$, and let $\widehat{B}_n := \mathbb{P}^1 \times \overline{\square}^{n-1}$. Let $(x, z_1, \cdots, z_{n-1}) \in \widehat{B}_n$ denote the coordinates.  We define a face $F \subset B_n$ (resp. $\overline{F} \subset \widehat{B}_n$) to be a closed subscheme defined by a finite set of equations of the form $\{ z_ {i_1}= \epsilon_1, \cdots, z_{i_s} = \epsilon_s \}$, where $1 \leq i_1 < \cdots < i_s \leq n-1$ and $\epsilon_i \in \{ 0, \infty \}$. We allow $F= B_n$ (resp. $\overline{F} = \widehat{B}_n$) as well.

\begin{defn}[{\cite[Definition 2.3]{KP moving}}] \label{defn:modulus condition}
Let $V  \subset B_n$ be an integral closed subscheme. Let $\overline{V}$ be its Zariski closure in $\widehat{B}_n$. Let 
$$
\nu: \overline{V}^N \to \overline{V} \hookrightarrow \widehat{B}_n
$$
be the composite from the normalization of $\overline{V}$. Let $m \geq 1$ be an integer.
\begin{enumerate}
\item We say that $V$ satisfies the \emph{sum modulus condition $M_{sum}$} with the modulus $m$, if we have on $\overline{V}^N$ 
\begin{equation}
 (m+1) [ \nu^* \{ x = 0 \}] \leq \sum_{i=1} ^{n-1}[  \nu^* \{ z_i = 1 \}].
 \end{equation}
 It was the version considered in Bloch-Esnault \cite{BE2} and K. R\"ulling \cite{R}.
 
\item We say that $V$ satisfies the \emph{strong sup modulus condition} $M_{ssup}$ with the modulus $m$, if for an integer $1 \leq i_0 \leq n-1$, we have on $\overline{V}^N$
\begin{equation}\label{eqn:modulus cond}
(m+1) [ \nu^* \{ x= 0 \} ] \leq  [\nu^* \{ z_{i_0} = 1 \}].
\end{equation}
It was introduced in Krishna-Park \cite{KP moving}.
\end{enumerate}
In general, any of the above modulus conditions implies that $V \cap \{ x= 0 \} = \emptyset$. 
\qed
\end{defn}

\begin{defn}[{\cite[Definition 2.5]{KP moving}}]\label{defn:ACH}
Let $M$ be a modulus condition, either $M_{sum}$ or $M_{ssup}$ in Definition \ref{defn:modulus condition}. Let $n \geq 1$, $m \geq 1$ be integers.
\begin{enumerate}
\item Let ${\TZ}^n (k , n; m)_M$ be the free abelian group on the set of closed points of $B_n$ disjoint from the divisor $\{ x = 0 \}$, and all proper faces of $B_n$.
\item Let $\underline{\TZ}^n (k, n+1;m)_M$ be the free abelian group on the set of integral closed $1$-dimensional subschemes $Z \subset B_{n+1}$ such that
\begin{enumerate}
\item (General Position) $Z$ intersects each face $F \subset B_{n+1}$ properly on $B_{n+1}$.
\item (Modulus condition) $Z$ satisfies the modulus condition $M$.
\end{enumerate}

\item Let $\TZ^n (k, n+1;m)_M:= \frac{ \underline{\TZ}^n (k, n+1;m)_M}{( \underline{\TZ}^n (k, n+1;m)_M)_{\rm degn}}$, modulo the degenerate cycles generated by pull-backs  via the projections $B_{n+1} \to B_{i}$ for$1 \leq i \leq n$ that omit some of the coordinates $\{ z_1, \cdots, z_{n+1} \}$.

\item For each $1 \leq i \leq n$ and $\epsilon \in \{ 0, \infty \}$, we define $\partial_i ^{\epsilon} (Z)$ to be the cycle $[(\iota_i ^{\epsilon})^* (Z)]$ associated to the face map $\iota_i ^{\epsilon}:B_{n} \hookrightarrow B_{n+1}$ which sends $(x, z_1, \cdots, z_{n-1})$ to $(x, z_1, \cdots, z_{i-1}, \epsilon, z_i, \cdots, z_{n-1})$. We have the associated boundary map $\partial = \sum_{i=1} ^{n} (-1)^i ( \partial_i ^{\infty} - \partial_i ^0)$. We let $\TCH^n (k, n;m)_M:= {\rm coker} (\partial : \TZ^n (k, n+1;m)_M \to \TZ^n (k, n; m)_M). $ \qed
\end{enumerate}
\end{defn}

\begin{remk}
Although the notation $\TZ^n(k, n;m)_M$ has the integer $m$ in the expression, in the case of $0$-cycles this group does not depend on $m$.
\qed
\end{remk}

The following Theorem \ref{thm:Rulling} was originally proven for the sum modulus $M_{sum}$ and ${\rm char} (k) \not = 2$ by K. R\"ulling \cite{R}. We reflect some small extensions to $M_{ssup}$ in Krishna-Park \cite[Theorem 3.4]{KP moving} and to ${\rm char} (k) = 2$ in K. R\"ulling \cite[Appendix]{KP crys}:

\begin{thm}[R\"ulling]\label{thm:Rulling}
Let $k$ be an arbitrary field. We have isomorphisms
$$
 \mathbb{W}_m \Omega_k ^{n-1} \underset{gr_k}{\overset{\sim}{\longrightarrow}} \TH^n (k, n;m)_{M_{ssup}} \overset{\sim}{\longrightarrow} \TH^n (k, n;m)_{M_{sum}}.
 $$
\end{thm}

To exploit certain technical simplicities provided by the modulus condition $M_{ssup}$, for the rest of the article, we stick to the modulus condition $M_{ssup}$ whenever we use the additive higher Chow cycles, thus we drop the subscript $M$ from now on.

\medskip

In this article, we need to make the following additional modification. It will be needed later in Lemma \ref{lem:tilde face}-(2) and beyond:

\begin{defn}\label{defn:ACH2}
We continue the notations of Definition \ref{defn:ACH}. Here, we define a subgroup $\TZ^n (k, n+1; m)' \subset \TZ ^n (k, n+1;m)$ as follows:

Let $\un{\TZ}^n (k, n+1;m) '$ be the free abelian group on the set of integral closed $1$-dimensional subschemes $Z \subset B_{n+1}$ such that
\begin{enumerate}
\item [(a)] (Extended general position) The closure $\overline{Z} \subset \widehat{B}_{n+1}$ intersects each extended face $\overline{F} \subset \widehat{B}_{n+1}$ properly on $\widehat{B}_{n+1}$.
\item [(b)] (Modulus condition) $Z$ satisfies the modulus condition $M_{ssup}$.
\end{enumerate}

We denote the group modulo the degenerate cycles by $\TZ^n (k, n+1;m)'$. Restricting the faces $\partial_i ^{\epsilon}$ and the boundary map $\partial$ to the subgroup $\TZ^n (k, n+1;m) '$, we define $\TCH^n (k, n;m)'$ to be ${\rm coker} (\partial: \TZ^n (k, n+1;m)' \to \TZ^n (k, n;m))$.
\qed
\end{defn}

We can improve Theorem \ref{thm:Rulling} to the following:

\begin{thm}[R\"ulling]\label{thm:Rulling2}
Let $k$ be an arbitrary field. Then we have isomorphisms 
$$
\mathbb{W}_m \Omega_k ^{n-1} \overset{\sim}{\to} \TH^n (k, n;m)' \simeq \TH^n (k, n;m).
$$
\end{thm}

\begin{proof}
For ${\rm char} (k) \not = 2$, this is proven essentially by following the outline of the original proof of R\"ulling \cite{R}, and checking that each step there works for the group $\TZ^n (k, n+1;,m)'$. We sketch the arguments.

First of all, by \cite[Theorem 3.6]{R}, there exists a map $\theta: \TZ^n (k, n;m) \to \mathbb{W}_m \Omega_k ^{n-1}$ that maps $\partial \TZ^n (k, n+1;m)$ to zero. In particular, the subgroup $\partial \TZ^n (k, n+1;m)'$ is also mapped to zero. This induces the homomorphisms
$$
\theta_{m,n}: \TH^n (k, n;m)' \to \TH^n (k, n; m) \to \mathbb{W}_m \Omega_k ^{n-1}.
$$

On the other hand, all operations in \cite[Definition-Proposition 3.9]{R} concerning the restricted Witt-complex structure on $\{ \TH^n (k, n;m)'\}_{m, n}$ are defined in the same way. To see that the operations satisfy the axioms of the restricted Witt-complex structure, everything works \emph{mutatis mutandis}, while a point one needs to address is that (cf. around \cite[(3.9.6)]{R}) for the multiplication $\mu: \mathbb{A}^1 \times \mathbb{A}^1 \to \mathbb{A}^1$, a closed point $\mathfrak{p} \in \TZ^n (k, n;m)$ and an integral cycle $C \in \TZ^n (k, n+1;m)'$, the cycle 
$$
 \mu_* ([\mathfrak{p} \times C])
$$ 
satisfies the extended general position condition and the modulus condition. Exactly the same argument given in \emph{ibid.} works for the strong sup modulus condition. For the extended general position condition, since $C$ already satisfies it and since $\mu_* ( [\mathfrak{p} \times - ] )$ scales only the $\mathbb{P}^1$-coordinate of $\widehat{B}_n$ by a constant in an extension field while the extended faces are given by $\overline{\square}^n$-coordinates of $\widehat{B}_{m+1}$, one deduces that the cycle $\mu_* ([\mathfrak{p} \times C])$ does satisfy the extended general position condition as well. 

All other intermediate arguments leading to \cite[Proposition 3.17]{R} work as well, and we deduce the homomorphisms
$$
\phi_{m,n}: \mathbb{W}_m \Omega_k ^{n-1} \to \TH^n (k,n;m)'.
$$
Combined with the natural surjection $\TH^n (k, n;m)' \to \TH^n (k, n;m)$ and the isomorphism of the main theorem of R\"ulling (Theorem \ref{thm:Rulling}), one deduces that $\theta_{m,n} \circ \phi_{m,n}$ and $\phi_{m,n} \circ \theta_{m,n}$ are the identities.

When ${\rm char} (k) = 2$, we use the Appendix of R\"ulling \cite[Appendix]{KP crys} to argue that the above maps exist and inverse to each other.
\end{proof}

Later in Proposition \ref{prop:inverse Bloch general adm}, this group of additive $1$-cycles with the extended general position and the strong sup modulus condition crucially relates the group $\TH^n (k, n;m)'$ and the group $\CH_{{\rm v}} ^n (X/ (m+1), n)$ recalled in \S \ref{sec:our cycle}.

\subsection{${\rm d}$-cycles and ${\rm v}$-cycles}\label{sec:our cycle}

In \S \ref{sec:our cycle}, we recall the notions of ${\rm d}$-cycles and ${\rm v}$-cycles defined in J. Park \cite{Park presentation}. Over some schemes $X$, they form subcomplexes
$$
z_{{\rm v}} ^q (X, \bullet) \subset z_{{\rm d}} ^q (X, \bullet) \subset z^q (X, \bullet)
$$
of the cubical higher Chow complex $z^q (X, \bullet)$ of \S \ref{sec:HC}.

\subsubsection{Extended conditions $(GP)_*$ and $(SF)_*$}\label{sec:SF* condition}

We recall from \cite{Park presentation} the following two conditions given with respect to the Zariski closures of the cycles and the extended faces:

 \begin{defn}[{\cite[Definition 3.3.1]{Park presentation}}]\label{defn:GP*}
 Let $X$ be an integral $k$-scheme of finite dimension. Let $n \geq 0$ be an integer. Let $Z \subset X \times \square^n$ be an integral closed subscheme and let $\overline{Z}$ be its Zariski closure in $X \times \overline{\square}^n$. 
 
 \begin{enumerate}
 \item  We say that $Z$ satisfies the condition $(GP)_*$ if for each face $F \subset \square^n$, the intersection $\overline{Z} \cap (X \times \overline{F})$ is proper on $X \times \overline{\square}^n$.
 
 \item When $Z$ is a cycle on $X \times \square^n$, we say that $Z$ satisfies $(GP)_*$ if each component satisfies the condition $(GP)_*$.   \qed

 \end{enumerate}
 \end{defn}

\begin{defn}[{\cite[Definition 3.3.2]{Park presentation}}]\label{defn:SF* condition}
Let $X$ be an integral $k$-scheme of finite dimension. Let $n \geq 0$ be an integer. Let $Z \subset X \times \square^n$ be an integral closed subscheme and let $\overline{Z}$ be its Zariski closure in $X \times \overline{\square}^n$. Let $p \in X$ be a given fixed closed point.

\begin{enumerate}
\item We say that $Z$ satisfies $(SF)_*$ with respect to $p$, if for each face $F \subset \square^n$, the intersection $\overline{Z} \cap ( \{ p \} \times \overline{F})$ is proper on $X \times \overline{\square}^n$. 

In particular, if $Z$ satisfies $(SF)_*$ and $\overline{Z} \cap (X \times \overline{F}) \not = \emptyset$, then for each component $\overline{W} \subset \overline{Z} \cap (X \times \overline{F})$, the intersection $\overline{W} \cap ( \{ p \} \times \overline{F})$ is proper on $X \times \overline{\square}^n$.

\item If $Z$ is a cycle on $X \times \square^n$, we say that $Z$ satisfies $(SF)_*$ if each component of $Z$ satisfies $(SF)_*$.

\item If $X$ is local, then we drop the reference to the point $p$.
\qed

\end{enumerate}
\end{defn}

Recall from \cite{Park presentation} that the cycles satisfying the condition $(SF)_*$ over a local scheme $X$ possess the following strong properties.

\begin{lem}[{\cite[Lemma 3.3.6]{Park presentation}}]\label{lem:proper int face *}
Let $X$ be an integral local $k$-scheme of finite dimension with the unique closed point $p$. Let $n \geq 1$ be an integer. Let $Z \subset X \times \square^n$ be an integral closed subscheme, and let $\overline{Z}$ be its closure in $X \times \overline{\square}^n$.

Suppose that
\begin{enumerate}
\item $Z \subset X \times \square^n$ is of codimension $n$, and
\item $Z$ satisfies $(SF)_*$.
\end{enumerate}
Then for each proper face $F \subsetneq \square^n$, we have
$$
\overline{Z} \cap (X \times \overline{F}) = \emptyset.
$$
\end{lem}

\begin{lem}[{\cite[Lemma 3.3.7]{Park presentation}}]\label{lem:codim>*}
Let $X$ be an integral local $k$-scheme of finite dimension with the unique closed point $p$. Let $n \geq 0$ be an integer. Let $Z \subset X \times \square^n$ be an integral closed subscheme, and let $\overline{Z}$ be its closure in $X \times \overline{\square}^n$.

Suppose $Z$ satisfies $(SF)_*$. Then the codimension $c$ of $Z$ in $X \times \square^n$ is $\leq n$.
\end{lem}

The condition $(SF)_*$ has another consequence:

\begin{defn}[{\cite[Definition 3.1.4]{Park presentation}}]\label{defn:DF}
Let $X$ be an integral $k$-scheme of finite dimension. Let $n \geq 0$ be an integer. For a closed subscheme $Z \subset X \times \square^n$, we let $\overline{Z}$ be its Zariski closure in $X \times \overline{\square}^n$. 
\begin{enumerate}
\item We say that $Z$ satisfies the property $(DO)$ if for each face $F \subset \square^n$ such that $Z \cap (X \times F ) \not = \emptyset$, each component $W \subset Z \cap (X \times F)$ is dominant over $X$.

\item If $Z$ is a cycle on $X \times \square^n$, we say that $Z$ satisfies $(DO)$ if each component of $Z$ satisfies the property $(DO)$.
\qed
\end{enumerate}
\end{defn}

\begin{lem}[{\cite[Lemma 3.3.8]{Park presentation}}]\label{lem:SF DF}
Let $X$ be an integral local $k$-scheme of dimension $ \leq 1$ with the unique closed point $p $. Let $n \geq 0$ be an integer. Let $Z \subset X \times \square^n$ be a nonempty integral closed subscheme and let $\overline{Z}$ be the Zariski closure in $X \times \overline{\square}^n$.

 If $Z$ satisfies $(SF)_*$, then $Z$ satisfies $(DO)$.
\end{lem}

\subsubsection{The ${\rm d}$-cycles over $X$}\label{sec:dsf}

Recall from \cite{Park presentation} the following notion of ${\rm d}$-cycles based on the discussions in \S \ref{sec:SF* condition}:

\begin{defn}[{\cite[Definition 3.4.1]{Park presentation}}]\label{defn:dsf}
Let $X$ be an integral local $k$-scheme of dimension $ \leq 1$ and let $n,q \geq 0$ be integers.

\begin{enumerate}
\item A cycle $Z \in z^q (X, n)$ is called a \emph{dominant cycle} or a ${\rm d}$\emph{-cycle} if each component of $Z$ satisfies the conditions $(GP)_*$ and $(SF)_*$. 

We know from Lemma \ref{lem:SF DF} that such $Z$ satisfies $(DO)$, thus we gave the name \emph{dominant} cycles.

\item Let $z^q _{{\rm d}} (X, n) \subset z^q (X, n)$ be the subgroup generated by the ${\rm d}$-cycles in $z^q (X, n)$. The boundary maps $\partial_{i} ^{\epsilon} : z^q (X, n+1) \to z^q (X, n)$ for $1 \leq i \leq n+1$ and $\epsilon \in \{ 0, \infty \}$ induce the corresponding boundary maps
$$
\partial_i ^{\epsilon}: z^q _{{\rm d}} (X, n+1) \to z^q _{{\rm d}} (X, n)
$$
so that $z^q_{{\rm d}} (X, \bullet)$ forms a subcomplex of $z^q (X, \bullet)$ with respect to the boundary operator $\partial:= \sum_{i=1}^{n+1} (-1)^i (\partial_i ^{\infty} - \partial_i ^0)$. We call it the (higher) Chow complex of ${\rm d}$-cycles, and we define the (higher) Chow group $\CH^q_{{\rm d}} (X, n)$ of ${\rm d}$-cycles to be the $n$-th homology of $z^q_{{\rm d}} (X, \bullet)$.
\qed
\end{enumerate}
\end{defn}

From Lemma \ref{lem:codim>*}, we deduced that for $q >n$ $z^q_{{\rm d}} (X, n) = 0$. In particular, $\CH^q _{{\rm d}} (X, n) = 0$. Recall from \emph{ibid.} that another aspect immediately coming from the property $(SF)_*$ is:

\begin{lem}\label{lem:specialization}
Let $X$ be an integral local $k$-scheme of dimension $1$ with the unique closed point $p$. Let $q \geq 0$ be an integer. The specialization at $p$ defines the morphism of complexes
$$
{\rm ev}_{p} : z^q _{ {\rm d}} (X, \bullet) \to z^q ( p, \bullet),
$$
so that we have the homomorphisms $\CH_{{\rm d}} ^q (X, n) \to \CH^q (p, n)$.
\end{lem}

\subsubsection{The ${\rm v}$-cycles over $X$}\label{sec:vanishing}

Recall the notion of ${\rm v}$-cycles from \cite{Park presentation}:

\begin{defn}[{\cite[Definitions 3.5.1, 3.5.8]{Park presentation}}]\label{defn:vanishing cycle}
Let $X$ be an integral local $k$-scheme of dimension $1$ with the unique closed point $p$. Let $n, q \geq 1$ be integers.

\begin{enumerate}
\item We say that an integral cycle $Z \in z^q (X, n)$ is a \emph{pre-vanishing cycle}, if the special fiber over $p$ is empty, i.e.  $Z \cap ( \{ p \} \times \square^n) = \emptyset$. 

\item An integral pre-vanishing cycle $Z$ is called a \emph{strict vanishing cycle}, or simply a ${\rm v}$-\emph{cycle}, if it is also a ${\rm d}$-cycle (Definition \ref{defn:dsf}).

\item A cycle $Z \in z^q _{{\rm d}} (X, n)$ is called a ${\rm v}$-\emph{cycle}, if each component is a ${\rm v}$-\emph{cycle}. 
Let $z^q _{\rm v} (X, n) \subset z^q _{{\rm d}} (X, n)$ be the subgroup of the ${\rm v}$-cycles. They form a subcomplex of $z_{{\rm d}} ^q (X, \bullet)$. The $n$-th homology of $z^q _{{\rm v}} (X, \bullet)$ is denoted by $\CH^q _{{\rm v}}(X, n)= {\rm H}_n (z^q _{{\rm v}} (X, \bullet))$, and called the (higher) Chow group of the strict vanishing cycles, or ${\rm v}$-cycles.
\qed
\end{enumerate}
\end{defn}

When $Z \in z^q (X, n)$ is an integral cycle and $\overline{Z} \subset X \times \overline{\square}^n$ is its Zariski closure, by definition, the cycle $Z$ is a pre-vanishing cycle if and only if we have
\begin{equation}\label{eqn:y_i=1 1}
\overline{Z} \cap (\{ p \} \times \overline{\square}^n ) \subset \bigcup _{i=1} ^n (  \{ p \} \times \{ y_i = 1 \}) =  \{ p \} \times \bigcup_{i=1} ^n \{ y_i = 1 \} ,
\end{equation}
where $\{ y_i = 1 \} \subset \overline{\square}^n$ denotes the divisor defined by the equation $y_i = 1$. We had in \cite{Park presentation} the following notion:

\begin{defn}[{\cite[Definition 3.5.4]{Park presentation}}]\label{defn:type i}
Let $Z \in z^q _{{\rm v}} (X, n)$ be an integral ${\rm v}$-cycle. 

For $1 \leq i \leq n$, we say that $y_i$ is a vanishing coordinate for $Z$ if there exists a point $x \in \overline{Z} \cap (\{ p \} \times \overline{\square}^n)$ such that $x \in \{ p \} \times \{ y_i = 1 \}$. There could be more than one such vanishing coordinates in general.
\qed
\end{defn}

\begin{remk}\label{remk:type i}
As seen from \cite[Remark 3.5.5]{Park presentation}, for an integral ${\rm v}$-cycle, there exists at least one vanishing coordinate for $Z$, and this aspect distinguishes ${\rm v}$-cycles from pre-vanishing cycles for which there may be no vanishing coordinate.
\qed
\end{remk}

In addition to the notions of pre-vanishing cycles and strict vanishing cycles, we have the following cohomological notion of vanishing cycles from \cite{Park presentation}: a cycle class, or a cycle that represents a cycle class in $\CH_{{\rm d}} ^q (X, n)$ is called a \emph{vanishing cycle} if it belongs to the kernel of the specialization map of Lemma \ref{lem:specialization}
$$
\CH_{{\rm d}} ^q (X, n) \to \CH^q (p, n).
$$
When $q=n$, we saw in \cite{Park presentation} that this is equivalent to saying that it belongs to the image of the natural map $\CH_{{\rm v}} ^n (X, n) \to \CH_{{\rm d}} ^n (X, n)$. A strict vanishing cycle gives a vanishing cycle, but we saw in \emph{ibid.} that not all representatives of vanishing cycles are strict vanishing cycles.

\medskip

From Lemma \ref{lem:codim>*}, we deduced that $z^q _{{\rm v}} (X, n) = 0$ for $q > n$. In particular, $\CH^q _{{\rm v}} (X, n) = 0$ for $q > n$ as well.

\subsubsection{The mod $I^{m+1}$-equivalence}\label{sec:mod t^{m+1}}

Now suppose that $X= \Spec (A)$ is an integral \emph{henselian} local $k$-scheme of dimension $1$. Let $p \in X$ be the closed point and let $I \subset A$ be the maximal ideal. For each integer $m \geq 1$, let $X_{m+1}:= \Spec (A/ I^{m+1})$. 

We recall an equivalence relation on the cycles $z_{{\rm d}} ^n (X, n)$ from \cite{Park presentation}. This is similar to a notion in Park-\"Unver \cite{PU Milnor} in spirit, while there are some subtle differences. 

This equivalence relation in Definition \ref{defn:mod t^{m+1}} was called the ``\emph{naive} mod $I^{m+1}$-equivalence" in \cite{Park presentation}, and it was proven that this is equivalent to the (true) mod $I^{m+1}$-equivalence defined also in \emph{ibid}. The former is easier to define and work with, while a bit hard to prove a few basic functoriality results. The latter is harder to define but can be used to prove those results. To simplify our presentation, here we present the former naive one only, and we simply recall some relevant needed results from \cite{Park presentation}, namely Lemmas \ref{lem:fpb} and \ref{lem:fpf}.

\begin{defn}[{\cite[Definition 2.3.1]{PU Milnor}}]\label{defn:mod t^{m+1}}
Let $m, n \geq 1$ be integers.

Let $Z_1, Z_2 \in z^n _{{\rm d}} (X, n)$ be two integral ${\rm d}$-cycles.
\begin{enumerate}
\item We say that $Z_1$ and $Z_2$ are \emph{mod $I^{m+1}$-equivalent} and write $Z_1 \sim_{I^{m+1}} Z_2$ if the Zariski closures $\overline{Z}_1, \overline{Z}_2$ in $X \times \overline{\square}^n$ satisfy the \emph{equality} of the closed subschemes of $X_{m+1} \times \overline{\square}^n$
$$
\overline{Z}_1 \times_X X_{m+1} = \overline{Z}_2 \times_X X_{m+1}.
$$
Unlike Park-\"Unver \cite{PU Milnor}, this notion here is defined using the closures. 

If $I= (t)$ is a principal ideal (e.g. when $X$ is regular so that $A$ is a DVR), we may say ``mod $t^{m+1}$-equivalent" synonymously.
\item Let $\mathcal{N}^n (m+1) \subset z^n_{{\rm d}} (X, n)$ be the subgroup generated by the differences $Z_1 - Z_2$ over all mod $I^{m+1}$-equivalent pairs $(Z_1, Z_2)$ with $Z_i \in z_{{\rm d}} ^n (X, n)$. Define $z_{{\rm d}} ^n (X/ (m+1), n):= z_{{\rm d}} ^n (X, n) / \mathcal{N}^n (m+1)$, and
\begin{equation}\label{eqn:CHd N}
\CH^n_{{\rm d}} (X/ (m+1), n)= \frac{ z^n _{{\rm d}} (X, n)}{ \partial z^n_{{\rm d}} (X, n+1) + \mathcal{N}^n (m+1)},
\end{equation}
and it is called the Chow group of the ${\rm d}$-cycles modulo $I^{m+1}$. \qed
\end{enumerate}
\end{defn}

Similarly, recall from \cite{Park presentation} that we induce the mod $I^{m+1}$-equivalence on the ${\rm v}$-cycles as well:

\begin{defn}[{\cite[Definitions 3.7.8, 3.7.12]{Park presentation}}]\label{defn:mod t^{m+1} v}
Let $\mathcal{N}_{{\rm v}} ^n (m+1) \subset z_{{\rm v}} ^n (X, n)$ be the subgroup generated by $Z_1 - Z_2$ over all pairs $(Z_1, Z_2)$ of mod $I^{m+1}$-equivalent integral cycles $Z_1, Z_2 \in z_{{\rm v}} ^n (X, n)$. Define $z_{{\rm v}} ^n (X/ (m+1), n) := z_{{\rm v}} ^n (X, n) / \mathcal{N}_{{\rm v}} ^n (m+1)$, and
$$
\CH_{{\rm v}} ^n (X/ (m+1), n) = \frac{ z_{{\rm v}}^n (X, n) }{ \partial z_{{\rm v}} ^n (X, n+1) + \mathcal{N}_{{\rm v}} ^n (m+1)}
$$
similar to \eqref{eqn:CHd N}.
\qed
\end{defn}

\subsubsection{Some pull-backs and push-forwards}

Recall some results on pull-backs and push-forwards. 
We recall just the cases needed for this article:

\begin{lem}[{\cite[Lemma 3.8.1]{Park presentation}}]\label{lem:fpb}
Let $X, Y$ be integral local $k$-schemes of dimension $\leq 1$ with the closed points $p_1, p_2$, respectively. Let $f: X \to Y$ be a flat surjective morphism of $k$-schemes. 
\begin{enumerate}
\item We have the induced flat pull-back morphism $f^*: z_{{\rm d}} ^q (Y, \bullet) \to z_{{\rm d}} ^q (X, \bullet).$
\item Suppose $\dim \ X = \dim \ Y = 1$. Then the restriction of the above $f^*$ on the subcomplex $z_{{\rm v}} ^q (Y, \bullet)$ induces
$
f^*: z_{{\rm v}} ^q (Y, \bullet) \to z_{{\rm v}} ^q (X, \bullet).
$
\item Suppose $Y= \Spec (A)$ is an integral henselian local $k$-scheme of dimension $1$ with the residue field $k=A/I$. For $X:= Y_{k'}$, let $f: X \to Y$ be the base change map for a finite extension $k \hookrightarrow k'$ of fields. Then in the Milnor range we have 
$f^*: \CH_{{\rm d}} ^n (Y/ (m+1), n) \to \CH_{{\rm d}} ^n (X/ (m+1), n)$ and $f^*: \CH_{{\rm v}} ^n (Y/ (m+1), n) \to \CH_{{\rm v}} ^n (X/ (m+1), n)$. 
\end{enumerate}
\end{lem}

\begin{lem}[{\cite[Lemma 3.8.2]{Park presentation}}]\label{lem:fpf}

Let $X, Y$ be integral local $k$-schemes of dimension $\leq 1$, and let $p_1, p_2$ be their closed points, respectively. Let $f: X \to Y$ be a finite surjective morphism of $k$-schemes.
\begin{enumerate}
\item We have the induced finite push-forward morphism
$
f_* : z_{{\rm d}} ^q (X, \bullet) \to z_{{\rm d}} ^q (Y, \bullet).
$

\item Suppose $\dim \ X = \dim \ Y = 1$. If we restrict the above $f_*$ to the subcomplex $z_{{\rm v}} ^q (X, \bullet)$, then it induces
$
f_*: z_{{\rm v}} ^q (X, \bullet) \to z_{{\rm v}} ^q (Y, \bullet).
$

\item Suppose $Y= \Spec (A)$ is an integral henselian local $k$-scheme of dimension $1$ with the residue field $k= A/I$. For $X:= Y_{k'}$, let $f: X \to Y$ is the base change for a finite extension $k\hookrightarrow k'$ of fields. Then in the Milnor range we have: $f_*: \CH_{{\rm d}} ^n (X/ (m+1), n) \to \CH_{{\rm d}} ^n (Y/ (m+1), n)$ and $f_*: \CH_{{\rm v}} ^n (X/ (m+1), n) \to \CH_{{\rm v}} ^n (Y/ (m+1), n)$. 
\end{enumerate}

\end{lem}

\subsubsection{The graph maps and transfer}\label{sec:4 graph}

Recall that we have the following theorem on the graph homomorphisms:

\begin{thm}[{\cite[\S 4.1, \S 5.5]{Park presentation}}]\label{thm:graph final}
Let $k$ be an arbitrary field. Let $m, n \geq 1$ be integers. Let $X= \Spec (A)$ be an integral henselian local $k$-scheme of dimension $1$ with the residue field $k=A/I$. Then:
\begin{enumerate}
\item There exist the graph homomorphisms 
$$
gr_{{\rm d}}: \widehat{K}_n ^M (A) \to \CH_{{\rm d}} ^n (X, n), \ \ \ gr_{{\rm d}}: \widehat{K}_n ^M (A/I^{m+1}) \to \CH_{{\rm d}} ^n (X/ (m+1), n)
$$
and
$$
 gr_{{\rm v}}: \widehat{K} ^M _n (A, I) \to \CH_{{\rm v}} ^n (X, n), \ \ \  gr_{{\rm v}}: \widehat{K}_n ^M (A/ I^{m+1}, I) \to \CH_{{\rm v}} ^n (X/ (m+1), n),
$$
given by sending the symbol $\{ a_1, \cdots, a_n \}$ with $a_i \in A^{\times}$ to the closed subscheme $ \Gamma_{ (a_1, \cdots, a_n) }$ in $X \times \square^n$ defined by $\{ y_1 = a_1, \cdots, y_n = c_n \}$. 

\item Furthermore, if $X$ is regular in addition, then all of the above maps are isomorphisms. 
\end{enumerate}
\end{thm}

Note that $A/ I ^{m+1} \simeq k_{m+1}$. Thus Theorem \ref{thm:graph final} is a generalization to $k_{m+1}$ of the theorem of Nesterenko-Suslin \cite{NS} and Totaro \cite{Totaro}, recalled in Theorem \ref{thm:NST}.

\medskip

\begin{thm}[{\cite[Theorem 1.1.2]{Park presentation}}]\label{thm:full transfer}
Let $m, n \geq 1$ be integers. Then for each finite extension of fields $k \hookrightarrow k'$, there exist the push-forward maps 
\begin{eqnarray*}
{\rm N}_{k'/k} = \pi_*: \widehat{K}^M_n (k_{m+1}') \to \widehat{K}^M_n (k_{m+1}), \\
 {\rm Tr}_{k'/k}=\pi_*: \widehat{K}^M_n (k_{m+1}', (t) ) \to \widehat{K}^M_n (k_{m+1}, (t)),
\end{eqnarray*}
where $\pi: \Spec (k') \to \Spec (k)$ is the associated morphism of schemes.
\end{thm}

\subsubsection{Compactified projections}\label{sec:compact proj}
We recall the notion of the compactified projections from \cite[\S 5.1]{Park presentation}. 

\medskip

Let $X= \Spec (A)$ be an integral henselian local $k$-scheme of dimension $1$ with the residue field is $k = A/ I$. For a nonempty subset $J \subset \{ 1, \cdots, n \}$, consider the projection 
$
\widehat{pr}_J: \overline{\square}_X ^n \to \overline{\square}_X ^{|J|}
$
that ignores the coordinates $y_i$ for $i \not \in J$.

\begin{defn}\label{defn:comp proj 0}
Let $W \subset \square_X ^n$ be an integral closed subscheme and let $\overline{W}$ be its Zariski closure in $\overline{\square}_X ^n$. The image of the projection is denoted by
$$
\overline{W}^{(J)}:= \widehat{pr}_J (\overline{W}) \subset  \overline{\square}_X ^{ |J|}.
$$
It is an integral closed subscheme. Its restriction to the open subscheme $\square^{|J|}_X$ is denoted by
$$
W^{(J)}:= \widehat{pr}_J (\overline{W}) |_{ \square^{|J|}_X}.
$$
 We call them the \emph{compactified projections of $W$} to the coordinates of $J$.
 
 In case $J = \{ 1, \cdots i \}$ for some $1 \leq i \leq n$, we simply write $\overline{W}^{(i)}$ and $W^{(i)}$. 
\qed
\end{defn}

For the cycles in $z^n_{{\rm d}} (X, n)$ and $z^n_{{\rm d}} (X, n+1)$, we proved in \cite[\S 5.1]{Park presentation} that their compactified projections behave well. We recall them:

\begin{lem}[{\cite[Lemma 5.1.2, Corollary 5.1.3]{Park presentation}}]\label{lem:adm compact proj}
Let $Z \in z^n _{{\rm d}} (X, n)$ be an integral cycle. Let $J \subset \{ 1, \cdots, n\}$ be a nonempty subset. Then $ Z^{(J)} \in z^{|J|}_{{\rm d}} (X, |J|).$

If $Z \in z_{{\rm v}} ^n (X, n)$ and $i \in J$ for a vanishing coordinate $y_i$ for $Z$, then $Z^{(J)} \in z_{{\rm v}} ^{|J|} (X, |J|)$.
\end{lem}

\begin{lem}[{\cite[Lemma 5.1.4]{Park presentation}}]\label{lem:adm comp proj 2}
Let $W \in z_{{\rm d}} ^n (X, n+1)$ be an integral cycle. Let $J$ be a nonempty subset of $ \{ 1, \cdots, n+1\}$ such that $\overline{W} \to \widehat{pr}_J (\overline{W})$ is quasi-finite. Then $W ^{(J)} \in z_{{\rm d}} ^{ |J|-1 } ( X, |J|)$.

In addition to the above assumptions, if $W \in z_{{\rm v}} ^n (X, n+1)$ and $i \in J$ for a vanishing coordinate $y_i$ for $W$, then $W^{ (J)} \in z_{{\rm v}} ^{ |J|-1} (X, |J|)$. 
\end{lem}

\section{The inverse Bloch map}\label{sec:inverse Bloch general}

From \S \ref{sec:inverse Bloch general}, we start the new contributions of this article. Here, we present a geometric construction of the inverse Bloch map $\varphi$. The algebraic version would have taken the form 
\begin{equation}\label{eqn:6 inverse Bloch}
\varphi: \mathbb{W}_m \Omega_k ^{n-1} \longrightarrow \widehat{K}_n ^M (k_{m+1}, (t)) 
\end{equation}
for any field $k$ of arbitrary characteristic, and any integers $m, n \geq 1$. As mentioned in the introduction, it is nontrivial to define $\varphi$ directly, so we replace the groups in \eqref{eqn:6 inverse Bloch} by certain cycle class groups which offer their presentations. 

More precisely, the former group $\mathbb{W}_m \Omega_k ^{n-1}$ is replaced by the additive higher Chow group $\TH^n (k, n;m)'$ using K. R\"ulling \cite{R} (Theorem \ref{thm:Rulling2}), and the latter group $\widehat{K}_n ^M (k_{m+1}, (t))$ is replaced by the Chow group $\CH_{\rm v}^n (X/ (m+1), n)$ of strict vanishing cycles using J. Park \cite{Park presentation} (Theorem \ref{thm:graph final}).

For the latter cycle class group, we have some flexibility in choosing $X$. So for the rest of the article, we take $X= \Spec (k[[t]])$ to facilitate our arguments. Let $p \in X$ denote the unique closed point.

\subsection{A geometric construction}\label{sec:geo construct}

Let $ n \geq 1$ be an integer. In the definition of $\TH^n (k, n;m)'$ in Definition \ref{defn:ACH2}, we used the ambient spaces $B_n = \mathbb{A}^1 \times \square^n$ and $\widehat{B}_n = \mathbb{P}^1 \times \overline{\square}^n$. Let $(x, z_1, \cdots, z_{n-1}) $ be the coordinates for $B_n$ and $\widehat{B}_n$. Let $(y_1, \cdots, y_n ) $ be the coordinates of $ \square_X ^n = X \times \square^n_k$. If needed, we use the projective coordinates $(Y_{i,0}; Y_{i,1}) \in \overline{\square}^1= \mathbb{P}^1$ with $y_i = Y_{i,0} / Y_{i,1}$ for $1 \leq i \leq n$.

\medskip

Let $n \geq 1$ be an integer. For the open subscheme $\widehat{B}_n ^0:=\mathbb{G}_m \times \overline{ \square}^{n-1} \subset \widehat{B}_{n} = \mathbb{P}^1 \times \overline{\square}^{n-1}$, form the product $\widehat{B}_n ^0 \times X \times \square^n$. Consider the codimension $n$ closed subscheme $\Lambda \subset \widehat{B}_n ^0 \times X \times \square^n$ defined by the system of equations
$$
\Lambda: \ \ \left\{ y_1 = 1 - \frac{t}{x} , \ \ y_2 = z_1 - xt , \ \ \cdots, \ \ y_n = z_{n-1} - xt \right\}.
$$
Let $\overline{\Lambda}$ be its Zariski closure in $ \widehat{B}_n \times X \times \square^n.$

\medskip

Consider the projections $pr_1, pr_2: \widehat{B}_n ^0 \times X \times \square^n $ to $\widehat{B}_n ^0$ and $X \times \square^n$, respectively. Let $pr_1 ^{\Lambda}, pr_2 ^{\Lambda}$ be their restrictions to $\Lambda$. Similarly, let $pr_1 ^{\overline{\Lambda}}, pr_2 ^{\overline{\Lambda}}$ be the restrictions to $\overline{\Lambda}$ of the projections $\widehat{B}_n \times X \times \square^n$ to $\widehat{B}_n$ and $X \times \square^n$, respectively. 

Since $\widehat{B}_n$ is projective over $k$, the morphism $pr_2 ^{\overline{\Lambda}}$ is projective. However $pr_2 ^{\Lambda}$ is not projective, and this may cause some concerns, for instance in taking the push-forwards of cycles. 

However, for a closed subscheme $V \subset \Lambda$, if the restriction $pr_2 ^{\Lambda}|_V$ is projective, then we can still take the push-forward. This sort of idea was used in handling certain algebraic cycles in the past, in e.g. Krishna-Park \cite[Lemma 3.4]{KP DGA}. We will try something similar.

\medskip

The composite $\Lambda \overset{pr_1 ^{\Lambda}}{ \to } \widehat{B}_n ^0  \hookrightarrow \widehat{B}_n$ of the projection $pr_1 ^{\Lambda}$ with the open immersion will also be denoted by $pr_1 ^{\Lambda}$. This coincides with the restriction to $\Lambda$ of $pr_1 ^{\overline{\Lambda}}$. Now we state:

\begin{lem}\label{lem:additive proj} Let $n \geq 1$ be an integer. Let $Z \subset B_n$ be an integral closed subscheme. Let $\overline{Z}$ be its Zariski closure in $\widehat{B}_n$. 

Suppose $Z$ satisfies the property 
\begin{equation}\label{eqn:additive proj 0}
(\star_n) : \ \ \  \overline{Z} \times_{\widehat{B}_n} \overline{\Lambda} = \overline{Z} \times_{\widehat{B}_n} \Lambda.
\end{equation}

Then the restriction to $\overline{Z} \times_{\widehat{B}_n} \Lambda$ of $pr_2 ^{\Lambda}$
is projective. Furthermore, for such $Z$, we have 
\begin{equation}\label{eqn:star2}
\overline{Z} \times_{\widehat{B}_n} \Lambda = \overline{Z}^0 \times_{\widehat{B}_n ^0} \Lambda,
\end{equation}
where $\overline{Z}^0$ is the restriction of $\overline{Z}$ to $\widehat{B}_n ^0$.

In other words, we have the following Cartesian diagram
\begin{equation}\label{eqn:Lambda fiber}
\xymatrix{
\overline{Z} ^0 \times_{\widehat{B}_n^0} \Lambda \ar[d] \ar@{^{(}->}[rr]^{\iota^{\Lambda}} & & \Lambda \ar[d] ^{pr^{\Lambda}_{1}} \ar[rr]^{ pr^{\Lambda}_{2}} & & X \times \square^n \\
 \overline{Z}^0 \ar@{^{(}->}[rr]^{\iota} & &  \widehat{B}_n^0, & }
\end{equation}
and the top horizontal composite is projective, where $\iota$ is the closed immersion and $\iota^{\Lambda}$ is the induced closed immersion.
\end{lem}

\begin{proof}
The morphism $pr_2 ^{\overline{\Lambda}}$ is projective so that its restriction to $\overline{Z} \times_{\widehat{B}_n} \overline{\Lambda}$ is projective. However, by the assumption $(\star_n)$ in \eqref{eqn:additive proj 0}, we have
$$ 
pr_2 ^{\overline{\Lambda}} |_{ \overline{Z} \times_{\widehat{B}_n} \overline{\Lambda}} = pr_2 ^{\Lambda}|_{\overline{Z} \times_{\widehat{B}_n} \Lambda}
$$
so that $pr_2 ^{\Lambda}|_{\overline{Z} \times_{\widehat{B}_n} \Lambda}$ is projective.

Since $\Lambda \subset \widehat{B}_n ^0 \times X \times \square^n$, under $(\star_n)$ we have \eqref{eqn:star2}.
\end{proof}

\begin{defn}\label{defn:inverse Bloch general}
Let $Z \subset B_n$ be an integral closed subscheme satisfying $(\star_n)$ of \eqref{eqn:additive proj 0}. Let $[ \overline{Z} ^0 \times_{\widehat{B}_n^0} \Lambda]$ be the associated cycle. Define $\varphi (Z)= \varphi (\overline{Z})$ to be the projective push-forward
$$
\varphi (Z) =\varphi(\overline{Z}):= pr^{\Lambda}_{2 *} ([ \ov{Z} ^0\times_{\widehat{B}_n^0} \Lambda] ) \ \ \mbox{ on } X \times \square^n
$$
defined by Lemma \ref{lem:additive proj}. Under the given assumption, the dimension of this cycle on $X \times \square^n$ is $d+1$ if $d:= \dim \ Z$ by a simple dimension counting. 
\qed
\end{defn}

What is wonderful is that additive higher Chow cycles in Definition \ref{defn:ACH2} satisfy the condition:

\begin{lem}\label{lem:additive star}
\begin{enumerate}
\item Let $Z \in \TZ^n (k, n;m)$ be an integral cycle. Then it satisfies the condition $(\star_n)$ of \eqref{eqn:additive proj 0}.
\item Let $Z \in \TZ^n (k, n+1;m)'$ be an integral cycle. Then it satisfies the condition $(\star_{n+1})$ of \eqref{eqn:additive proj 0}.
\end{enumerate}
\end{lem}

\begin{proof}
(1) If $Z \in \TZ^n (k, n;m)$, then it is a closed point in $\mathbb{G}_m \times \square^{n-1}$ disjoint from all divisors given by $z_i = \epsilon$ with $\epsilon \in \{ 0, \infty\}$, whose closure in $\widehat{B}_n$ is still $Z$ itself, which is in $\mathbb{G}_m \times \square^{n-1} \subset \widehat{B}_n ^0$. So, $\overline{Z} \times_{\widehat{B}_n} \overline{\Lambda} = \overline{Z}^0 \times_{\widehat{B}_n ^0} \Lambda = Z \times_{\widehat{B}_n} \Lambda$, and the condition $(\star_n)$ holds.

\medskip

(2) Assume $Z \in \TZ^n (k, n+1; m)'$. Here, we replace $n$ in Lemma \ref{lem:additive proj} by $n+1$. Consider the fiber product $\overline{Z} \times_{\widehat{B}_{n+1}} \overline{ \Lambda}$. Since $\Lambda \subset \widehat{B}_{n+1} ^0 \times X \times {\square}^{n+1}$ and $\overline{\Lambda} \subset \widehat{B}_{n+1}  \times X \times \overline{\square}^{n+1}$, to compare $\overline{Z} \times_{\widehat{B}_{n+1}} \overline{ \Lambda}$ with $\overline{Z} \times_{\widehat{B}_{n+1}} { \Lambda}$, we need to inspect the points of $\overline{Z} \times_{\widehat{B}_{n+1}} \overline{ \Lambda}$ whose $x$-coordinates in $\mathbb{P}^1$ of $\widehat{B}_{n+1}$ are either $0$ or $\infty$.

\medskip

Suppose there is a point $\mathfrak{q} \in \overline{Z} \times_{\widehat{B}_{n+1}} \overline{ \Lambda}$ whose $x$-coordinate is $\infty$. Note that the closure $\overline{ \overline{Z} \times_{\widehat{B}_{n+1}} \Lambda}$ in $\overline{\Lambda}$ is $\overline{Z} \times_{\widehat{B}_{n+1}} \overline{\Lambda}$. By the first defining equation $y_1 = 1 - \frac{t}{x}$ of $\Lambda$, we have $y_1 = 1$ at $\mathfrak{q}$. Since $\square= \mathbb{P}\setminus \{1 \}$, such a point in $\widehat{B}_{n+1} \times X \times \square^{n+1}$ does not exist. So such $\mathfrak{q}$ does not exist.

\medskip

Suppose there is a point $\mathfrak{q} \in \overline{Z} \times_{\widehat{B}_{n+1}} \overline{\Lambda}$ whose $x$-coordinate is $0$. Take the normalization $\nu: \overline{Z} ^N \to \overline{Z}$ and consider the induced surjection 
$$
\overline{Z}^N \times_{\widehat{B}_{n+1}} \overline{\Lambda} \to \overline{Z} \times_{\widehat{B}_{n+1}} \overline{\Lambda}.
$$
So, there is a point $\mathfrak{q} ' \in \overline{Z}^N \times_{\widehat{B}_{n+1}} \overline{\Lambda} $ that maps to $\mathfrak{q}$, and the $x$-coordinate of $\mathfrak{q}'$ is $0$. Let $\mathfrak{p}' \in \overline{Z}^N$ be its image. By the modulus condition satisfied by $Z$, this implies that there exists some $1 \leq i_0 \leq n$ such that $z_{i_0} = 1$ at this point $\mathfrak{p}'$. Then from the equation $y_{i_0+1} = z_{i_0} - x t$ of $\Lambda$, at the point $\mathfrak{q}'$ we have $y_{i_0+1} = 1 - 0 \cdot t = 1$. Hence at $\mathfrak{q}$, we also have $y_{i_0+1 } = 1$ and such point does not exist in $\widehat{B}_{n+1} \times X \times \square^{n+1}$ because $\square= \mathbb{P}^1 \setminus \{ 1 \}$. So such $\mathfrak{q}$ does not exist.

\medskip

Since there is no point $\mathfrak{q} \in \overline{Z} \times_{\widehat{B}_{n+1}} \overline{\Lambda}$ whose $x$-coordinates are in $\{ 0, \infty \}$, this fiber product actually belongs to $\Lambda \subset \widehat{B}_{n+1} ^0 \times X \times \square^{n+1}$. Hence we deduce $\overline{Z} \times_{\widehat{B}_{n+1}} \overline{\Lambda} = \overline{Z} ^0 \times_{\widehat{B}_{n+1}^0} \Lambda= \overline{Z} \times_{\widehat{B}_{n+1}} \Lambda$, which is the equality $(\star_{n+1})$, as desired.
\end{proof}

Combining Lemmas \ref{lem:additive proj} and \ref{lem:additive star}, we deduce:

\begin{cor}\label{cor:additive varphi}
\begin{enumerate}
\item Let $Z$ be an integral cycle in $\TZ^n (k, n;m)$. Then the push-forward $\varphi (Z) $ in Definition \ref{defn:inverse Bloch general} is defined as a cycle on $X \times \square^n$. 
\item Let $Z$ be an integral cycle in $ \TZ^n (k, n+1;m)'$. Then the push-forward $\varphi (Z)$ is defined on $X \times \square^{n+1}$.
\end{enumerate}
\end{cor}

\subsection{Additive $0$-cycles}

Let $\mathfrak{p} \in \TZ^n (k, n;m)$ be an integral cycle. It is a closed point $\mathfrak{p} \in B_n$ such that $\mathfrak{p} \not \in \{ x = 0 \}$ and $\mathfrak{p} \not \in \{ z_i = \epsilon \}$, for all $1 \leq i \leq n-1$ and $\epsilon \in \{ 0, \infty \}$. 

Take $Z= \mathfrak{p}$ in Definition \ref{defn:inverse Bloch general}. Here, we have $\ov{\mathfrak{p}} = \overline{\mathfrak{p}} ^0= \mathfrak{p}$ in $\widehat{B}_n^0$. By Corollary \ref{cor:additive varphi}, the cycle $\varphi(\mathfrak{p})$ on $X \times \square^n$ exists, and 
$
\ov{Z}^0 \times_{\widehat{B}_n^0} \Lambda = ( pr^{\Lambda}_{1})^{-1} (\mathfrak{p}) = pr_1 ^{-1} (\mathfrak{p}) \cap \Lambda.
$

\begin{defn}
We denote by
\begin{equation}\label{eqn:varphi n 0}
\varphi_n (\mathfrak{p}) := \varphi (\mathfrak{p})  \ \ \mbox{ on } X \times \square^n.
\end{equation}

It is an $1$-dimensional cycle on $X \times \square^n$, thus of codimension $n$.
\qed
\end{defn}

\begin{lem}\label{lem:inverse 0-cycle}
Let $\mathfrak{p} \in \TZ^n (k, n;m)$ be an integral cycle, Then $\varphi_n (\mathfrak{p}) \in z_{{\rm v}} ^n (X, n).$
\end{lem}

\begin{proof}
Let $k':= \kappa (\mathfrak{p})$ and let $\pi : \Spec (k') \to \Spec (k)$ be the morphism corresponding to the finite extension $k \hookrightarrow k'$. Then there exists a $k'$-rational point $\mathfrak{p}' \in \TZ^n (k', n; m)$ such that $\pi_* ( [ \mathfrak{p}']) = [ \mathfrak{p}]$, and we can write
\begin{equation}\label{eqn:0-cycle ex 1}
\mathfrak{p}' = \left( a, b_1, \cdots, b_{n-1} \right) \in ( \mathbb{A}^1 \times \square^{n-1}) (k'),
\end{equation}
for some $a, b_1, \cdots, b_{n-1} \in (k')^{\times}$ with $b_i \not = 1$ for all $i$. Consider the commutative diagram where all squares are fiber squares:
\begin{equation}\label{eqn:0-cycle ex 2}
\xymatrix{
 (pr_1 ^{\Lambda_{k'}})^{-1} (\mathfrak{p}') \ar[d] ^{\pi} \ar@{^{(}->}[rr]^{\ \ \ \ \iota^{\Lambda_{k'}}} & & \Lambda_{k'} \ar[d]^{\pi} \ar[rr] ^{ pr_2 ^{\Lambda_{k'}}} & & X_{k'} \times_{k'} \square_{k'}^n \ar[d] ^{\pi} \\
 (pr_1 ^{\Lambda_{k}})^{-1} (\mathfrak{p})   \ar[d] \ar@{^{(}->}[rr]^{\ \ \ \ \iota^{\Lambda}} & & \Lambda \ar[d]^{pr_1 ^{\Lambda}}  \ar[rr] ^{ pr_2 ^{\Lambda}} & & X \times \square^n \\
\{ \mathfrak{p} \} \ar@{^{(}->}[rr] & &  \widehat{B}_n^0 ,  & & }
\end{equation}
where $X_{k'}:= \Spec (k[[t]])$. Using the expression \eqref{eqn:0-cycle ex 1} over $k'$, we see that $ pr_{2*} ^{\Lambda_{k'}} [ (pr_1 ^{\Lambda_{k'}})^{-1} (\mathfrak{p}')] $ is given by the  codimension $n$ closed subscheme
$$
\left\{ y_1 = 1 - \frac{t}{a}, y_2 = b_1 - a t , \cdots, y_{n} = b_{n-1} -a t \right\} \subset X_{k'} \times_{k'} \square^n _{k'},
$$
which belongs to $z_{{\rm v}} ^n (X_{k'}, n)$ by Theorem \ref{thm:graph final}. Applying the push-forward $\pi_*$ to this cycle, from the fiber diagram \eqref{eqn:0-cycle ex 2} we have $ \pi_* \circ  (pr_{2*} ^{\Lambda_{k'}})  =(pr_{2*} ^{\Lambda}) \circ  \pi_* $ and
$$
\varphi_n (\mathfrak{p}) = \pi_* \left\{ y_1 = 1 - \frac{t}{a}, y_2 = b_1 -at , \cdots, y_{n} = b_{n-1} -at \right\}  \in z_{{\rm v}} ^n (X, n)
$$
by Lemma \ref{lem:fpf}-(2).
\end{proof}

\begin{defn}\label{defn:inverse Bloch 1}
By Lemma \ref{lem:inverse 0-cycle}, we deduce the homomorphism
\begin{equation}\label{eqn:inverse Bloch 1}
\varphi_n: \TZ^n (k,n;m) \to z^n_{{\rm v}} (X, n)
\end{equation}
by $\mathbb{Z}$-linearly extending the map \eqref{eqn:varphi n 0}.
\qed
\end{defn}

A similar argument implies the following:

\begin{cor}
Let $k \hookrightarrow k'$ be a finite extension of fields, and let $\pi: \Spec (k') \to \Spec (k)$ be the corresponding morphism. 
Then the following diagram commutes:
$$
\xymatrix{
\TZ^n (k', n; m) \ar[d] ^{\pi_*} \ar[rr] ^{ \varphi_{n, k'}} & & z_{{\rm v}} ^n (X_{k'}, n) \ar[d] ^{\pi_*}\\
\TZ^n (k, n; m) \ar[rr] ^{\varphi_{n,k}} & & z_{{\rm v}} ^n (X, n).}
$$
\end{cor}

\subsection{Additive $1$-cycles} We now apply Definition \ref{defn:inverse Bloch general} to the $1$-cycles in $\TZ^n (k, n+1;m)'$. For a psychological reason, we denote the integral additive $1$-cycles by $W$ instead of the letter $Z$.

\medskip

Let $W \in \TZ^n (k, n+1; m)'$ be an integral cycle. Take $Z=W$ in Definition \ref{defn:inverse Bloch general} with $n$ replaced by $n+1$. This time the inclusion $W \subset \overline{W}$ may not be an equality in general, unlike the case of $0$-cycles. By Corollary \ref{cor:additive varphi}, we obtain the cycle $\varphi (W)$ on $X \times \square^{n+1}$ of dimension $2$.

\begin{defn}
In what follows, often we denote $\varphi (W)$ by
\begin{equation}\label{eqn:varphi n+1 0}
\widetilde{\varphi}_{n+1} (W):= \varphi (W) \subset X \times \square^{n+1},
\end{equation}
to distinguish it from $\varphi_n$ in \eqref{eqn:varphi n 0}.
\qed
\end{defn}

We want to prove that $\varphi (W)$ is a ${\rm v}$-cycle. This requires some works and we complete its proof in Proposition \ref{prop:inverse Bloch general adm}. We begin with a discussion on its behaviors with respect to the codimension $1$ extended faces of $X \times \overline{\square}^{n+1}$.

\begin{lem}\label{lem:tilde face}
Let $W \in \TZ^n (k, n+1;m)'$ be an integral cycle, and let $\iota_{i,X} ^{\epsilon} : \square^{n}_X \hookrightarrow {\square}^{n+1}_X$ be the closed immersion giving the codimension $1$ face $F_i ^{\epsilon}=\{ y_i = \epsilon \}$ for $1 \leq i \leq n+1$ and $\epsilon \in \{ 0, \infty \}$. The immersion $\overline{\square}^n _X\hookrightarrow \overline{\square}^{n+1}_X$ for the extended face $\overline{F}_i ^{\epsilon}$ is also denoted by $\iota_{i,X} ^{\epsilon}$. Similarly the closed immersions $\widehat{B}_n ^0 \hookrightarrow \widehat{B}_{n+1} ^0$ for the codimension $1$ face $\{ z_i = \epsilon \}$ for $1 \leq i \leq n$ and $\epsilon \in \{ 0, \infty \}$ is denoted by $\iota_i ^{\epsilon}$. Then:
\begin{enumerate}
\item The intersection $\overline{\varphi (W)} \cap (X \times \overline{F})$ is proper on $X \times \overline{\square}^{n+1}$.

\item Suppose $2 \leq i \leq n$ and $\epsilon \in \{ 0, \infty \}$. Consider the Cartesian square
$$
\xymatrix{
\overline{W}^0 \times_{\widehat{B}_{n+1} ^0} \widehat{B}_n ^0 \ar@{^{(}->}[d]^{\iota_{i-1} ^{\epsilon}} \ar@{^{(}->}[rr] & & \widehat{B}_n ^0 \ar@{^{(}->}[d] ^{\iota_{i-1} ^{\epsilon}} \\
\overline{W}^0 \ar@{^{(}->}[rr] & & \widehat{B}_{n+1} ^0. }
$$
Then inside $\widehat{B}_{n+1}$ we have $\overline{W}^0 \times_{ \widehat{B}_{n+1} ^0} \widehat{B}_n ^0 = W \cap F_{i-1} ^{\epsilon},$ where $F_{i-1} ^{\epsilon} \subset B_{n+1}$ is the face given by $\{ z_{i-1} = \epsilon \}$.

\item Furthermore, for $2 \leq i \leq n$ and $\epsilon \in \{ 0, \infty \}$, we have
$$
\overline{\varphi (W)} \cap (X \times \overline{F}_i ^{\epsilon}) = \overline{ \varphi (W \cap F_{i-1} ^{\epsilon})},
$$
where the first $X \times \overline{F}_i ^{\epsilon}$ is the extended face of $X \times \overline{\square}^{n+1}$, identified with $X \times \overline{\square}^n$, and the second $F_{i-1} ^{\epsilon}$ is the face in $B_{n+1}$.

In particular, on the open subscheme $X \times \square^{n}$ identified with $X \times F_i ^{\epsilon}$,
$$
\partial_i ^{\epsilon} ( \varphi (W) )= \varphi ( \partial_{i-1} ^{\epsilon} (W)),
$$
where $\partial_i ^{\epsilon} (\varphi (W)) =[( \iota_{i} ^{\epsilon}) ^* (\varphi (W))]= [ \varphi (W) \cap (X \times F_i ^{\epsilon}) ].$
\end{enumerate}
\end{lem}

\begin{proof}
(1) Note that $X \times \overline{F}$ is an effective integral divisor of $X\times \overline{\square}^{n+1}$. From the shape of the defining equations of $\Lambda$, we have $\overline{\varphi (W)} \not \subset X \times \overline{F}$. Thus the intersection $\overline{\varphi (W) }\cap (X \times \overline{F})$ is indeed proper on $X \times \overline{ \square}^{n+1}$.

\medskip

(2) We first make the following observation. Recall that the modulus condition forces $W$ to satisfy $W \subset \mathbb{G}_m \times \square^n$. Furthermore, if $\overline{W} \subset \mathbb{P}^1 \times \overline{\square}^n$ has a point $\mathfrak{p} \in \overline{W}$ whose $z_j$-coordinate is $1$ for some $1 \leq j \leq n$, then we must have $\mathfrak{p} \in \overline{W} \setminus W$. Hence its $x$-coordinate is either $0$ or $\infty$. In particular, for every point $\mathfrak{p} \in \overline{W} ^0$, its $z_j$-coordinate is not $1$ for all $1 \leq j \leq n$.

\medskip

Let's now prove the assertion that $\overline{W}^0 \times_{ \widehat{B}_{n+1} ^0} \widehat{B}_n ^0 = W \cap F_{i-1} ^{\epsilon}$ for $2 \leq i \leq n+1$. Since $\overline{W}^0 \times_{ \widehat{B}_{n+1} ^0} \widehat{B}_n ^0 = \overline{W} ^0 \cap \overline{F}_{i-1} ^{\epsilon}$, where $\overline{F}_{i-1} ^{\epsilon} \subset \widehat{B}_{n+1} ^0$ is the face given by $\{ z_{i-1} = \epsilon \}$, the inclusion $(\supset)$ is apparent. 

Let's prove the inclusion $(\subset)$. Toward contradiction, suppose there is a point $\mathfrak{p} \in \overline{W}^0 \cap \overline{F}_{i-1} ^{\epsilon}$ that does not belong to $W \cap F_{i-1} ^{\epsilon}$. By the above observation, no $z_j$-coordinate of the point $\mathfrak{p}$ is $1$. On the other hand, identifying $F_{i-1} ^{\epsilon}$ with $B_{n}$, the points in $W \cap F_{i-1} ^{\epsilon}$ give $\partial_{i-1} ^{\epsilon } (W) \in \TZ^n (k, n;m)$ so that every point of $\partial_{i-1} ^{\epsilon} (W)$ belongs to $\mathbb{G}_m \times ( \mathbb{G}_m) ^{n-1}$. In particular, written in terms of the points on $\overline{W} \cap \overline{F}_{i-1} ^{\epsilon}$, for the point $\mathfrak{p}$, there must exist some $j \not = i-1$ such that its $z_j$-coordinate is $\epsilon' \in \{ 0, \infty\}$. This means
$$
\mathfrak{p} \in \overline{W} \cap \overline{F}_{i-1} ^{\epsilon} \cap \overline{F}_{j} ^{\epsilon'}.
$$

However, this contradicts the extended general position condition in Definition \ref{defn:ACH2} for $\overline{W}$ which implies that the intersection is empty, being of codimension $2$ in $\overline{W}$ with $\dim \ \overline{W} = 1$. Hence we must have $(\subset)$, proving (2).

\medskip

(3) To compute the intersection for $2 \leq i \leq n+1$, we first consider the following diagram where all squares are fibre squares:
\begin{small}$$
\xymatrix{ 
& \overline{W} ^0 \times_{\widehat{B}_{n+1}^0 } {\Lambda}_{n+1}  \ar@{^{(}->}[rr] \ar[dd] & & {\Lambda}_{n+1} \ar[dd] \ar[rr]^{pr_2 ^{ {\Lambda}_{n+1}}} & & X \times \overline{\square}^{n+1} \\
\overline{W}^0 \times _{\widehat{B}_{n+1}^0} {\Lambda}_{n} \ar@{^{(}->}[ru]^{\iota_{i,X} ^{\epsilon}} \ar[dd] \ar@{^{(}->}[rr] & & {\Lambda}_{n} \ar[dd]  \ar[rr] ^{ \ \ \ \ \ \ \ pr_2 ^{ {\Lambda}_{n}}} \ar@{^{(}->}[ru] ^{\iota_{i,X} ^{\epsilon}}& & X \times \overline{\square}^{n} \ar@{^{(}->}[ru]^{\iota_{i,X} ^{\epsilon}} &\\
& \overline{W}^0 \ar@{^{(}->}[rr] & & \widehat{B}_{n+1} ^0 & & \\
\overline{W}^0 \times_{\widehat{B}_{n+1}^0} \widehat{B}_{n}^0 \ar@{^{(}->}[ru]^{\iota_{i-1} ^{\epsilon}} \ar@{^{(}->}[rr] & & \widehat{B}_{n}^0, \ar@{^{(}->}[ru] ^{\iota_{i-1} ^{\epsilon}}& & & }
$$
\end{small}
where the bottom square is the Cartesian square of (2). (One can check that all squares are Cartesian one-by-one directly. For instance, the reader might wonder if the middle left corner vertex is indeed the pull-backs of all three adjacent squares. Indeed, by a direct computation we have the isomorphisms
\begin{small}
$$
\begin{cases}
\left( \overline{W} ^0 \times_{\widehat{B}_{n+1}^0 } {\Lambda}_{n+1} \right) \times _{{\Lambda}_{n+1}} {\Lambda}_{n} \simeq \overline{W}^0 \times_{\widehat{B}_{n+1}^0} {\Lambda}_{n}, \\
\left( \overline{W} ^0 \times _{\widehat{B}_{n+1}^0 } \widehat{B}_{n}^0 \right) \times_{\overline{W}^0} \left( \overline{W}^0 \times_{\widehat{B}_{n+1}^0} {\Lambda}_{n+1} \right) \simeq \overline{W}^0 \times_{\widehat{B}_{n+1}^0} \widehat{B}_{n}^0 \times_{\widehat{B}_{n+1}^0} {\Lambda}_{n+1} \simeq ^{\dagger} \overline{W} ^0 \times_{\widehat{B}_{n+1}^0} {\Lambda}_{n},\\
\left( \overline{W}^0  \times_{\widehat{B}_{n+1}^0} \widehat{B}_{n} ^0 \right) \times_{\widehat{B}_{n}^0} {\Lambda}_{n} \simeq \overline{W}^0  \times_{\widehat{B}_{n+1}^0} {\Lambda}_{n},
\end{cases}
$$
\end{small}
where $\simeq^{\dagger}$ follows from that the middle vertical square is Cartesian.
Likewise one can inspect the compatibility at every vertex.)

Coming back to the proof of (3), we see that
$$
\overline{\varphi (W)} \cap \overline{F}_{i, X} ^{\epsilon} = (\iota_{i,X } ^{\epsilon})^*  \overline{ pr_{2*} ^{{\Lambda}_{n+1}} ([\overline{W}^0 \times_{\widehat{B}_{n+1}^0} {\Lambda}_{n+1}])} = \overline{pr_{2*} ^{{\Lambda}_{n}}( \iota_{i, X} ^{\epsilon})^* ( [ \overline{W}^0 \times_{\widehat{B}_{n+1}^0} {\Lambda}_{n+1}]) }
$$
$$
= \overline{ pr_{2*} ^{ {\Lambda}_{n}} ([ \overline{W}^0 \times_{\widehat{B}_{n+1}^0 } {\Lambda}_{n}])} = \overline{\varphi (  \overline{W}^0 \times_{\widehat{B}_{n+1}^0} \widehat{B}_{n}^0})=^{\dagger} \overline{\varphi ( {W} \cap {F}_{i-1} ^{\epsilon})},
$$
where $\dagger$ holds by (2). 
This proves (3).
\end{proof}

As remarked before, the group $\TZ^n (k, n+1; m)'$ uses the strong sup modulus condition of Definition \ref{defn:modulus condition}. This modulus condition for $W$ plays significant roles in the following Proposition \ref{prop:inverse Bloch general adm} and Corollary \ref{cor:tilde W face y_1}.

\begin{prop}\label{prop:inverse Bloch general adm}
Let $W \in \TZ^n (k, n+1;m)'$ be an integral cycle. Then, $\varphi (W) \in z_{{\rm v}}^n (X, n+1).$
\end{prop}

\begin{proof}

\textbf{Step 1:} We first prove the condition $(GP)_*$ for $\varphi (W)$ that for each face $F \subset \square^{n+1}$, the intersection $\overline{\varphi (W)} \cap (X \times \overline{F})$ is proper on $X \times \overline{\square}^{n+1}$. This is apparent when $F= \square^{n+1}$. Suppose that $F \subset \square^{n+1}$ is a proper face. 

\medskip

\textbf{Case 1:} If $F$ is of codimension $1$, it was checked in Lemma \ref{lem:tilde face}-(1).

\textbf{Case 2:} Let $F$ be a codimension $2$ face, We can then write $\overline{F}= \overline{F}_{i_1} ^{\epsilon_1} \cap \overline{F}_{i_2} ^{\epsilon_2}$ for two distinct pairs $(i_1, \epsilon_i), (i_2, \epsilon_2)$. 

\textbf{Subcase 2-1:} If $i_1= i_2$ and $\epsilon_1 \not = \epsilon_2$, then $F= \emptyset$ and $\overline{F} = \emptyset$. So there is nothing to prove. 

\textbf{Subcase 2-2:} If $i_1 < i_2$, then we have $1 < i_2$ at least. First consider $\overline{ \varphi (W)} \cap (X \times \overline{F}_{i_2} ^{\epsilon_2})$. This is equal to $\overline{\varphi ( W \cap F_{i_2 -1} ^{\epsilon_2})}$ by Lemma \ref{lem:tilde face}-(3). By Lemma \ref{lem:inverse 0-cycle}, each component of $\varphi (W \cap F_{i_2 -1} ^{\epsilon_2})$ is in $z_{{\rm v}} ^n (X, n)$ so that it satisfies the condition $(SF)_*$. Hence we have $\overline{ \varphi ( W \cap F_{i_2-1} ^{\epsilon_2})} \cap ( X \times \overline{F}_{1_1} ^{\epsilon_1}) = \emptyset$ by Lemma \ref{lem:proper int face *}. Hence
$$
\overline{\varphi (W)} \cap (X \times \overline{F}) = \overline{ \varphi (W)} \cap (X \times \overline{F}_{i_2} ^{\epsilon_2}) \cap (X \times \overline{F}_{i_1} ^{\epsilon_2})
$$
$$
= \overline{ \varphi (W \cap F_{i_2 -1} ^{\epsilon_2})} \cap (X \times \overline{F}_{i_1} ^{\epsilon_1} )= \emptyset.
$$

\textbf{Case 3:} Let $F$ be a codimension $\geq 3$ face. Then we can write $\overline{F}= \overline{F}_1 \cap \overline{F}_2$, for two proper faces $F_1, F_2 \subset \square^{n+1}$, where $F_1$ itself has the codimension $2$. Since $\overline{\varphi (W)} \cap (X \times \overline{F}_1) = \emptyset$ by \textbf{Case 2}, we deduce that $\overline{\varphi (W)} \cap (X \times \overline{F}) = \emptyset$.

From the above cases, we deduce that $\varphi (W)$ satisfies the condition $(GP)_*$.

\bigskip

\textbf{Step 2:} We now check the condition $(SF)_*$ for $\varphi(W)$. 
Here, we need to inspect $\overline{\varphi (W)} \cap (\{ p \} \times \overline{F})$ and check that it is a proper intersection on $X \times \overline{\square}^{n+1}$. We handle the easier cases first.

\medskip

\textbf{Case 1:}  If $F= \square^{n+1}$, it is enough to check that $\overline{\varphi (W) } \not \subset  \{ p \} \times \overline{\square}^{n+1}$. Toward contradiction, suppose $\overline{\varphi{(W)}} \subset \{ p \} \times \overline{\square}^{n+1}$. This means that the image of the proper morphism $\overline{\varphi (W) } \to X$ is concentrated at the unique closed point $ p \in X$. However, this is impossible: for any closed point $\mathfrak{p} \in W$ which is not on the faces of $B_{n+1}$ or on the divisor $ \{ x = 0 \}$ of $B_{n+1}$, we have $\varphi (\mathfrak{p}) \subset \varphi (W)$, and $\varphi (\mathfrak{p}) \to X$ is dominant because it is in $z_{{\rm v}} ^{n+1} (X, n+1)$ by Lemma \ref{lem:inverse 0-cycle} so that by Lemma \ref{lem:SF DF} $\varphi (\mathfrak{p})$ satisfies the condition $(DO)$. Hence $\varphi (W) \to X$ is dominant as well, while we saw that the image of $\overline{\varphi (W)} \to X$ is $p \in X$. This is a contradiction. Hence we deduce that the intersection $\overline{\varphi (W) } \cap   ( \{ p \} \times \overline{\square}^{n+1} )$ is proper on $X \times \overline{\square}^{n+1}$.

\medskip

\textbf{Case 2:} If $F \subset \square^{n+1}$ is of codimension $\geq 2$, in \textbf{Cases 2, 3} of \textbf{Step 1} we saw that $\overline{\varphi (W)} \cap ( X\times \overline{F}) = \emptyset$ so that $\overline{\varphi (W)} \cap (\{ p \} \times \overline{F}) = \emptyset$ as well, hence the intersection is proper.

\medskip

\textbf{Case 3:} Suppose $F \subset \square^{n+1}$ is of codimension $1$ given by $F_i ^{\epsilon}=\{ y_i = \epsilon \}$ for some $1 \leq i \leq n+1$ and $\epsilon \in \{ 0, \infty \}$.

\medskip

\textbf{Subcase 3-1:} Suppose $2 \leq i \leq n+1$. Then by Lemma \ref{lem:tilde face}-(3), we have
$
\overline{ \varphi (W)} \cap (X \times \overline{F}_i ^{\epsilon}) = \overline{ \varphi (W \cap F_{i-1} ^{\epsilon})},
$
where we identify $X \times \overline{F}_i ^{\epsilon} \simeq X \times \overline{\square}^n$. Hence
$$
\overline{ \varphi (W)} \cap (\{ p \} \times \overline{F}_i ^{\epsilon}) = \overline{ \varphi (W \cap F_{i-1} ^{\epsilon})} \cap ( \{ p \} \times \overline{\square} ^n ).
$$
Here, we have $\varphi (W \cap F_{i-1} ^{\epsilon}) \in z_{{\rm v}} ^n (X, n)$ by Lemma \ref{lem:inverse 0-cycle}, so that by the condition $(SF)_*$ for $\varphi (W \cap F_{i-1} ^{\epsilon})$, the intersection $\overline{ \varphi (W \cap F_{i-1} ^{\epsilon})} \cap ( \{ p \} \times \overline{\square} ^n )$ is proper on $X \times \overline{\square}^n$, i.e. its dimension is $0$. Hence the intersection $\overline{ \varphi (W)} \cap (\{ p \} \times \overline{F}_i ^{\epsilon}) $ is also of dimension $0$, thus of codimension $(n+2)$ in $X \times \overline{\square}^{n+1}$. Since the codimensions of $\overline{\varphi (W)}$ and $\{p \} \times \overline{F}_{i} ^{\epsilon}$ in $X \times \overline{\square}^{n+1}$ are $n$ and $2$, respectively, this shows that the intersection $\overline{ \varphi (W)} \cap (\{ p \} \times \overline{F}_i ^{\epsilon})$ is proper on $X \times \overline{\square}^{n+1}$.

\medskip

\textbf{Subcase 3-2:} Suppose $F= F_{1 } ^0$ given by $\{ y_1 = 0 \}$. We analyze $\overline{\varphi (W)} \cap (\{p \} \times \overline{F}_1 ^0)$. Consider the diagram
with $\Lambda= \Lambda_{n+1}$
\begin{equation}\label{eqn:sub3-2}
\xymatrix{
\ov{W} \times_{\widehat{B}_{n+1}} {\Lambda} \ar[d] \ar@{^{(}->}[rr]^{\iota^{{\Lambda}}} & &{ \Lambda } \ar[d] ^{pr^{{\Lambda} }_{1}} \ar[rr]^{ pr^{{\Lambda}_{n+1}}_{2}} & & X \times \overline{ \square}^{n+1} \\
 \ov{W} \ar@{^{(}->}[rr]^{\iota} & &  \widehat{B}_{n+1}. & }
\end{equation}

For $\overline{W} \times_{\widehat{B}_{n+1}} \Lambda$, let $\overline{ \overline{W} \times_{\widehat{B}_{n+1}} \Lambda}$ be its closure in $\widehat{B}_{n+1} \times X \times \overline{\square}^{n+1}$. It maps surjectively onto $\overline{\varphi (W)}$.

Let's look at the points of $\overline{ \overline{W} \times_{\widehat{B}_{n+1}} {\Lambda}}$ that lie over the divisor $X \times \overline{F}_1 ^0$ of $X \times \overline{\square}^n$. Previously we saw in the proof of Lemma \ref{lem:additive star} that if there is a point $\mathfrak{p} \in \overline{ \overline{W} \times_{\widehat{B}_{n+1}} {\Lambda}}$ whose $x$-coordinate is $\infty$, then its $y_1$ coordinate becomes $1$. So such $\mathfrak{p}$ cannot lie over $X \times \overline{F}_1 ^0$. So, we consider a point $\mathfrak{p} \in \overline{ \overline{W} \times_{\widehat{B}_{n+1}} {\Lambda}}$, whose $x$-coordinate is not $\infty$. 

When $x \not =  \infty$ and $y_1 \not = \infty$ at a point $\mathfrak{p} \in  \overline{\overline{W} \times_{\widehat{B}_{n+1}} {\Lambda}}$, then near this point the equation $y_1 = 1 - \frac{t}{x}$ of $\Lambda$ can be rewritten as $x y_1 = x - t$. So, when $y_1 = 0$, we obtain $0 = x-t$, i.e. $x=t$. However, $x$ is a variable in $\mathbb{P}^1$, and $t $ is a formal variable in $k[[t]]$, so there is no possibility to have this equality, except possibly when $x=0$ and $t=0$.

When $\mathfrak{p} \in \overline{ \overline{W} \times_{\widehat{B}_{n+1}} {\Lambda}}$ is such a point, it gives a closed point $\mathfrak{q} \in \overline{W}$ such that $\mathfrak{q} \in \overline{W} \setminus W$ whose $x$-coordinate is $0$. Since $\dim \ W = 1$, there exist only finitely many such closed points $\mathfrak{q}$. Furthermore, near each such point $\mathfrak{q}$, its coordinates $(x=0, z_1, \cdots, z_{n})$ in $\widehat{B}_{n+1}$ gives the coordinates $(y_1 = 0, y_2 = z_1, \cdots, y_{n+1} = z_n)$ for $\mathfrak{p}$ in $\overline{\overline{W} \times_{\widehat{B}_{n+1}} {\Lambda}}$ by the remaining defining equations $y_{i+1} = z_i   - xt= z_i$ for $2 \leq i \leq n+1$ of $\Lambda$. Hence there exist only finitely many closed points $\mathfrak{p}$ that lie over $\{ y_1 = 0\}$. In particular, only finitely many such points lie over $\{ p \} \times \{ y_1 = 0 \}$. Hence $\overline{\varphi (W)} \cap ( \{ p \} \times \overline{F}_1 ^0)$ is $0$-dimensional, i.e. of codimension $(n+2)$ in $X \times \overline{\square}^{n+1}$. Since the codimensions of $\overline{\varphi(W)}$ and $\{ p \} \times \overline{F}_1 ^{0}$ in $X \times \overline{\square}^{n+1}$ are $n$ and $2$, respectively, the intersection is proper.

\medskip

\textbf{Subcase 3-3:} Suppose $F= F_{1} ^{\infty}$ given by $\{ y_1 = \infty \}$. We analyze $\overline{\varphi (W)} \cap ( \{ p \} \times \overline{F}_1 ^{\infty})$. Consider the diagram \eqref{eqn:sub3-2} again.

Let $\overline{ \overline{W} \times_{\widehat{B}_{n+1}} \Lambda}$ be the closure in $\widehat{B}_{n+1} \times X \times \overline{\square}^{n+1}$. To inspect $\overline{\varphi (W)} \cap ( \{ p \} \times \overline{F}_1 ^{\infty } )$, we are primarily interested in the points $\mathfrak{p}$ of $\overline{ \overline{W} \times_{\widehat{B}_{n+1}} {\Lambda}}$ that lie over $\{ p \} \times \overline{F}_1 ^{\infty} \subset X \times \overline{\square}^{n+1}$. 

\medskip

To study such points, let's use the projective coordinate $(Y_{10}; Y_{11})$ for $y_1= Y_{10}/ Y_{11}$. Here, the first equation $y_1 = 1- \frac{t}{x}$ of ${\Lambda}$ has the homogenization $x Y_{10} = (x-t) Y_{11}$, and letting $y_1= \infty$ is equivalent to letting $Y_{11} = 0$ with $Y_{10}\not = 0$. Hence we have $x=0$ for such points $\mathfrak{p}$. Each such point $\mathfrak{p}$ gives a closed point $\mathfrak{q} \in \overline{W}$ such that $\mathfrak{q} \in \overline{W} \setminus W$, and since $\dim \ W = 1$, there exist only finitely such closed points $\mathfrak{q}$. Furthermore, near such such point $\mathfrak{q}$, its coordinates $(x=0, z_1, \cdots, z_{n})$ in $\widehat{B}_{n+1}$ gives the coordinates $(y_1 = \infty, y_2 = z_1, \cdots, y_{n+1} = z_n)$ by the remaining defining equations $y_{i+1} = z_i   - xt= z_i$ for $2 \leq i \leq n+1$ of $\Lambda$. Hence there exist only finitely many closed points $\mathfrak{p}$ that lie over $\{ y_1 = \infty \}$. Hence $\overline{\varphi (W)} \cap ( \{ p \} \times \overline{F}_1 ^{\infty})$ is also a $0$-dimensional cycle on $X \times \overline{\square}^{n+1}$, i.e. of codimension $(n+2)$. Since the codimensions of $\overline{\varphi(W)}$ and $\{ p \} \times \overline{F}_1 ^{\infty}$ in $X \times \overline{\square}^{n+1}$ are $n$ and $2$, respectively, the intersection is proper.

\medskip

Hence we checked all the cases of the property $(SF)_*$ for $\varphi (W)$. Thus $\varphi (W) \in z_{{\rm d}} ^n (X, n+1)$. 

\bigskip

\textbf{Step 3:} We check that $\varphi(W)$ is a pre-vanishing cycle, i.e. the special fiber $\varphi (W)_p$ over the closed point $p \in X$ is empty. This is more delicate, and we need to use the help of the modulus condition. First take the normalization $\nu: \overline{W}^N \to \overline{W}$ and its pull-back to form the diagram with $\Lambda= \Lambda_{n+1}$
$$
\xymatrix{
\overline{W}^N \times_{\widehat{B}_{n+1}}{ \Lambda }\ar[d] \ar@{>>}[rr] ^{\nu^{{\Lambda}}} & & \overline{W} \times_{\widehat{B}_{n+1}} { \Lambda} \ar[d]  \ar@{^{(}->}[rr]^{\ \ \ \ \ \iota^{{\Lambda}}} & & {\Lambda} \ar[d] ^{pr_1 ^{\Lambda}} \ar[rr] ^{pr_2 ^{\Lambda} \ \ \ } & & X \times {\square}^{n+1} \ar[d] \\
\overline{W}^N \ar@{>>}[rr] ^{\nu} && \overline{W} \ar@{^{(}->}[rr] ^{\iota} & & \widehat{B}_{n+1} & & X. }
$$

To prove that the special fiber $\varphi (W)_p$ is empty, it is enough to show that the special fiber of $\overline{W}^N \times _{\widehat{B}_{n+1}} \Lambda$ over $p \in X$ is empty.

Note from the first equation $y_1 = 1 - \frac{t}{x}$ of $\Lambda$ written in the form $-x (y_1 -1) = t$, we have either 
\begin{enumerate}
\item [(i)] $t \ | \ x$, or 
\item [(ii)] $t \ | \ y_1 -1$.
\end{enumerate}

In the case of (ii), over $p$ we have $t=0$ so that such points in $\overline{W}^N \times _{\widehat{B}_{n+1}} \Lambda$ satisfy $y_1 = 1$, and it defines the empty set in $\widehat{B}_{n+1} \times X \times \square^{n+1}$. Hence the claim holds automatically.

In the case of (i), over $p$ we have $t=0$ so that such points $\mathfrak{p}$ in $\overline{W}^N \times _{\widehat{B}_{n+1}} \Lambda$ satisfy $x=0$. The projections to $\overline{W}^N$ of such points give the closed points $\mathfrak{q} \in \overline{W}^N$ in $\overline{W}^N  \setminus W^N$ such that their projections to the $x$-coordinates $\mathbb{P}^1$ of $\widehat{B}_{n+1}= \mathbb{P}^1 \times \overline{\square}^n$ are $0\in \mathbb{P}^1$. In particular, there are only finitely many such closed points $\mathfrak{q}$ because $\dim \ W = 1$.

Let's analyze what the strong sup modulus condition $M_{ssup}$ of $W$ says about it. Let $\mathfrak{q}$ be any such closed point. For the regular projective curve $\overline{W}^N$ over $k$, let $\pi$ be a uniformizer of the DVR $\mathcal{O}:= \mathcal{O}_{\overline{W}^N, \mathfrak{q}}$ at $\mathfrak{q}$. Let ${\rm ord}_{\mathfrak{q}}$ be the associated discrete valuation. For the rational functions $\nu^* (x), \nu^* (z_1), \cdots, \nu^* (z_{n})$ on $\overline{W}^N$ given by the coordinates $(x, z_1, \cdots, z_n) \in \widehat{B}_{n+1}$, we just denote them by $x, z_1, \cdots, z_n$ for simplicity. Let $r:= \ord_{\mathfrak{q}} ( x)$ so that we can write $x= c_1 \cdot \pi ^{r},$ for some $c_1\in \mathcal{O}^{\times}$ and $r \geq 1$. 

The strong sup modulus $(m+1)$ condition satisfied by $W$ says that there exists some $1 \leq i_0 \leq n$ such that we have
$$
 r (m+1) \leq s_{i_0},
 $$
where $s_{i_0} := \ord_{\mathfrak{q}} (z_{i_0} -1)$. In particular, there is an integer $a \geq 0$ such that $r (m+1) + a = s_{i_0}$, so that near $\mathfrak{q} \in \overline{W}^N$ we can express $z_{i_0}= 1 + c'  \pi^{s_{i_0}}$ for some $c'  \in \mathcal{O}^{\times}$, and
$$
z_{i_0} = 1 + c' \cdot \pi ^{s_{i_0}} = 1 + c' \cdot \pi ^{ r (m+1) + a} = 1 + c''  x^{m+1} \cdot \pi ^a
$$
for some $c'' \in \mathcal{O}^{\times}$. Since we saw that $t \ | \ x$, we have 
\begin{equation}\label{eqn:step3}
z_{i_0} \equiv 1 \mod t^{m+1}
\end{equation}
near $\mathfrak{q}$. In particular, $z_{i_0} \equiv 1 \mod t$.

Under the assumption $x=0$ at such points $\mathfrak{p}$ with their projections $\mathfrak{q} \in \overline{W}^N$, the remaining defining equations $y_i = z_{i-1} - x t$ of $\Lambda$ imply that $y_i = z_{i-1}$, so that the push-forward of the point $\mathfrak{p}$ to $X \times \square^{n}$ gives the values $y_2 = z_1, \cdots, y_{n+1} =z_{n}$, where near every such $\mathfrak{p}$ we have $y_{i_0 + 1} = z_{i_0} \equiv 1 \mod t$. This shows that $\mathfrak{p}$ is also uniquely determined by $\mathfrak{q}$, and in particular there exist only finitely many closed points $\mathfrak{p}$ that satisfy the condition (i). Furthermore, every such point has a coordinate $y_{i_0+1} \equiv 1 \mod t$. This shows that $\varphi (W)_p= \emptyset$, i.e. $\varphi (W)$ is a pre-vanishing cycle. Hence $\varphi (W) \in z^n_{{\rm v}} (X , n+1)$, as desired.
\end{proof}

\begin{defn}
Using Proposition \ref{prop:inverse Bloch general adm}, we define the homomorphism of groups
\begin{equation}\label{eqn:inverse Bloch 1'}
\varphi=\widetilde{\varphi}_{n+1}: \TZ^n (k, n+1;m) ' \to z^n_{{\rm v}} (X, n+1),
\end{equation}
by $\mathbb{Z}$-linearly extending $\widetilde{\varphi}_{n+1}$ in \eqref{eqn:varphi n+1 0}.
\qed
\end{defn}

Unfortunately, the maps $\varphi_n$ in \eqref{eqn:inverse Bloch 1} and $\widetilde{\varphi}_{n+1}$ in \eqref{eqn:inverse Bloch 1'} fail to be compatible with respect to the boundary maps on $\TZ^n (k, n+1, m)' $ and $z_{{\rm v}} ^n (X, n+1)$. The problem occurs for the boundary operators $\partial_1 ^{\epsilon}$ for $\epsilon \in \{ 0, \infty \}$. However, we can overcome this last hurdle in defining the inverse Bloch map by going modulo $I^{m+1}$. We define:

\begin{defn}
The composite with the quotient map of Definition \ref{defn:mod t^{m+1} v}
\begin{equation}\label{eqn:inverse Bloch n+1 mod m+1}
\widetilde{\varphi}_{n+1} : \TZ^n (k, n+1; m)'  \to z^n _{{\rm v}} (X, n+1) \twoheadrightarrow z^n_{{\rm v}} (X/ (m+1), n+1)
\end{equation}
is denoted by the same name $\widetilde{\varphi}_{n+1}$ whenever no confusion arises.
\qed
\end{defn}

In this mod $I^{m+1}$ cycle group, we will exploit the full capacity of the modulus condition for $W$ to prove:

\begin{cor}\label{cor:tilde W face y_1}
Let $W \in \TZ^n (k, n+1; m)' $ be an integral cycle. Then, the cycle $\widetilde{\varphi}_{n+1} (W)  \in z_{{\rm v}} ^n (X/ (m+1), n+1)$ satisfies
$$
\partial_1 ^0 ( \widetilde{\varphi}_{n+1} (W)) = 0, \ \ \  \partial_1 ^{\infty} (\widetilde{ \varphi }_{n+1} (W))= 0
$$
in $z_{{\rm v}} ^n (X/ (m+1), n)$.
\end{cor}

\begin{proof}
Write $\varphi(W)= \widetilde{\varphi}_{n+1} (W)$ for simplicity.

\medskip

Part of the needed reasoning is similar to some arguments already used in the proof of Proposition \ref{prop:inverse Bloch general adm}, so we recycle them here.

First note that in \textbf{Subcases 3-2} and \textbf{3-3} of \textbf{Step 2} of the proof of Proposition \ref{prop:inverse Bloch general adm}, we know $\overline{\varphi(W)} \cap ( X \times \overline{F}_1 ^{\epsilon})$ for $\epsilon \in \{ 0, \infty \}$ come from some finitely many closed points $\mathfrak{p}$ of $ \overline{ \overline{W} \times_{\widehat{B}_{n+1}} {\Lambda} } $ whose $x$-coordinate is $0$. 

In particular, the points $\mathfrak{p}$ of $\varphi(W) \cap (X \times F_1 ^{\epsilon} )$ come from some finitely many closed points $\mathfrak{q}$ of $\overline{W}^N $ that belong to $\overline{W}^N \setminus W^N$, whose $x$-coordinate is $0$. In \textbf{Step 3} we saw that such points also lie over the special fiber $p\in X$, while $y_1 \not = 1$.

\medskip

So, from the argument in the case (i) of the proof of \textbf{Step 3} of Proposition \ref{prop:inverse Bloch general adm}, in terms of the notations there, for each such closed point $\mathfrak{q}$ we saw in \eqref{eqn:step3} that there exists some integer $1 \leq i_0 \leq n$ such that
$$
z_{i_0} \equiv 1 \mod t^{m+1},
$$
and for the corresponding closed point $\mathfrak{p} \in  \overline{ \overline{W} \times _{\widehat{B}_{n+1}} {\Lambda}}$ their $x$-coordinate is $0$, so we have $y_i = z_{i-1}$ for $2 \leq i \leq n+1$.

The face $\partial_1 ^{\infty} (\varphi (W))$ consists of the finitely many points $(y_2 (\mathfrak{p}), \cdots, y_{n+1} (\mathfrak{p}))$ for such points $\mathfrak{p}$, while for each such point, we have $y_{i_0 + 1} (\mathfrak{p}) = z_{i_0} \equiv 1  \mod t^{m+1}$ by the above. This means $(y_2 (\mathfrak{p}), \cdots, y_{n+1} (\mathfrak{p})) \equiv 0$ in $z_{{\rm v}} ^n (X/ (m+1), n)$ in the cycle group modulo $t^{m+1}$. Hence $\partial_1 ^{\epsilon} (\varphi (W)) = 0$ in $z_{{\rm v}} ^n (X/ (m+1), n)$ for $\epsilon \in \{ 0, \infty\}$ as desired.
\end{proof}

\begin{cor}\label{cor:tilde W}
We have the commutative diagram:
$$
\xymatrix{
\TZ^n (k, n+1;m)'  \ar[d] ^{\partial} \ar[rrr] ^{ \widetilde{\varphi}_{n+1}} & & & z_{{\rm v}} ^n (X/ (m+1), n+1) \ar[d] ^{ (-1)\cdot \partial}\\
\TZ^n (k, n;m) \ar[rrr] ^{\varphi_n} & &  & z_{{\rm v}} ^n (X/ (m+1), n),}
$$
where the bottom horizontal map is the composite of \eqref{eqn:inverse Bloch 1} with the quotient map $z_{{\rm v}} ^n (X, n) \to z_{{\rm v}} ^n (X/ (m+1), n)$, the top horizontal map is the one given in \eqref{eqn:inverse Bloch n+1 mod m+1}, the left vertical boundary map $\partial$ is from Definition \ref{defn:ACH2}, and the right vertical boundary map exists by \cite[Lemmas 3.7.10, 3.7.11]{Park presentation}.
\end{cor}

\begin{proof}
Let $W \in \TZ^n (k, n+1; m)' $ be an integral cycle. By Corollary \ref{cor:tilde W face y_1}, we have $\partial_1 ^{\epsilon}( \widetilde{\varphi}_{n+1} (W) )= 0$ for $\epsilon \in \{ 0, \infty \}$ in $z_{{\rm v}} ^n (X/ (m+1), n)$. Hence
$$
\partial \left( \widetilde{\varphi}_{n+1} (W) \right)= \sum_{i=1} ^{n+1}  (-1)^i ( \partial_i ^{\infty} - \partial_i ^0) ( \widetilde{\varphi}_{n+1} (W)) = \sum_{i=2} ^{n+1} (-1)^i ( \partial_i ^{\infty} - \partial_i ^0) ( \widetilde{\varphi}_{n+1} (W)) 
$$
$$ 
= ^{\dagger} \sum_{i=2} ^{n+1} (-1)^i \varphi_n \left( ( \partial_{i-1} ^{\infty} - \partial_{i-1} ^0 ) ( W) \right) =- \sum_{i=1} ^n (-1)^i \varphi_n  \left( ( \partial_{i} ^{\infty} - \partial_{i} ^0 ) ( W) \right) = - \varphi_n (\partial W).
$$
where $\dagger$ holds by Lemma \ref{lem:tilde face}-(3). This proves the commutativity.
\end{proof}

From Corollary \ref{cor:tilde W}, we now deduce:

\begin{prop}\label{prop:inverse Bloch}
The composite
$$
\TZ^n (k, n+1; m)' \overset{\partial}{\longrightarrow} \TZ^n (k, n;m) \overset{\varphi_n}{\longrightarrow} \CH_{{\rm v}} ^n (X/ (m+1), n)
$$
is zero. In particular, we have the induced homomorphism
\begin{equation}\label{eqn:inverse Bloch cy}
{\varphi}_n: \TCH^n (k, n;m) '  \to \CH_{{\rm v}} ^n (X/ (m+1), n).
\end{equation}
\end{prop}

This is the geometric (=cycle-theoretic) inverse Bloch map, that we have long been after. Under Theorems \ref{thm:Rulling2} and \ref{thm:graph final}, Proposition \ref{prop:inverse Bloch} can be rephrased in terms of the original algebraically defined groups as:

\begin{defn}\label{defn:inverse Bloch}
Let $k$ be a field. Let $ m, n \geq 1$ be integers. The inverse Bloch map is defined to be the composite of the following maps 
$$
\xymatrix{
\mathbb{W}_m \Omega_k ^{n-1} \ar[d] ^{\simeq} _{gr_k} & & \widehat{K}_n ^M (k_{m+1}, (t))  \\
\TH^n (k, n;m)' \ar[rr] ^{\varphi} & & \CH_{{\rm v}} ^n (X/ (m+1), n) \ar[u] ^{ gr_{{\rm v}} ^{-1}} _{\simeq} .
}
$$
\end{defn}
In what follows, if needed we often identify the groups along the vertical isomorphisms, to simplify our reasonings and the notations. For $n=1$ we have $\TH^1 (k, 1;m)' = \TH^1 (k, 1;m)$. We prove that:

\begin{thm}\label{thm:inverse Bloch n=1}
Let $k$ be a field, and let $X= \Spec (A)$ be an integral regular henselian local $k$-scheme of dimension $1$ with the residue field $k = A/I$. Let $m \geq 1$ be an integer. Then for $n=1$, the inverse Bloch map
$$
\varphi_1 : \TH^1 (k, 1;m) \to \CH_{{\rm v}} ^1 (X/ (m+1), 1)
$$
is an isomorphism.
\end{thm}

From Theorem \ref{thm:Rulling2} and \ref{thm:graph final}, we already know that both of the groups are isomorphic to $\mathbb{W}_m (k)$. So, our task is \emph{not} to prove that these groups are isomorphic to each other, but that the inverse Bloch map $\varphi_1$ is indeed an isomorphism.

\begin{proof}
We may assume that $A= k[[t]]$ as the group on the right is independent of the choice of such $A$. 
Recall that $\TH^1 (k, 1; m)$ is generated by closed points $\mathfrak{p} \in \mathbb{A}^1$, where $\mathfrak{p} \not = 0$. Each such point defines a class $[ \mathfrak{p}] \in \TH^1 (k, 1;m)$. 

Given such a class $[\mathfrak{p}]$, we will first show that $\varphi_1 ([\mathfrak{p}]) \in \CH_{{\rm v}} ^1 (X/ (m+1), 1)$ can determine the class $[\mathfrak{p}] \in \TH^1 (k, 1;m)$, which proves the injectivity of $\varphi_1$.

\medskip

Let $k':= \kappa (\mathfrak{p})$, and let $\pi: \Spec (k') \to \Spec (k)$ be the associated finite morphism. Then there is $k'$-rational closed point $\mathfrak{p'} \in \mathbb{A}^1$, with $\mathfrak{p}' \not = 0$ such that the class $[\mathfrak{p}'] \in \TH^1 (k', 1; m)$ satisfies $\pi_* ([\mathfrak{p}'] ) = [ \mathfrak{p}]$. Consider the commutative diagram
$$
\xymatrix{
\TH^1 (k', 1;m) \ar[d]^{\pi_*} \ar[r] ^{\varphi_{1, k'} \ \ \ }  & \CH_{{\rm v}} ^1 (X_{k'}/(m+1), 1) \ar[d] ^{\pi_*} \ar[r] ^{\ \ \ \ \ \ \  \phi_{1, k'}} &  \mathbb{W}_m (k') \ar[d] ^{\Tr_{k'/k}} \ar[r] ^{ gr_{k'} \ \ \ } & \TH^1 (k', 1;m) \ar[d] ^{\pi_*}  \\
\TH^1 (k, 1;m) \ar[r] ^{\varphi_{1, k} \ \ \ }  & \CH_{{\rm v}} ^1 (X/ (m+1), 1) \ar[r] ^{\ \ \ \ \ \ \ \phi_{1, k}}  & \mathbb{W}_m (k) \ar[r] ^{gr_k \ \ \ } & \TH^1 (k, 1;m),}
$$
where $\Tr_{k'/k}$ is the trace map for the Witt vectors for a finite free ring extension (see, e.g. \cite[Proposition A.9]{R}), and $\phi_1$ is the map defined in \cite[Definition 4.2.3]{Park presentation} which gives the inverse of the graph isomorphism of Theorem \ref{thm:graph final} for $n=1$.

For the minimal polynomial $\tilde{f} (t) = c-t \in k'[t]$ of $\mathfrak{p}' \in \mathbb{A}_{k'}^1$ over $k'$ where $c \not = 0$, the map $\varphi_{1, k'}$ sends $[\mathfrak{p}']$ to the graph cycle $\Gamma':= \{ y= 1 - \frac{t}{c} \} \in \CH_{{\rm v}} ^1 (X_{k'}/ (m+1), 1)$. This is sent by $\phi_{1, k'}$ to the Teichm\"uller lift $[c^{-1}] \in \mathbb{W}_m (k')$ of $c^{-1}$. (In terms of the identification $\mathbb{W}_m (k') = (1+ t k' [[t]])^{\times}$, this corresponds to $1- \frac{t}{c}$.)

However, by the definition of the graph map $gr_{k'}$ in K. R\"ulling \cite{R} (Theorem \ref{thm:Rulling2}), we have $gr_{k'} ( [c ^{-1}] ) = [ \mathfrak{p}']$. Hence the upper horizontal composite is injective and the class $[\mathfrak{p}'] \in \TH^1 (k', 1;m)$ is determined by $\varphi_{ 1, k'} ([ \mathfrak{p}'])$. Thus applying $\pi_*$, we deduce that $\varphi_{1, k} ([\mathfrak{p}])$ also determines the class $[\mathfrak{p}] \in \TH^1 (k, 1;m)$. This proves the injectivity of $\varphi_{1, k}$.

\medskip

To show that $\varphi_{1, k}$ is surjective, we use the identification $\CH_{{\rm v}} ^1 (X/ (m+1), 1) \simeq \mathbb{W}_m (k)$ of Theorem \ref{thm:graph final}, given by the isomorphism $\phi_{1, k}$, inverse to the graph map $gr_{{\rm v}} : \mathbb{W}_m (k) \to \CH_{{\rm v}} ^1 (X/ (m+1), 1)$, as well as the identification $\mathbb{W}_m (k) = (1 + t k[[t]])^{\times} / (1+ t^{m+1} k[[t]])^{\times}$. 

\medskip

Note that each $\alpha \in \mathbb{W}_m (k)= ( 1+ tk[[t]])^{\times} / (1 + t^{m+1} k[[t]])^{\times}$ has a non-constant polynomial representative $f(t)$ such that $f \equiv 1 \mod t$. It may not be irreducible, but taking the factorization $f(t) = f_1 (t) \cdots f_r (t)$ into irreducible polynomials in $k[t]$, where $f_i$ is scaled so that $f_i \equiv 1 \mod t$ for each $1 \leq i \leq r$, and replacing $(\alpha, f)$ by $(\alpha_i := \bar{f}_i, f_i)$, we deduce that $\mathbb{W}_m (k)$ is generated by elements $\alpha \in \mathbb{W}_m (k)$ that have their irreducible non-constant polynomial representatives $f(t) \in k[t]$ with $f(t) \equiv 1 \mod t$.

For such $\alpha$ and a chosen irreducible polynomial representative $f(t)$, this $f(t)$ defines a closed point $\mathfrak{p} \in \mathbb{A}^1$, distinct from $0$, so it defines a class $[\mathfrak{p}] \in \TH^1 (k, 1;m)$. For $k':= k[t]/(f(t))= \kappa (\mathfrak{p})$, we have the finite extension $k \hookrightarrow k'$ and let $\pi: \Spec (k') \to \Spec (k)$ be the corresponding map. Then we have a member $c \in k'\setminus \{ 0 \}$ which is a solution of $f(t)$, and under the identification $\mathbb{W}(k) = \CH_{{\rm v}} ^1 (X/ (m+1), 1)$,
$$
\varphi_{1, k} ([\mathfrak{p}]) = \pi_* \left\{ y = 1 - \frac{t}{c} \right\} = \left\{ y= \frac{f (t)} { f(0)}\right\} = \{ y = f (t) \} = \{ f(t) \mod t^{m+1} \} = \alpha.
$$
This proves that $\varphi_{1, k}$ is surjective. Hence $\varphi_{1, k}$ is an isomorphism.
\end{proof}

\begin{remk}\label{remk:strong surj n=1}
In terms of an auxiliary notion of ``strong surjectivity" to be defined later in Definition \ref{defn:strong surj}, the above argument actually shows that $\varphi_1$ is strongly surjective. This point will be used later, e.g. in Proposition \ref{prop:inverse Bloch surj}.
\qed
\end{remk}

We will eventually prove in Corollary \ref{cor:inverse Bloch surj 3} that for all $ n \geq 1$, the map $\varphi$ in \eqref{eqn:inverse Bloch cy} is an isomorphism. Its injectivity for $n \geq 2$ is going to be discussed in  \S \ref{sec:inj inverse Bloch} and the surjectivity for $n \geq 2$ in \S \ref{sec:surj inverse Bloch}. Both of them are nontrivial.

For the injectivity, we resort to a special homomorphism that we call, the \emph{de}-concatenation, to be constructed in \S \ref{sec:decon 1} below.

\section{Deconcatenation and injectivity of the inverse Bloch map}\label{sec:decon 1}

We continue to suppose $X= \Spec (k[[t]])$ as before, with the unique closed point $p \in X$ and the generic point $\eta \in X$. Let $\mathbb{F} = \kappa (\eta) = k((t))$, the function field. The objective of \S \ref{sec:decon 1} is to construct geometrically a nontrivial homomorphism for each $n \geq 2$
$$
{\rm Dec}: \CH^n _{\rm v} (X, n) \to \CH^1 _{\rm v} (X, 1) \otimes \CH^{n-1} (k, n-1),
$$
which we call the deconcatenation, and to deduce the injectivity of $\varphi$ in \S \ref{sec:inj inverse Bloch}. At least to the author, its existence was rather surprising. The proof of its existence takes up most of \S \ref{sec:decon 1} before the proof of the injectivity of $\varphi$ in \S \ref{sec:inj inverse Bloch}.

\subsection{Deconcatenation}

Let $n \geq 2$ be an integer. Let $Z \in z^n_{{\rm v}} (X, n)$ be an integral cycle so that $Z \cap ( \{ p \} \times \square^n) = \emptyset$. Recall (Definition \ref{defn:type i} and Remark \ref{remk:type i}) that $Z$ has a vanishing coordinate $y_i$ for some $1 \leq i \leq n$. So, there exists a cyclic permutation $\sigma \in S_n$ on the coordinates of the form $\sigma = (1, 2, \cdots, n)^i$ for some $1 \leq i \leq n$, such that $\sigma \cdot Z$ has $y_1$ as a vanishing coordinate. 

\begin{defn}\label{defn:dec 00}
For a choice of such $\sigma$, define its \emph{deconcatenation} to be
$$
{\rm Dec} (Z) := {\rm sgn} (\sigma) \left( (\sigma \cdot Z)^{ (1)}  \otimes {\rm ev}_{p} \left( (\sigma \cdot Z) ^{ (1')}\right) \right) \in z_{{\rm v}} ^1 (X, 1) \otimes z^{n-1} (k, n-1),
$$
where the superscripts $(1)$ and $(1')$ mean the compactified projections in Lemma \ref{lem:adm compact proj}, with $1':= \{ 1, 2, \cdots, n \} \setminus \{ 1 \} = \{ 2, \cdots, n \}$, and ${\rm ev}_p: z_{{\rm d}} ^{n-1} (X, n-1) \to z^{n-1} (k, n-1)$ is the specialization map of Lemma \ref{lem:specialization}.
\qed
\end{defn}

\emph{A priori}, our definition makes a choice of an integer $1 \leq i \leq n$ (equivalently a cyclic permutation $\sigma$). 
This choice is immaterial:

\begin{lem}\label{lem:decon well}
Let $Z \in z^n _{{\rm v}} (X,n )$ be an integral cycle.

Suppose that there are at least two vanishing coordinates $y_{i_1}$ and $y_{i_2}$ of $Z$ for two distinct integers $1 \leq i_1 < i_2 \leq n$. Let $\sigma_1, \sigma_2$ be the corresponding cyclic permutations $(1, \cdots, n)^{i_1}$ and $(1, \cdots, n)^{i_2}$. Then for both $j=1,2$ we have 
$$
{\rm sgn} (\sigma_j) \left( (\sigma_j \cdot Z)^{ (1)}  \otimes {\rm ev}_p  \left( (\sigma_j \cdot Z) ^{ (1')} \right) \right)  = 0  \ \ \mbox{ in } z_{{\rm v}} ^1 (X, 1) \otimes z^{n-1} (k, n-1).
$$
In particular, ${\rm Dec} (Z)$ is well defined.
\end{lem}

\begin{proof}

Since $Z$ has at least two vanishing coordinates, for each $j = 1,2$, the compactified projection $(\sigma_j \cdot Z)^{ (1')} \in z_{{\rm d}} ^{n-1} (X, n-1)$ always has a vanishing coordinate, thus it belongs to the subgroup $z_{{\rm v}} ^{n-1} (X, n-1)$. Consequently its image ${\rm ev}_p  \left( (\sigma_j \cdot Z) ^{ (1')} \right) $ in $z^{n-1} (k, n-1)$ is $0$. This implies the desired vanishing.
\end{proof}

By Lemma \ref{lem:decon well}, a potentially nontrivial deconcatenation could be obtained from an integral cycle $Z$ which has only one vanishing coordinate. 

\medskip

We have the induced group homomorphism via the quotient map
$$
 z_{{\rm v}} ^1 (X, 1) \otimes_{\mathbb{Z}} z^{n-1} (k, n-1) \to \CH_{{\rm v}} ^1 (X, 1) \otimes_{\mathbb{Z}} \CH^{n-1} (k, n-1).
 $$
Combined with the map ${\rm Dec}$ in Definition \ref{defn:dec 00}, we deduce the homomorphism
\begin{equation}\label{eqn:decon 1}
{\rm Dec} : z^n _{{\rm v}} (X, n) \to \CH_{{\rm v}} ^1 (X, 1) \otimes_{\mathbb{Z}} \CH^{n-1} (k, n-1).
\end{equation}

\medskip

We claim that \eqref{eqn:decon 1} kills the boundaries:

\begin{prop}\label{prop:decon bdry 0}
${\rm Dec} (\partial z_{{\rm v}} ^n (X, n+1)) = 0$ in $ \CH_{{\rm v}} ^1 (X, 1) \otimes_{\mathbb{Z}} \CH^{n-1} (k, n-1).$

In particular, we have the induced homomorphism
\begin{equation}\label{eqn:decon 2}
{\rm Dec} : \CH^n _{{\rm v}} (X, n) \to \CH_{{\rm v}} ^1 (X, 1) \otimes_{\mathbb{Z}} \CH^{n-1} (k, n-1).
\end{equation}
\end{prop}

The proof runs from \S \ref{sec:4.2} to \S \ref{sec:4.4}.

\subsection{Constant (vanishing) coordinates}\label{sec:4.2}
We introduce some auxiliary notions:

\begin{defn}
Let $n \geq 2$ be an integer. Let $W \in z^n _{{\rm v}} (X, n+1)$ be an integral cycle. Let $1 \leq i \leq n+1$.
\begin{enumerate}
\item We say that the coordinate $y_i$ is a \emph{constant coordinate of $W$} (on the generic fiber) if the composite
$$
\widehat{pr}_i: \overline{W} \hookrightarrow  \overline{\square}^{n+1}_X  \to  \overline{\square}_X^1,
 $$
where the second arrow is the projection to the $i$-th coordinate, has the property that the base change to the generic point $\eta \in X$,
$$ 
\widehat{pr}_{i, \eta}: \overline{W}_{\eta} \to \overline{\square}_{\mathbb{F}}^1
$$
is \emph{not} dominant, where $\mathbb{F} = \kappa (\eta )$. In particular, its image is a closed point in $\overline{\square}_{\mathbb{F}}$, thus it is indeed constant on the generic fiber.

\item We say that the coordinate $y_i$ is a \emph{constant vanishing coordinate of $W$}, if $y_i$ is a vanishing coordinate of $W$, and it is also a constant coordinate of $W$ (on the generic fiber) in the sense of (1).
\qed
\end{enumerate}
\end{defn}

\begin{lem}\label{lem:const}
Let $n \geq 2$ be an integer. Let $W \in z^n _{{\rm v}} (X,n+1)$ be an integral cycle. 

\begin{enumerate}
\item Suppose $W$ has a constant coordinate $y_i$ for some $1 \leq i \leq n+1$. 

Then for the extended codimension $1$ face $\overline{F}_i ^{\epsilon}= \{ y_i = \epsilon \}$ with $\epsilon \in \{ 0, \infty \}$, we have $\overline{W} \cap (X \times \overline{F}_i ^{\epsilon}) = \emptyset$.  In particular $\partial_i ^{\epsilon} (W) = 0$ for $\epsilon \in \{ 0, \infty \}$ as well.
\item Suppose $y_i$ is \emph{not} a constant coordinate for $W$ for some $1 \leq i \leq n+1$. 

Then for $\epsilon \in \{ 0, \infty \}$, the extended codimension $1$ face $\overline{F}_i ^{\epsilon}= \{ y_i = \epsilon \}$ satisfies $\overline{W} \cap (X \times \overline{F}_i ^{\epsilon}) \not = \emptyset$.
\end{enumerate}
\end{lem}

\begin{proof}
(1) After relabeling, we may assume that $y_1$ is a constant coordinate of $W$.

The base change $\overline{W}_{\eta}$ (which is in $\overline{ \square}_{\mathbb{F}} ^{n+1}$) via $\eta \to X$ is a $1$-dimensional scheme over $\mathbb{F}$, and its image under $\widehat{pr}_{1, \eta}$ in $\overline{\square}_{\mathbb{F}}^1$ is a closed point. Since $W$ satisfies the $(GP)_*$, one checks that its flat pull-back $\overline{W}_{\eta}$ intersects all extended faces of $\overline{\square}_{\mathbb{F}} ^{n+1}$ properly. Hence the image of $\overline{W}_{\eta}$ under $\widehat{pr}_{1, \eta}$ in $\overline{\square}_{\mathbb{F}}^1$ cannot be $\{ 0\}$ or $\{ \infty \}$.

If $\overline{W} \cap (X \times \overline{F}_i ^{\epsilon}) \not = \emptyset$ for some $\epsilon \in \{ 0, \infty \}$, then we have a nonempty irreducible component $\overline{Z}_0 \subset \overline{W} \cap (X \times \overline{F}_i ^{\epsilon})$. The image of the composite of its base change to $\eta$
 $$
 \overline{Z}_{0, \eta} \hookrightarrow \overline{W}_{\eta} \to \overline{\square}_{\mathbb{F}}^1
 $$
 is $\epsilon \in \{ 0, \infty \}$ by definition. This is a contradiction because we saw in the above that all of $\overline{W}_{\eta}$ is mapped to a closed point in $\overline{\square}_{\mathbb{F}}^1$ distinct from $\{0 \}$ and $\{ \infty \}$. Hence $\overline{W} \cap (X \times \overline{F}_i ^{\epsilon}) = \emptyset$.

 \medskip
 
 (2) Since $y_i$ is not a constant coordinate, the projective morphism $\overline{W} _{\eta} \to \overline{\square}_{\mathbb{F}}^1$ is dominant, thus surjective. In particular, for $\epsilon \in \overline{\square}_{\mathbb{F}}$ with $\epsilon \in \{ 0, \infty \}$, the fiber over $\epsilon$ is nonempty. Taking the closure in $\overline{\square}_X$, we have $\overline{W} \cap (X \times \overline{F}_i ^{\epsilon}) \not = \emptyset$.
\end{proof}

\begin{lem}\label{lem:const proj}
Let $n \geq 2$ be an integer. Let $W \in z^n _{{\rm v}} (X,n+1)$ be an integral cycle. Suppose $W$ has a constant vanishing coordinate $y_i$ for some $1 \leq i \leq n+1$, while not all coordinates are constant coordinates for $W$.

For the compactified projection to the $i$-th coordinate
$$
\widehat{pr}_i : \overline{\square}_X ^{n+1} \to \overline{\square}_X ^1,
$$
consider $\overline{Z}_i := \widehat{pr}_i (\overline{W})$, and let $Z_i$ be its restriction to the open subscheme $\square_X^1$. 

Then $Z_i \in z_{{\rm v}} ^1 (X, 1)$. 
\end{lem}

\begin{proof}

After applying a suitable permutation of coordinates, we may assume that $y_1$ is a constant vanishing coordinate of $W$, while $y_2$ is \emph{not} a constant coordinate of $W$.

\medskip

Since $y_1$ is a constant coordinate, the morphism $\overline{W}_{\eta} \to \overline{\square}^1_{\mathbb{F}}$ is not dominant, thus its image is a closed point in $\overline{\square}^1_{\mathbb{F}}$, distinct from $\{0, \infty, 1 \}$. The closure in $\overline{\square}_X ^1$ of this point is equal to $\overline{Z}_1:= \widehat{pr} _1 (\overline{W}) = \overline{W}^{(1)} \subset \overline{\square}_X ^1$, and it has $\dim \ \overline{Z}_1 = 1$. In terms of the notations of the compactified projections in Definition \ref{defn:comp proj 0}, let $Z_1 = W^{(1)} = \overline{W} ^{(1)} |_{\square_X ^1} \subset \square_X ^1$. Since $\dim \ \overline{Z}_1 = \dim \ Z_1$, the codimension of $Z_1$ in $\square_X ^1$ is $1$.

 Here, the morphism $\overline{W} \to \overline{Z}_1$ \emph{not} quasi-finite, so Lemma \ref{lem:adm comp proj 2} does \emph{not} apply here. Nonetheless, we prove that still $Z_1 \in z_{{\rm v}} ^1 (X, 1)$ in what follows.
 
 \medskip

 \medskip
 
 \textbf{Claim :} \emph{For a proper face $F \subset \square^1_k$ of codimension $1$, we have $\overline{Z}_1 \cap (X \times \overline{F}) = \emptyset$.}
 
 \medskip
 
 Here, $y_1$ is a constant coordinate of $W$, but we supposed that $y_2$ is \emph{not} a constant coordinate of $W$. In particular, by Lemma \ref{lem:const} for $\epsilon \in \{ 0, \infty \}$
 and the extended codimension $1$ face $\overline{F}_{2} ^{\epsilon}$ given by $\{ y_2 = \epsilon \}$, we have $\overline{W} \cap (X \times \overline{F}_2 ^{\epsilon}) \not = \emptyset$, and the composite
 \begin{equation}\label{eqn:const van 1}
 \overline{W} \cap (X \times \overline{F}_2 ^{\epsilon}) \hookrightarrow \overline{W} \to \overline{Z}_1
 \end{equation}
 is also surjective due to the dimension reason. 
 
 On the other hand, $\overline{Z}_1 \cap (X \times \overline{F})$ is the image under the projection $\widehat{pr}_1$ of $\overline{W} \cap (X \times \overline{G})$ for the codimension $1$ proper face $G \subset \square^{n+1}$ given by the same equation of $F$ in the larger space. By the surjectivity of \eqref{eqn:const van 1}, the subscheme $\overline{W} \cap (X \times \overline{G}')$, where $ \overline{G}' = \overline{G} \cap \overline{F}_2 ^{\epsilon}$, also has the same image $\overline{Z}_1 \cap (X \times \overline{F})$.
 
 However, $\overline{W} \cap (X \times \overline{F}_2 ^{\epsilon})$ has the codimension $n$ in $X \times \overline{F}_2 ^{\epsilon} \simeq X \times \overline{\square}^n$, and each of its components satisfies $(SF)_*$ inherited from the property $(SF)_*$ of $W$. Hence by Lemma \ref{lem:proper int face *}, its proper extended faces are empty, in particular $\overline{W} \cap (X \times \overline{G}') = \emptyset$. Hence its image $\overline{Z}_1 \cap (X \times \overline{F})$ under $\widehat{pr}_1$ is also empty. This proves the Claim.
 
The claim implies the condition $(GP)_*$ for $Z_1$.
 
 \medskip
 
 We now check the condition $(SF)_*$ for $Z_1$. In case $F \subset \square_k ^1$ is a proper face, by the Claim we have $\overline{Z}_1 \cap (X \times \overline{F}) = \emptyset$ so that we deduce that $\overline{Z}_1 \cap (\{ p \} \times \overline{F}) = \emptyset$ as well. In particular, the intersection is proper.
 
 In case $F= \square^1$, recall that we know that $\overline{W} \to X$ is surjective by Lemma \ref{lem:SF DF}, and by construction, this map factors via $\overline{Z}_1$. Hence $\overline{Z}_1 \to X$ is also surjective. Thus $\overline{Z}_1 \not \subset \{ p \} \times \overline{\square}^1$, and this implies that the intersection $\overline{Z}_1 \cap (\{ p \} \times \overline{\square}^1)$ is proper on $X \times \overline{\square}^1$, because $\{ p \} \times \overline{\square}^1$ is an integral divisor of $X \times \overline{\square}^1$. Hence $Z_1 \in z_{{\rm d}} ^1 (X, 1)$.
 
 \medskip
 
Now since $y_1$ is a vanishing coordinate of $W$, we have $W \cap (\{p \} \times \square^{n+1}) = \emptyset$, while $\overline{W} \cap (\{ p \} \times \{ y_1=1 \} ) \not = \emptyset$. Hence, projecting to $\overline{Z}_1 = \widehat{pr}_1 (\overline{W})$, we have
$Z_1 \cap (\{ p \} \times \square^1) = \emptyset$, while $\overline{Z}_1 \cap (\{ p \} \times \{ y_1=1 \}) \not = \emptyset$. Thus $Z_1$ is a pre-vanishing cycle. Hence $Z_1 \in z_{{\rm v}} ^1 (X, 1)$. 
\end{proof}

\begin{lem}\label{lem:const van}
Let $n \geq 2$ be an integer. Let $W \in z^n _{{\rm v}} (X,n+1)$ be an integral cycle. Suppose $W$ has a constant vanishing coordinate $y_i$ for some $1 \leq i \leq n+1$. 

Then ${\rm Dec} (\partial W) = 0$ in $\CH^1_{{\rm v}} (X, 1) \otimes_{\mathbb{Z}} \CH^{n-1} (k, n-1).$
\end{lem}

\begin{proof}
\textbf{Case 1:} If \emph{all} coordinates $y_i$ for $1 \leq i \leq n+1$ are constant coordinates of $W$, then by Lemma \ref{lem:const}, for each $1 \leq i \leq n+1$ and each $\epsilon \in \{ 0, \infty \}$, we have $\partial_i ^{\epsilon} (W) = 0$. In particular, ${\rm Dec} (\partial_i ^{\epsilon} (W)) = 0$ so that ${\rm Dec} (\partial W) = 0$. 

\medskip

\textbf{Case 2:} Suppose that there exists a coordinate that is \emph{not} a constant coordinate of $W$. 
After applying a suitable permutation, we may assume $y_1$ is a constant vanishing coordinate for $W$ and $y_2$ is \emph{not} a constant coordinate for $W$. Then in terms of the notations in Lemma \ref{lem:const proj}, we have $Z_1 \in z_{{\rm v}} ^1 (X, 1)$.

Since $y_1$ is a constant coordinate of $W$, for $i \geq 2$ and $\epsilon \in \{ 0, \infty\}$, each component $Z$ in $\partial_i ^{\epsilon} (W) \in z_{{\rm v}} ^n (X, n)$ has $Z^{(1)}= Z_1$. So, we may write $(\partial_i ^{\epsilon} (W))^{(1)} = Z_1$. We use the notations of the compactified projections as in Definition \ref{defn:dec 00}.

Since $y_2$ is not a constant coordinate for $W$, for the compactified projection
$$
\widehat{pr}_{ \{ 1' \}} : \overline{\square}_X ^{n+1} \to \overline{\square}_X ^n
$$
that ignores only $y_1$, the morphism $\overline{W} \to \overline{W}^{(1')}$ is quasi-finite, where $\{ 1' \} := \{ 2, 3, \cdots, n+1\}$ as in Definition \ref{defn:comp proj 0}. Thus by Lemma \ref{lem:adm comp proj 2} we have $ W^{(1')} \in z_{{\rm d}} ^{n-1} (X, n)$. Applying the specialization ${\rm ev}_p$ at $p$, we have $W':= {\rm ev}_p ( W^{(1')} ) \in z ^{n-1} (k, n)$. Hence for $2 \leq i \leq n+1$ and $\epsilon \in \{ 0, \infty \}$, we have
$$
{\rm Dec} (\partial_i ^{\epsilon} (W)) = Z_1 \otimes \partial_{i-1} ^{\epsilon} (W') \in z_{{\rm v}} ^1 (X, 1) \otimes z^{n-1} (k, n-1),
$$
while by Lemma \ref{lem:const}-(1), we have ${\rm Dec} (\partial_1 ^{\epsilon} (W)) = 0$. Hence
$$
{\rm Dec} (\partial W) = Z_1 \otimes (-1) \cdot \partial W'  \equiv Z_1 \otimes  0 = 0 \ \ \mbox{ in } \CH_{{\rm v}} ^1 (X, 1) \otimes \CH^{n-1} (k, n-1),
$$
because $\partial W' \equiv 0$ in $ \CH^{n-1} (k, n-1) = z^{n-1} (k, n-1) / \partial z^{n-1} (k, n)$ by definition.
\end{proof}

\subsection{The number of vanishing coordinates}

For $W \in z_{{\rm v}} ^n (X, n+1)$, the number of vanishing coordinates plays an important role in analyzing ${\rm Dec} (\partial W)$. Recall from Remark \ref{remk:type i} that the number of vanishing coordinates for $W$ is $\geq 1$.

\begin{lem}\label{lem:single van}
Let $n \geq 2$ be an integer. Let $W \in z^n_{{\rm v}} (X, n+1)$ be an integral cycle.

Suppose there is only one vanishing coordinate $y_i$ for some $1 \leq i \leq n+1$. Then:

\begin{enumerate}
\item $y_i$ is a constant vanishing coordinate for $W$.
\item ${\rm Dec} (\partial W) = 0$ in $\CH_{{\rm v}} ^1 (X, 1) \otimes_{\mathbb{Z}} \CH^{n-1} (k, n-1)$.
\end{enumerate}
\end{lem}

\begin{proof}
(1) Without loss of generality, we may assume $i=1$, i.e. $y_1$ is the unique vanishing coordinate of $W$. We claim that it is a constant coordinate.

\medskip

If $y_1$ is \emph{not} a constant coordinate, then by Lemma \ref{lem:const}-(2), we have $\overline{W} \cap (X \times \overline{F}_1 ^{\epsilon}) \not = \emptyset$ for $\epsilon \in \{ 0, \infty \}$.

Each nonempty component $\overline{Z} \subset \overline{W} \cap (X \times \overline{F}_1 ^{\epsilon})$ satisfies
\begin{equation}\label{eqn:single van 1}
\overline{Z} \cap ( \{ p \} \times \overline{\square}^{n+1}) \subset \overline{W} \cap (X \times \overline{F}_1 ^{\epsilon}) \cap ( \{ p \} \times \overline{\square}^{n+1})
\end{equation}
$$
= \overline{W} \cap ( \{ p \} \times \overline{\square}^{n+1}) \cap (X \times \overline{F}_1 ^{\epsilon}).
$$

By the given assumption, $y_1$ is the only vanishing coordinate of $W$ so that using \eqref{eqn:y_i=1 1} in \S \ref{sec:vanishing}, we have
\begin{equation}\label{eqn:single van 2}
\overline{W} \cap ( \{ p \} \times \overline{\square}^{n+1}) \subset \{ p \} \times \{ y_1 = 1 \}.
\end{equation}

Hence combining \eqref{eqn:single van 1} and \eqref{eqn:single van 2}, and recalling that $\overline{F}_1 ^{\epsilon} = \{ y_1 = \epsilon \}$, we have
\begin{equation}\label{eqn:single van 3}
\overline{Z} \cap ( \{ p \} \times \overline{\square}^{n+1}) \subset ( \{ p \} \times \{ y_1= 1 \}) \cap ( X \times \{ y _ 1= \epsilon \}) =^{\dagger} \emptyset,
\end{equation}
where $=^{\dagger}$ holds because $\epsilon \not = 1$. In particular, $\overline{Z} \to X$ is \emph{not} surjective. However, the condition $(SF)_*$ for $W$ induces the condition $(SF)_*$ for $Z$, and by Lemma \ref{lem:SF DF} the morphism $\overline{Z} \to X$ is proper and dominant, thus surjective. This is a contradiction. Hence $y_1$ is indeed a constant coordinate, thus a constant vanishing coordinate, proving (1).

(2) follows immediately from (1) together with Lemma \ref{lem:const van}.
\end{proof}

\begin{lem}\label{lem:>2 van}
Let $n \geq 2$ be an integer. Let $W \in z^n _{{\rm v}} (X, n+1)$ be an integral cycle. Suppose $W$ has at least three vanishing coordinates.

Then we have ${\rm Dec} ( \partial_i ^{\epsilon} (W) ) = 0$ in $\CH_{{\rm v}} ^1 (X, 1) \otimes \CH^{n-1} (k, n-1)$ for all $1 \leq i \leq n+1$ and $\epsilon \in \{0, \infty \}$. In particular, ${\rm Dec} (\partial W ) = 0$.
\end{lem}

\begin{proof}
If $\partial_i ^{\epsilon} (W)=0$ for some $1 \leq i \leq n+1$ and $\epsilon \in \{ 0, \infty\}$, then apparently ${\rm Dec} (\partial_i ^{\epsilon} (W)) = 0$ for this face.

\medskip

Suppose $\partial _i^{ \epsilon} (W) \not = 0$ for some $1 \leq i \leq n+1$ and $\epsilon \in \{ 0, \infty\}$. Let $Z$ be any nonempty component of $\partial_i ^{\epsilon} (W)$. Since $W$ has at least three vanishing coordinates, this $Z$ has at least two vanishing coordinates. In this case, we have ${\rm Dec} (Z) = 0$ by Lemma \ref{lem:decon well}. Hence ${\rm Dec} (\partial_i ^{\epsilon} (W)) = 0$ as well.
\end{proof}

The remaining case is the following:

\begin{lem}\label{lem:only 2 vanishing}
Let $n \geq 2$ be an integer. Let $W \in z^n_{{\rm v}} (X, n+1)$ be an integral cycle with exactly two vanishing coordinates. 

Then ${\rm Dec} (\partial W) = 0$ in $\CH_{{\rm v}} ^1 (X, 1) \otimes_{\mathbb{Z}} \CH^{n-1} (k, n-1).$
\end{lem}

\begin{proof}
If any one of the vanishing coordinates is a constant coordinate, then by Lemma \ref{lem:const van}, the assertion of Lemma \ref{lem:only 2 vanishing} holds automatically. 

Hence, we may assume that both of the two vanishing coordinates are non-constant coordinates. After a suitable permutation of the coordinates, we may assume that $W$ has exactly two vanishing coordinates $y_1, y_2$, both of them are non-constant coordinates.

\medskip

Remark that the special fiber $\overline{W}_p = \overline{W} \cap ( \{ p \} \times \overline{\square}^{n+1})$ of the closure $\overline{W}$ is connected: this follows from that $\overline{W} \to X$ is proper flat and surjective (using Lemma \ref{lem:SF DF} and that $X$ is regular of dimension $1$), so that the connectedness of the generic fiber implies that of the special fiber.

\medskip

\textbf{Claim 1:} \emph{All the remaining coordinates $y_3, \cdots, y_{n+1}$ are constant on the special fiber $\overline{W}_p$.}

\medskip

The given condition that \emph{none} of $y_i$ for $i \geq 3$ is a vanishing coordinate of $W$ implies that (Remark \ref{remk:type i}) there is no point of $\overline{W}_p$ which intersects with the divisor $\{ p \} \times \{ y_i = 1 \} \subset \{ p \} \times \overline{\square}^{n+1}_k = \overline{\square}^{n+1}_k$, i.e. for $i \geq 3$, we have
$$
\overline{W}_p \cap ( \overline{\square}^{i-1}_k \times \{ 1 \} \times \overline{\square}^{n+1-i}_k) = \emptyset.
$$
 In particular, the image of $\overline{W}_p$ under the projection to the $i$-th coordinate $\overline{\square}^{n+1}_k \to \overline{\square}_k ^1$ is \emph{not} surjective. Being the image of a connected projective scheme, the image is a connected proper closed subset of $\overline{\square}_k^1= \mathbb{P}^1_k$, thus a closed point. Hence for $i \geq 3$, $y_i$ is constant on the special fiber, proving the Claim 1.

\medskip

\textbf{Claim 2:} \emph{ ${\rm Dec} (\partial W) = \sum_{i=1} ^2 (-1)^i  {\rm Dec}  ( ( \partial_i ^{\infty} (W)) - \partial_i ^0 (W))$.}

\medskip

For the proof of Claim 2, it is enough to show that ${\rm Dec} (\partial_i ^{\epsilon} (W))=0$ for each $3 \leq i \leq n$ and $\epsilon \in \{ 0, \infty \}$.

If a face $\partial_i ^{\epsilon} (W)= 0$, then ${\rm Dec} (\partial_i ^{\epsilon} (W)) = 0$, so it is OK.

If $\partial_i ^{\epsilon} (W)\not = 0$, then since $3 \leq i \leq n+1$, each of its component has two vanishing coordinates $y_1, y_2$, so its deconcatenation is $0$, because its image in the $z^{n-1} (k, n-1)$ part obtained by taking the evaluation at the special fiber is $0$, thus it is $0$ in the $\CH^{n-1} (k, n-1)$ as in Lemma \ref{lem:decon well}. Thus ${\rm Dec} (\partial_i ^{\epsilon} (W)) = 0$ in this case. This proves the Claim 2.

\medskip

Let $Z'$ denote the image of the projection of $\overline{W}_p$ under the projection $\widehat{pr}_{\{3, \cdots, n+1\}}: \overline{\square}^{n+1} \to \overline{\square}^{n-1}$ that forgets $y_1$ and $y_2$. By the Claim 1, it is a closed point, and one notes that $Z' \in z^{n-1} (k, n-1)$. 

Since $y_1, y_2$ are non-constant coordinates, the compactified projection $\overline{W} \to \overline{W} ^{ \{ 1,2 \}}$ is quasi-finite. Hence by Lemma \ref{lem:adm comp proj 2}, we have $W^{ \{ 1, 2 \}} \in z_{{\rm v}} ^1 (X, 2)$. For $i=1,2$, we deduce that ${\rm Dec} ( \partial _i ^{\epsilon} (W)) = \partial_i ^{\epsilon} (W ^{ \{ 1,2 \}} ) \otimes Z' \in \CH^1 _{{\rm v}} (X, 1) \otimes \CH^{n-1} (k, n-1)$. Hence by the Claim 2, ${\rm Dec} (\partial W) = \partial ( W ^{ \{ 1,2 \}} ) \otimes  Z' \in \CH^1 _{{\rm v}} (X, 1) \otimes \CH^{n-1} (k, n-1)$, where $\partial (W ^{ \{ 1 , 2 \}} ) \equiv  0$ in $\CH^1_{{\rm v}} (X, 1)= z^1 _{{\rm v}} (X, 1) / \partial z^1 _{{\rm v}} (X, 2)$ by definition. This implies that ${\rm Dec} (\partial (W)) = 0$. proving the lemma.
\end{proof}

\subsection{Proof of the proposition}\label{sec:4.4}

We are now ready to finish:

\medskip

\begin{proof}[Proof of Proposition \ref{prop:decon bdry 0}]
Let $W \in z^n_{{\rm v}} (X, n+1)$ be an integral cycle. Since it is a strict vanishing cycle, there is at least one vanishing coordinate of $W$ by Remark \ref{remk:type i}. Let $r \geq 1$ be the number of the vanishing coordinates of $W$.

Here, for $r=1$ by Lemma \ref{lem:single van}, for $r \geq 3$ by Lemma \ref{lem:>2 van}, and for $r=2$ by Lemma \ref{lem:only 2 vanishing}, we have ${\rm Dec} (W) = 0$ in $\CH^1_{{\rm v}} (X, 1) \otimes \CH^{n-1} (k, n-1)$, as desired.
\end{proof}

\begin{cor}\label{cor:decon mod m}
Let $m \geq 1$ be an integer. Then we have the induced deconcatenation homomorphism
$$
{\rm Dec} : \CH^n _{{\rm v}} (X/(m+1), n) \to \CH^1 _{{\rm v}} (X/(m+1), 1) \otimes \CH^{n-1} (k, n-1).
$$
\end{cor}

\begin{proof}
We want to complete the following commutative diagram by proving that the top horizontal map coming from Proposition \ref{prop:decon bdry 0} induces the bottom horizontal one:
$$
\xymatrix{
\CH_{\rm v}^n (X, n) \ar@{>>}[d]^f \ar[rr] ^{{\rm Dec} \ \ \ \ \ \ \ \ \ \ } & & \CH_{\rm v}^1 (X, 1) \otimes \CH^{n-1} (k, n-1) \ar@{>>}[d]^g \\
\CH_{\rm v} ^n (X/ (m+1), n) \ar@{-->}[rr] ^{ ? \ \ \ \ \ \ \ \ \ \  } & & \CH_{\rm v} ^1 (X/ (m+1), 1) \otimes \CH^{n-1} (k, n-1),}
$$
where $f, g$ are the natural quotient maps.

\medskip

It is enough to show that if a pair of integral cycles $Z_1, Z_2 \in z^n_{{\rm v}} (X, n)$ satisfies 
$$
\overline{Z}_1 \times _{X} X_{m+1} = \overline{Z}_2 \times_X X_{m+1}
$$
as closed subschemes of $X_{m+1} \times \overline{\square}^n$, where $X_{m+1} = \Spec (k_{m+1})$, then their images in $\CH_{\rm v}^1 (X/ (m+1), n) \otimes \CH^{n-1} (k, n-1)$ under $g \circ {\rm Dec}$ coincide. After applying a suitable permutation, we may assume that $y_1$ is a common vanishing coordinate for $Z_1$ and $Z_2$.

\medskip

If there is another common vanishing coordinate $y_i$ for $i \geq 2$, then we have ${\rm Dec} (Z_j) = 0$ for $j=1,2$ by Lemma \ref{lem:decon well}, so the statement of the Corollary holds trivially.

\medskip

Thus, now suppose that none of $y_i$ for $i \geq 2$ is a common vanishing coordinate for $Z_1$ and $Z_2$. For the compactified projections $\widehat{pr} ^{ (1)} (\overline{Z}) \subset  \overline{\square}_X ^1$ as in Definition \ref{defn:comp proj 0} and Lemma \ref{lem:adm compact proj}, we have $\widehat{pr} ^{(1)} (Z_j) \in z_{{\rm v}} ^1 (X, 1)$ for $j=1,2$, and 
$$
\widehat{pr} ^{(1)} (\overline{Z}_1) \times_X  X_{m+1} = \widehat{pr} ^{(1)} (\overline{Z}_2) \times_X  X_{m+1}
$$
as closed subschemes of $X_{m+1} \times \overline{\square}^1$, while
$$
 {\rm ev}_p \left( \widehat{pr} ^{(1')} (Z_1) \right) = {\rm ev}_p \left( \widehat{pr} ^{(1')} (Z_2) \right) 
$$
in $z^{n-1} (k, n-1)$. Let's call this common cycle by $Z'$. Thus we deduce that 
$$
 [ \widehat{pr} ^{(1)} (Z_1)] \otimes [Z'] = {\rm Dec} (Z_1) ={\rm Dec} (Z_2) = [ \widehat{pr} ^{(1)} (Z_1)] \otimes [Z']
$$
in $\CH_{\rm v} ^1 (X/ (m+1), 1) \otimes \CH^{n-1} (k, n-1).$ This proves the corollary.
\end{proof}

Recall the following result, from e.g. Krishna-Park \cite[(7.4)]{KP crys}):

\begin{thm}\label{thm:KPCM7.4}
Let $R$ be a regular semi-local algebra essentially of finite type over a field $k$. Let $m, n \geq 1$ be integers.

Then we have the natural homomorphism
$$
{\rm Con}:  \mathbb{W}_m (R) \otimes_{\mathbb{Z}} K_{n-1} ^M (R) \to \mathbb{W}_m \Omega_R ^{n-1}
$$
which maps
$
b \otimes \{ a_1, \cdots, a_{n-1} \} \mapsto b d \log [a_1] \wedge \cdots \wedge d \log [a_{n-1}].
$
\end{thm}

The existence of ${\rm Dec}$ in Corollary \ref{cor:decon mod m} for $1 \leq m$ induces the following:

\begin{defn}\label{defn:additive regulator}
Let $X= \Spec (k[[t]])$ for a field $k$. For $1 \leq m $, define the following homomorphism, a composite of ${\rm Con}$ with ${\rm Dec}$:
$$
\rho_k : \CH^n_{{\rm v}} (X/ (m+1), n) \overset{\rm Dec}{\to} \CH^1_{{\rm v}} (X/ (m+1), 1) \otimes \CH^{n-1} (k, n-1) 
$$
$$
\hskip1cm  \overset{\simeq^{\dagger}}{\to} \mathbb{W}_m (k) \otimes K_{n-1} ^M (k) \overset{\rm Con}{\to} \mathbb{W}_m\Omega_k ^{n-1},$$
where the isomorphism $\dagger$ follows from Theorems \ref{thm:NST} and \ref{thm:graph final}, while the homomorphism ${\rm Con}$ is the map  in Theorem \ref{thm:KPCM7.4} with $R=k$.
\qed
\end{defn}

This map $\rho_k$ is \emph{not} an isomorphism in general, in particular it is \emph{not} the desired inverse of the inverse Bloch map $\varphi$. However, we can use it to prove the injectivity of $\varphi$ in \S \ref{sec:inj inverse Bloch}.

\subsection{Injectivity of the inverse Bloch map}\label{sec:inj inverse Bloch}

With the helps of the homomorphisms $\rho_k$ of Definition \ref{defn:additive regulator}, we are now ready to prove:

\begin{lem}\label{lem:inverse Bloch inj}
The inverse Bloch map of \eqref{eqn:inverse Bloch cy}
$$
\varphi_n: \TH^n (k, n;m)' \to \CH_{{\rm v}} ^n (X/ (m+1), n)
$$
for $X= \Spec (k[[t]])$ is injective. 
\end{lem}

\begin{proof}
Since we use the above $\varphi_n$ over different base fields, we denote it by $\varphi_k$ to emphasize the base field $k$ in the following argument.

\medskip

Let $\mathfrak{p} \in \TZ^n (k, n;m)$ be an integral $0$-cycle. It is enough to show that we can recover the cycle class $[\mathfrak{p}] \in \TCH^n (k, n;m)'$ from $\varphi_k ([\mathfrak{p}])$. 

\medskip

Let $k':= \kappa (\mathfrak{p})$ be the residue field and let $\pi: \Spec (k') \to \Spec (k)$ be the associated morphism. Let $\varphi_{k'}$ be the inverse Bloch map corresponding to $k'$. There exists a $k'$-rational point $\mathfrak{p}'  \in B_n$ and its class $[\mathfrak{p}'] \in \TZ^n (k', n;m)$ such that we have $\pi_* ([\mathfrak{p}']) = [\mathfrak{p}]$.

 Consider the commutative diagram

$$
\xymatrix{
\TCH^n (k', n;m)' \ar[d] ^{\pi_*} \ar[r]^{\varphi_{k'}} & \widehat{K}_n ^M (k'_{m+1}, (t))  \ar[d] ^{{\rm Tr}_{k'/k}} \ar[r] ^{\ \ \ \ \rho_{k'}} & \mathbb{W}_m \Omega_{k'}^{n-1}   \ar[d] ^{{\rm Tr}_{k'/k}} \ar[r]^{gr_{k'}\ \ \ \ } _{\simeq \ \ \ \ }  & \TCH^n (k', n;m) ' \ar[d] ^{\pi_*} \\
\TCH^n (k, n;m)' \ar[r] ^{\varphi_k} & \widehat{K}_n ^M (k_{m+1}, (t)) \ar[r] ^{\ \ \ \ \rho_k} & \mathbb{W}_m \Omega_k ^{n-1} \ar[r] ^{gr_k \ \ \ \ }_{\simeq \ \ \ \ } & \TCH^n (k, n;m)',}
$$
where $\rho_{k}$ and $\rho_{k'}$ are the maps of Definition \ref{defn:additive regulator} under the identifications $\widehat{K}^M_{n} (k_{m+1}, (t)) \simeq \CH^n_{{\rm v}} (X/ (m+1), n)$ and $\widehat{K}^M_{n} (k'_{m+1}, (t)) \simeq \CH^n_{{\rm v}} (X_{k'}/ (m+1), n)$ of Theorem \ref{thm:graph final}, with $X_{k'} = \Spec (k'[[t]])$, $gr_k$ and $gr_{k'}$ are the isomorphisms of R\"ulling \cite{R} (Theorem \ref{thm:Rulling2}), the maps $\pi_*$ are the push-forwards of cycles, the first ${\rm Tr}_{k'/k}$ is the push-forward constructed in \cite[Theorem 1.1.2]{Park presentation} (recalled in Theorem \ref{thm:full transfer}) and the second $\Tr_{k'/k}$ is the trace map on the big de Rham-Witt forms constructed by K. R\"ulling \cite[Theorem 2.6]{R}.

Since $\pi_* ([ \mathfrak{p}']) = [ \mathfrak{p}]$, to show that $\varphi_k ([\mathfrak{p}])$ determines $[\mathfrak{p}] \in \TCH^n (k, n;m)'$, it is enough to show that $\varphi_{k'} ([\mathfrak{p}'])$ determines $[\mathfrak{p}'] \in \TCH^n (k', n;m)'$. We check it directly in what follows.

Since $\mathfrak{p}'$ is a $k'$-rational point of $B_n=\mathbb{A}^1 \times \square^{n-1}$, we can express it as
$$
 \mathfrak{p}' = \left( a, b_1, \cdots, b_{n-1} \right) \in \left(\mathbb{A}^1 \times  \square^{n-1}\right)(k'),
 $$
where $a \in k' \setminus \{ 0 \}$ and $b_i \in k' \setminus \{ 0, 1 \}$. 
Hence by the definition of $\varphi_{k'}$, we have
$$
 \varphi_{k'} ([\mathfrak{p}']) = \left\{ 1- \frac{ t}{a}, b_1-a t, \cdots, b_{n-1}-a t \right\} \in \widehat{K}_n ^M (k'_{m+1}, (t)).
 $$
For each $1 \leq i \leq n-1$, we have 
$$
\{b_i - a t \}=  \left\{ b_i \left( 1- \frac{ at}{ b_i} \right) \right\} = \{b_i \} + \left\{1- \frac{ a t}{b_i}\right\},
$$
 and by the multi-linearity and the distributive law on the Milnor symbols, we can express $\varphi_{k'}([\mathfrak{p}'])$ as the sum of the two parts
\begin{eqnarray*}
& &  \left\{ 1- \frac{ t}{a}, b_1, \cdots, b_{n-1} \right\}  + \left(  \left\{ 1- \frac{ t}{a}, 1 - \frac{at}{b_1} , b_2- a t, \cdots, b_{n-1} - a t  \right\} + \cdots \right) \\
&=:& \varphi^1_{k'} ([\mathfrak{p}']) + \varphi^2 _{k'}([\mathfrak{p}']),
\end{eqnarray*}
where $\varphi^2 _{k'} ([\mathfrak{p}']) $ has $(2^{n-1}-1)$ terms of Milnor symbols, and each has at least two coordinates $\equiv 1 \mod t$, e.g. $y_1 = 1 - \frac{t}{a}$ and $y_i = 1 - \frac{a t}{b_{i-1}}$ for some $i \geq 2$. Hence $\varphi^2 _{k'} ([\mathfrak{p}']) \in \ker ( \rho_{k'})$ by Lemma \ref{lem:decon well}. Thus 
$$
\rho_{k'} (\varphi_{k'} ([\mathfrak{p}']) )  = \rho_{k'} (\varphi^1_{k'} ([\mathfrak{p}']) ) = \frac{1}{ [ a]} \frac{ d [ b_1]}{[b_1]} \wedge \cdots \wedge \frac{d  [ b_{n-1}]}{[ b_{n-1}]},
$$
where $[c]$ for $c \in k'$ denotes the Teichm\"uller lift of $c$ in $\mathbb{W}_m (k')$. Following the definition of the isomorphism $gr_{k'}$ of R\"ulling \cite{R} (Theorem \ref{thm:Rulling2}), this gives
$$
gr_{k'} \left(  \frac{1}{ [ a]} \frac{ d [ b_1]}{[b_1]} \wedge \cdots \wedge \frac{d  [ b_{n-1}]}{[ b_{n-1}]} \right)= [ (a, b_1, \cdots, b_{n-1})] = [ \mathfrak{p}'] \in \TCH^n (k', n;m)',
$$
recovering the class of $[\mathfrak{p}']$. Thus applying $\pi_*$, we recover $[\mathfrak{p}]$ from $\rho_k ([\mathfrak{p}])$ as well. This proves the injectivity of $\varphi_k$. 
\end{proof}

\begin{remk}
In the above proof, the reader notes that $\varphi^1_{k'} ([\mathfrak{p}'])$ remembers $[\mathfrak{p}']$ and $\rho_{k'}$ behaves like a filter that allows us to recover $[\mathfrak{p}']$ from $\varphi_{k'} ([\mathfrak{p}'])$. The remining filtered part of $\varphi$, namely $\varphi_{k'} ^2 ([\mathfrak{p}'])$, plays a role in the proof of the surjectivity behind the scene.
\qed
\end{remk}

\section{Surjectivity of the inverse Bloch map}\label{sec:surj inverse Bloch}

In \S \ref{sec:surj inverse Bloch}, we prove the surjectivity of the inverse Bloch map $\varphi$ of \eqref{eqn:inverse Bloch cy}. This is the last remaining part of the proof of Theorem \ref{thm:intro main vanishing}. We still use $X = \Spec (k[[t]])$ and $p\in X$ is the closed point.

\subsection{A prepration}

Note that for each irreducible cycle $Z \in z^1_{{\rm d}} (X, 1)$, its closure $\overline{Z} \subset X \times \overline{\square}^1$ satisfies $\overline{Z} \cap \{ y_1 = \infty \} = \emptyset$ by Lemma \ref{lem:proper int face *}. Thus $\overline{Z} \subset X \times \mathbb{A}^1$ in fact, and it is defined by a height $1$ prime ideal $P \subset k[[t]] [ y_1]$ of the UFD. Hence there exists a unique irreducible polynomial $f= f_Z (t, y_1) \in k[[t]] [y_1]$ monic in $y_1$, that defines $\overline{Z}$ (see also \cite[\S 4.2]{Park presentation}). Its special fiber $Z_s$ over $p \in X$ is given by $f_Z (0, y_1 ) \in k[y_1]$, and the associated cycle on $\square_k ^1$ is denoted by $[Z_s ]$ in what follows.

\begin{defn}
Let $z^1_{{\rm d, nc}} (X, 1) \subset z^1 _{{\rm d}} (X, 1)$ be the subgroup generated by the irreducible cycles of the form $Z=\pi_* ( \{ y = c-t \})$, where
$\pi: \Spec (k') \to \Spec (k)$ runs over all finite extensions $k \hookrightarrow k'$ of fields, and $c $ runs over all members of $k' \setminus \{ 0, 1\}$. 
\qed
\end{defn}

Recall from Definition \ref{defn:ACH} that $\TZ^n (k, n;m)$ is independent of $m$.

\begin{defn}\label{defn:J n}
Let $n \geq 1$ be an integer. Let $\mathfrak{p} \in \TZ^n (k, n;m)$ and $Z \in z^1 _{{\rm d, nc}} (X, 1)$ be irreducible cycles. Here, $\mathfrak{p} \in \mathbb{A}^1 \times \square_k ^{n-1}$ and $Z \subset \square_X ^1$.

Define $J_n ([ \mathfrak{p}] \otimes Z )$ to be the $0$-cycle $ [\mathfrak{p}] \boxtimes [Z_s]= [\mathfrak{p} \times_k Z_s ]$ on $\mathbb{A}^1 \times \square^n$. One checks immediately that this belongs to $\TZ^{n+1} (k, n+1;m)$. We define the homomorphism
$$
J_{n}= J_{n, k}: \TZ^n (k, n;m) \otimes z^1 _{{\rm d, nc}} (X, 1) \to \TZ^{n+1} (k , n+1; m)
$$
by $\mathbb{Z}$-bilinearly extending it. 
\qed
\end{defn}

One checks the following from the definition of the push-forward on cycles:

\begin{lem}\label{lem:J n comm}
Let $n \geq 1$ be an integer. Let $\pi: \Spec (k') \to \Spec (k)$ be the morphism corresponding to a finite extension $k \hookrightarrow k'$ of fields. Then we have the commutative diagram
$$
\xymatrix{ 
\TZ^n (k', n;m) \otimes z_{{\rm d, nc}} ^1 (X_{k'}, 1) \ar[d] ^{ \pi_* \otimes \pi_*} \ar[rr] ^{ J_{n,k'}} & & \TZ^{n+1} ( k', n+1;m) \ar[d] ^{\pi_*} \\
\TZ^n (k, n;m) \otimes z_{{\rm d, nc}} ^1 (X, 1) \ar[rr]^{J_{n,k}} & & \TZ^{n+1} (k, n+1; m)},
$$
where $X_{k'}= \Spec (k'[[t]])$.
\end{lem}

\subsection{The twist map and the inverse Bloch map}

Before we proceed further, we introduce the following actions:
\begin{defn}\label{defn:twist} Let $X= \Spec (k[[t]])$.
\begin{enumerate}
\item For $a \in k^{\times}$, consider the scaling automorphism of the $k$-scheme
$$
 \tau_a: X \times \square^1 \to X \times \square^1
 $$
given by sending $t \mapsto at$. It induces the automorphism
$$ 
\tau_a ^*: z^1_{{\rm d}} (X, 1) \to z^1_{{\rm d}} (X, 1).
$$
We call it the \emph{twist by $a$}.

\item We generalize the above as follows. Let $\mathfrak{p} \in \TZ^n (k, n;m)$ be an integral $0$-cycle. Let $k':= \kappa (\mathfrak{p})$ be the residue field and let $\pi: \Spec (k') \to \Spec (k)$ be the corresponding morphism. There is a $k'$-rational point $\mathfrak{p}' \in B_n= \mathbb{A}^1 \times \square^{n-1}$ such that $\pi_*([\mathfrak{p}']) = [ \mathfrak{p}]$. Let $\mathfrak{q}' \in \mathbb{A}^1$ be the projection of the closed point $\mathfrak{p}' \in B_n$ to $\mathbb{A}^1$, which is a closed point $\not = 0$. Let $c \in k'\setminus \{ 0 \}$ be the coordinate value of $\mathfrak{q}'$. 

Let $Z \in z_{{\rm d}} ^1 (X, 1)$ be an irreducible cycle. Define
$$ 
\tau_{\mathfrak{p}} ^*: z_{{\rm d}} ^1 (X, 1) \to z_{{\rm d}} ^1 (X, 1)
$$
to be the map given by $Z \mapsto  \pi_* \left( \tau_c^* ( \pi^* (Z)) \right)$, where $\pi^*: z_{{\rm d}} ^1 (X, 1) \to z_{{\rm d}} ^1 (X', 1)$ and $\pi_* : z_{{\rm d}} ^1 (X', 1) \to z_{{\rm d}} ^1 (X, 1)$ are the pull-back and the push-forward for the base change $\pi: X'= \Spec (k'[[t]]) \to X$, and $\tau_c ^*$ is as in (1) over the base field $k'$. 

This induces the $\mathbb{Z}$-bilinear map 
$$
\tau: \TZ^n (k,n;m) \otimes z^1_{{\rm d}} (X, 1) \to z^1_{{\rm d}} (X, 1)
$$
given by $\mathfrak{p} \otimes Z \mapsto \tau_{\mathfrak{p}} ^* (Z)$. We call it the \emph{twist action}.

\item We define the \emph{twist map}
$$
{\rm tw}:=({\rm Id}, \tau) : \TZ^n (k,n;m) \otimes z^1_{{\rm d}} (X, 1) \to \TZ^n (k, n;m) \otimes z^1_{{\rm d}} (X, 1)
$$
given by ${\rm tw} (\mathfrak{p} \otimes Z) := \mathfrak{p} \otimes \tau ( \mathfrak{p} \otimes Z) = \mathfrak{p} \otimes \tau_{\mathfrak{p}} ^* (Z)$. \qed
\end{enumerate}
\end{defn}

\begin{prop}\label{prop:J comm}
Let $n \geq 1$ be an integer. The following diagram commutes: 
\begin{equation}\label{eqn:J comm 0}
\xymatrix{
\TZ^n (k, n;m) \otimes z^1_{{\rm d, nc}} (X, 1) \ar[d] _{ {\rm Id} \otimes \iota} \ar[rr]^{J_{n,k}}  \ar@{-->}[ddddrr]^{G_{n,k}}& & \TZ^{n+1} (k, n+1; m) \ar[ddd] ^{\varphi_{n+1}} \\
\TZ^n (k, n;m) \otimes z^1 _{{\rm d}} (X, 1) \ar[d] _{{\rm tw}=({\rm Id}, \tau)} & &  \\
\TZ^n (k, n;m) \otimes z^1 _{{\rm d}} (X, 1) \ar[d] _{\varphi_n \otimes {\rm Id}} & &  \\
z^n_{{\rm v}} (X, n) \otimes z^1 _{{\rm d}} (X, 1) \ar[rr] _{-\boxtimes-} & &z^{n+1} _{{\rm v}} (X, n+1) \ar@{->>}[d]^{{\rm can}}   \\
 & &\CH^{n+1}_{{\rm v}} (X/ (m+1), n+1), }
\end{equation}
where $\iota$ is the inclusion, ${\rm tw}$ is the twist map in Definition \ref{defn:twist}, $\boxtimes$ is the juxtaposition map, ${\rm can}$ is the canonical quotient map, and $\varphi_n, \varphi_{n+1}$ are the inverse Bloch maps as in \eqref{eqn:varphi n 0} for $n$ and $n+1$, respectively.
\end{prop}

The broken arrow $G_{n,k}$ is defined to be the counter-clockwise composition of the morphisms. It is used in Lemma \ref{lem:vertical surj}.

\begin{proof}
Let $\mathfrak{p} \in \TZ^n (k, n;m)$ and $Z \in z^1 _{{\rm d, nc}} (X, 1)$ be irreducible cycles.

\medskip

\textbf{Case 1:} First suppose $\mathfrak{p}$ is a $k$-rational point, and $Z=\{y= c-t \}$ for some $c \in k \setminus \{ 0, 1 \}$. Write $\mathfrak{p} = (a, b_1, \cdots, b_{n-1}) \in (\mathbb{A}^1 \times \square^{n-1}) (k)$ for some $a \in k^{\times}$, and $ b_1, \cdots, b_{n-1} \in k \setminus \{ 0, 1 \}$. 

Then $J_{n,k} ([\mathfrak{p}] \otimes Z ) = [ ( a, b_1, \cdots, b_{n-1}, c)]$ so that $(\varphi_{n+1} \circ J_{n,k} ) ( [ \mathfrak{p}] \otimes Z)$ is
\begin{equation}\label{eqn:J comm 1}
 \varphi_{n+1} ( [ (a, b_1, \cdots, b_{n-1}, c)]) = \left ( 1 - \frac{t}{a} , b_1 -at , \cdots, b_{n-1} -at, c - at \right).
\end{equation}

On the other hand, $((\varphi_n \otimes {\rm Id}) \circ {\rm tw}) ( [ \mathfrak{p}] \otimes \iota (Z) ) = \varphi_n ( [ \mathfrak{p}]) \otimes \tau_{a} ^* Z$
\begin{eqnarray*}
&= &\left( 1 - \frac{t}{a}, b_1 -at, \cdots, b_{n-1} - a  t\right) \otimes \tau_{a} ^* [ y=\{ c-  t\} ] \\
&=& \left( 1 - \frac{t}{a}, b_1 - at, \cdots, b_{n-1} - at \right) \otimes [y= \{  c - at \}] 
\end{eqnarray*}
so that after applying $\boxtimes$, the above becomes equal to \eqref{eqn:J comm 1}. Thus the commutativity of \eqref{eqn:J comm 0} holds for such elements.

\medskip

\textbf{Case 2:} In general, let $Z= \pi_* ( \{ y= c-t \})$ for some finite extension $\pi: \Spec (k') \to \Spec (k)$ and $c \in k' \setminus \{ 0, 1 \}$. Let $k''= \kappa (\mathfrak{p})$ be the residue field of $\mathfrak{p}$. By replacing $k'$ and $k''$ by the smallest finite extension $L$ (in an algebraic closure $\bar{k}$ of $k$) containing both $k'$ and $k''$, we may assume $\mathfrak{p}$ is $L$-rational and $Z= \pi_* ( \{ y = c-t \})$, for some $c \in L \setminus \{ 0,1 \}$, where $\pi: \Spec (L) \to \Spec (k)$ is the corresponding morphism. We can find an $L$-rational point $\mathfrak{p}' \in \TZ^n (L, n;m)$ such that $\pi_* ([\mathfrak{p}']) = [ \mathfrak{p}]$. Let $a \in L^{\times}$ be the first coordinate of $\mathfrak{p}' \in ( \mathbb{A}^1 \times \square^{n-1} )(L)$. Then,
\begin{eqnarray*}
J_{n, k} ( [ \mathfrak{p}] \otimes Z ) & =& J_{n,k} ( \pi_* [ \mathfrak{p}'] \otimes \pi_* ( \{ y= c-t \} )) = ^{\dagger}\pi_* \left( J_{n,L} ( [ \mathfrak{p}'] \boxtimes  \{ y= c-t \}) \right) \\
& = & \pi_* \left( [ \mathfrak{p}']  \boxtimes  \{ y = c \} \right),
\end{eqnarray*}
where we used Lemma \ref{lem:J n comm} for $\dagger$. Hence applying $\varphi_{n+1, k}$ to the above equation, we deduce
\begin{eqnarray*}
\varphi_{n+1,k} ( J_{n,k} ( [ \mathfrak{p}] \otimes Z)) & =& \varphi_{n+1,k} \left( \pi_*  \left( [ \mathfrak{p}'] \boxtimes  \{ y = c \} \right) \right) = ^{\dagger}\pi_* \left( \varphi_{n+1, L}   \left( [ \mathfrak{p}'] \boxtimes  \{ y = c \}  \right) \right) \\
& = ^{\ddagger}& \pi_* \left( \varphi_{n, L} ( [ \mathfrak{p}']) \boxtimes  \{ y = c - a t \}  \right) \\
&=& \pi_* \left( \varphi_{n,L} ( [ \mathfrak{p}']) \boxtimes \tau_{\mathfrak{p}'}^* \{ y = c-t \}\right) \\
&=& \pi_* ( \boxtimes \circ (\varphi_{n,L} \otimes {\rm Id}) \circ {\rm tw} \circ ({\rm Id} \otimes \iota) ( [ \mathfrak{p}'] \otimes \{ y = c -t \})) \\
&=& \boxtimes \circ (\varphi_{n,k} \otimes {\rm Id}) \circ {\rm tw} \circ ({\rm Id} \otimes \iota) ( [ \mathfrak{p}] \otimes Z),
\end{eqnarray*}
where $\dagger$ holds due to the way $\varphi_{n+1}$ is defined using push-forwards, $\ddagger$ holds by \textbf{Case 1} applied to $\mathfrak{p}'$ and $\{ y=c \}$ with the field $L$, and for the rest we used that $\pi_*$ commutes with the respective operations that appeared in the above, which is essentially due to how the operations are defined. We shrink some details on checking all the commutativity. This proves the commutativity of \eqref{eqn:J comm 0} in general.
\end{proof}

\begin{lem}\label{lem:gamma m n}
Let $m, n \geq 1$ be integers. Then we have the surjective homomorphism of groups
\begin{equation}\label{eqn:gamma m n}
 \gamma_{m,n}: \widehat{K}_n ^M (k_{m+1}, (t)) \otimes_{\mathbb{Z}} \widehat{K}_1 ^M (k_{m+1}) \to \widehat{K}^M_{n+1} (k_{m+1}, (t))
\end{equation}
induced from the identity map
$$
{\rm Id}: \underset{n}{\underbrace{k_{m+1} ^{\times} \otimes \cdots \otimes k_{m+1} ^{\times} }} \otimes k_{m+1} ^{\times} \to  \underset{n+1}{\underbrace{k_{m+1} ^{\times} \otimes \cdots \otimes k_{m+1} ^{\times} }}.
$$
\end{lem}

\begin{proof} The identity map ${\rm Id}$ induces the apparent surjective map
$$
K_n ^M (k_{m+1}) \otimes K_1 ^M (k_{m+1}) \to K_{n+1} ^M (k_{m+1}).$$
Hence using Lemma \ref{lem:KS}, we deduce the surjective map
$$ K_n ^M (k_{m+1}, (t)) \otimes K_1 ^M (k_{m+1}) \to K_{n+1} ^M (k_{m+1}, (t)) \to \widehat{K}_{n+1} ^M (k_{m+1}, (t)).$$
Since $\widehat{K}_1 ^M (k_{m+1}) = K_1 ^M (k_{m+1})$ and the target satisfies the (\'etale) transfer on $k$ (Theorem \ref{thm:full transfer}), we deduce the surjective homomorphism \eqref{eqn:gamma m n} as desired.
\end{proof}

\subsection{Generators with irreducible representatives}

To simplify our arguments, we continue to use Theorem \ref{thm:graph final} to identify the groups
$
 \CH_{{\rm v}} ^n (X/ (m+1), n) \simeq \widehat{K}_n ^M (k_{m+1}, (t)).
$

\medskip

Recall from Lemma \ref{lem:KS} and the anti-symmetricity, this group is generated by the symbols of the form $\{ \alpha_1, \cdots, \alpha_n \}$, where $\alpha_i \in k_{m+1} ^{\times}$ with $\alpha_1 \equiv 1 \mod t$. Here is an improvement of Lemma \ref{lem:KS} needed for our purpose:

\begin{lem}\label{lem:KS improved}
The group $\widehat{K}_n ^M (k_{m+1}, (t))$ is generated by symbols $\{ \alpha_1, \cdots, \alpha_n \}$, where $\alpha_i \in k_{m+1} ^{\times}$ with $\alpha_1 \equiv 1 \mod t$, such that we can find irreducible polynomial representatives $p_i (t) \in k[t]$ of $\alpha_i$ for all $1\leq i \leq n$. 

Furthermore, we may suppose that an irreducible representative $p_1 (t) \in k[t]$ of $\alpha_1$ is of the form $1- \frac{t^j}{a}$ for some $1 \leq j < m+1$ and $a \in k^{\times}$. 
\end{lem}

\begin{proof}
We make the following elementary observation for $k_{m+1} ^{\times}$.

\medskip

\textbf{Claim :} \emph{ Let $c \in k^{\times}$. Let $\alpha \in k_{m+1} ^{\times}$. Then there exists a non-constant polynomial in $k[t]$ representing $\alpha$, whose leading coefficient is exactly $c$.}

\medskip

Since $k_{m+1} =k[t]/(t^{m+1})$, we can find a polynomial $p (t) \in k[t]$ of degree $< m+1$ representing $\alpha$. In case it does not satisfy the requirement of the claim, choose any integer $d \geq m+1$, and let
$$
p(t):= p_1 (t) + c t^d \in k[t].
$$

Since $d \geq m+1$, this $p(t)$ also represents $\alpha \in k[t] /(t^{m+1}) = k_{m+1}$, and it is certainly a non-constant polynomial whose leading coefficient is exactly $c$, proving the Claim.

\medskip

Now, let $\{ \alpha_1, \alpha_2, \cdots, \alpha_n\}$ be any member of $\widehat{K}_n ^M (k_{m+1}, (t))$ such that $\alpha_i \in k_{m+1} ^{\times}$ and $\alpha_1 \equiv 1 \mod t$. Such elements generate the group by Lemma \ref{lem:KS}. We will replace these generators by the symbols with better properties required by the lemma.

\medskip

Using the Claim, we find non-constant polynomials $p_1 (t), \cdots, p_n (t) \in k[t]$ of degrees $\geq m+1$ that represent $\alpha_1, \cdots, \alpha_n$, respectively. For $i \geq 2$, if any one of them is reducible, say $p_i = p_{i1} \cdots p_{ir}$ for some non-constant irreducible polynomials $p_{ij}\in k[t]$, then for $\alpha_{ij}:= p_{ij} \mod t^{m+1}$ (they are invertible in $k_{m+1}^{\times}$ again) we have
$$
\{ \alpha_i \} = \{ \alpha_{i1} \} + \cdots + \{ \alpha_{in} \}.
$$
Hence replacing $\alpha_i$ by each of $\alpha_{ij}$, and using the distributive law, we may assume that for each $\alpha_i$ with $i \geq 2$, we have found its non-constant irreducible polynomial representative $p_i (t) \in k[t]$.

\medskip
For a representative $p_1 (t)$ of $\alpha_1$, we have $p_1 (t) \equiv \alpha_1 \equiv 1 \mod t$ so that by Proposition \ref{prop:Witt elements}, there is a factorization of $p_1(t)$ into a product of polynomials of the form 
$$
 \prod_{j=1} ^m \left( 1 -b_i t^i \right) \ \ \mod t^{m+1}
$$
for some $b_i \in k$. Hence by the multi-linearity of the Milnor symbols, replacing $\alpha_1$ by each of $1- b_i t^i$, and $p_1$ also by $1- b_i t^i \mod t^{m+1}$, we reduce the proof of the lemma to the case when $\alpha_1 = p_1 (t) = 1 - \frac{ t^i}{a}$ for some $a \in k^{\times}$ (taking $a= b^{-1}$). Depending on $a$ and the base field $k$, it may or may not be reducible. (e.g. $1- \frac{t^2}{r^2} = (1 - \frac{t}{r}) (1 + \frac{t}{r})$ is reducible always, while $1+ \frac{t^2}{r^2}$ may or may not be irreducible, depending on $k$.) If it is reducible, we can break it into a polynomial representative with smaller degrees, and using Proposition \ref{prop:Witt elements} again if needed, and repeating, we can reduce to the case when $p_1 (t) =1- \frac{ t^i}{a}$ is irreducible itself. This proves the lemma. \qed
\end{proof}

\begin{remk}
Because of the shape $1- \frac{t^i}{a}$ of the polynomial, this $p_1(t)$ is irreducible if either $i=1$ or $ i \geq 2$ and $p_1(t) = 0$ has no solution in $k$. \qed
\end{remk}

\subsection{Strong surjectivity}

\begin{defn}\label{defn:single coordinate}
Let's say that a $0$-cycle $C:= \sum_{i=1} ^r n_i \mathfrak{p}_i \in \TZ^n (k, n;m)$ with $n_i >0$ has \emph{a single $\mathbb{A}^1$-coordinate} if all projections of $\mathfrak{p}_i$ to $\mathbb{A}^1$ are equal. 
\qed
\end{defn}

\begin{defn}\label{defn:strong surj}
 Let's say that $\varphi_n$ is \emph{strongly surjective}, if for each generator $\{ \alpha_1, \cdots, \alpha_{n} \} \in \widehat{K}_n ^M (k_{m+1}, (t))$ with chosen irreducible polynomial representatives $p_1 (t), \cdots, p_n (t)$ satisfying the conditions of Lemma \ref{lem:KS improved}, there exists a $0$-cycle $C:=\sum_{i=1} ^r n_i \mathfrak{p}_i \in \TZ^n (k, n;m)$, $n_i > 0$, with a single $\mathbb{A}^1$-coordinate such that $\varphi_n (C) = \{ \alpha_1, \cdots, \alpha_n\}$ in $\widehat{K}_n ^M (k_{m+1}, (t))$. 
 
 Apparently, if $\varphi_n$ is strongly surjective, then it is surjective. One also remarks that $\varphi_1$ is strongly surjective by our argument in the proof of Theorem \ref{thm:inverse Bloch n=1} in the case $n=1$, and Remark \ref{remk:strong surj n=1}.
\qed
\end{defn}

\begin{lem}\label{lem:vertical surj}
Let $n \geq 1$ be an integer. Suppose that
\begin{equation}\label{eqn:simple inverse Bloch surj 1}
\varphi_n : \TZ^n (k, n;m) \to \CH_{{\rm v}} ^n (X/ (m+1), n)
\end{equation}
is strongly surjective. Then the broken arrow $G_{n,k}$ in \eqref{eqn:J comm 0} is surjective. 
\end{lem}

\begin{proof}
The two bottom maps $\boxtimes$ and ${\rm can}$ in \eqref{eqn:J comm 0} induce the homomorphism
\begin{equation}\label{eqn:s 1 1}
z_{{\rm v}} ^n (X, n) \otimes z^1 _{{\rm d}} (X, 1) \twoheadrightarrow \CH_{{\rm v}} ^n (X/ (m+1), n)  \otimes \CH_{{\rm d}} ^1 (X/ (m+1), 1)
\end{equation}
$$
 \to \CH_{{\rm v}} ^{n+1} ( X/ (m+1), n+1),
 $$
and under the identifications
$$
\CH_{{\rm d}} ^i (X/ (m+1), i) \simeq \widehat{K}_i ^m (k_{m+1})
$$
and
$$
\CH_{{\rm v}} ^i (X/ (m+1), i) \simeq \widehat{K}_i ^m (k_{m+1}, (t)),
$$
of Theorem \ref{thm:graph final}, the second arrow \eqref{eqn:s 1 1} is identical to $\gamma_{m,n}$ in Lemma \ref{lem:gamma m n}. 

\medskip

To show that the map $G_{n,k}$ in \eqref{eqn:J comm 0} is surjective, it is enough to show that for each symbol in $\widehat{K}_{n+1} ^M (k_{m+1}, (t))$ of the following form with $(n+1)$ entries
\begin{equation}\label{eqn:alpha}
\{ \alpha_1, \alpha_2, \cdots, \alpha_{n+1}\},
\end{equation}
where $\alpha_i \in k_{m+1} ^{\times}$ with $\alpha_1 \equiv 1 \mod t$, we can find a member of $\TZ^n (k, n;m) \otimes z_{{\rm d, nc}} ^1 (X, 1)$ that maps to the Milnor symbol.

\medskip

Using Lemma \ref{lem:KS improved}, we may assume that the given $\alpha_1, \cdots, \alpha_{n+1} \in k_{m+1} ^{\times}$ have irreducible polynomial representatives $p_1 (t), \cdots, p_n (t), p_{n+1} (t)$ satisfying the requirements there. Since we are given that $\varphi_n$ is strongly surjective, for the Milnor symbol $\{ \alpha_1, \cdots, \alpha_n \}$ with the first $n$ entries and their chosen irreducible polynomial representatives $p_1 (t), \cdots, p_n (t)$, we can find a $0$-cycle $C \in \TZ^n (k, n;m)$ with a single $\mathbb{A}^1$-coordinate such that $\varphi_n (C) = \{ \alpha_1, \cdots, \alpha_n\}$. The remaining issue is to find a suitable $Z \in z_{{\rm d, nc}} ^1 (X, 1)$ that takes care of the remaining $\alpha_{n+1}$ and $p_{n+1} (t)$.

\medskip

Recall that we supposed that the chosen representative $p_1(t)$ of $\alpha_1$ is of the form $p_1 (t) = 1- \frac{t^i}{a}\in k[t]$ for some $a \in k^{\times}$, and it is assumed to be irreducible. Let $\zeta$ be an $i$-th root of $a$. Let $\{ \zeta_1, \cdots, \zeta_i\}$ be the set of all conjugates in an algebraic closure of $k$. Let $k' \supset k$ be a finite extension that contains $\zeta$. Choose a sufficiently large integer $d$ such that $(d-1) i > \max \{ m, \deg p_{n+1}, i \} $. For some unspecified members $u_1, \cdots, u_i \in k$, define
$$ 
\tilde{p}_{n+1} (t) := p_{n+1} (t) + u_i t ^{ (d-1)i} + u_{i-1} t^{ (d-1)i + 1} + \cdots + u_1 t^{ di - 1} +  (-1)^{di} ( \zeta t)^{di} 
$$
$$
= p_{n+1} (t) +  \sum_{j=1} ^i u_j t^{di-j}  + (-1)^{di} ( \zeta t)^{di}.
$$
Regardless of the choices of $u_1, \cdots, u_i \in k$, since $\zeta^i = a \in k$, we have $\tilde{p}_{n+1} (t) \in k[t]$.  Since $\tilde{p}_{n+1} (t) \equiv p_{n+1} (t) \mod t^{m+1}$, it also represents $\alpha_{n+1}$. We make the following elementary observation:

\medskip

\textbf{Observation :} {There exist $u_1, \cdots, u_i \in k$ such that $\tilde{p}_{n+1} (\zeta_j^{-1} ) \not =0$ for all $j$. In particular, $\tilde{p}_{n+1} (t) = 0$ does not have $\zeta^{-1} $ as a solution in any extension.}

\medskip

Choose such $u_1, \cdots, u_i \in k$. This polynomial $\tilde{p}_{n+1} (t)$ may not be irreducible in general. Let $\tilde{p}_{n+1} (t) = \prod_{\ell=1} ^r p_{n+1, \ell} (t)$ be a factorization into irreducible polynomials in $k[t]$. By construction, for all $1 \leq \ell \leq r$, we have $p_{n+1, \ell } (\zeta^{-1} )\not = 0$. Then for a finite extension $k''$ that contains $k'$ and all the roots of $\tilde{p}_{n+1}$, we can factorize $\tilde{p}_{n+1} (t)$ into a product of linear polynomials over $k''$
$$
\tilde{p}_{n+1} (t) = \prod_{j=1} ^{di}  (c_j - \zeta t),
$$
where $c_j / \zeta $ runs over all the roots of $\tilde{p}_{n+1} (t)$, for some $c_j \in (k'')^{\times}$. Here we have $c_j \not = 1$ for all $j$ for otherwise $\tilde{p}_{n+1} (t)$ has $\zeta^{-1}$ as a root, contradicting the choices of $u_1, \cdots u_i$ in the Observation.

After relabeling them if necessary, we may assume that $c_1/ \zeta, \cdots, c_r / \zeta$ are the solutions of the distinct irreducible factors $p_{n+1, 1} (t), \cdots, p_{n+1, r} (t)$ of $\tilde{p}_{n+1} (t)$, respectively.

Let $\pi: \Spec (k'') \to \Spec (k)$ be the corresponding morphism. Here, $\varphi_n (C) = \{ \alpha_1, \cdots, \alpha_n \} = \pi_* \left\{ 1 - \frac{t}{\zeta}, \alpha_2, \cdots, \alpha_n \right\}$. Take $Z= \sum_{\ell=1} ^r \pi_* \{ y= c_{\ell} - t \}$. Then
\begin{eqnarray*}
G_{n,k} (C \otimes Z) &=& \pi_*\left( \left\{ 1 - \frac{t}{\zeta}, \alpha_2 \cdots, \alpha_n \right\} \boxtimes \sum_{\ell=1} ^r \tau_{\zeta} ^*\{ y_{n+1} = c_{\ell} - t \} \right) \\
&=& \pi_*\left( \left\{ 1 - \frac{t}{\zeta}, \alpha_2 \cdots, \alpha_n \right\} \boxtimes \sum_{\ell=1 } ^r \{ y_{n+1} = c_{\ell} - \zeta t \} \right) \\
&=& \pi_* \left( \left\{ 1 - \frac{t}{\zeta}, \alpha_2 \cdots, \alpha_n \right\} \right) \boxtimes \sum_{\ell=1} ^r \pi_*  \{ y_{n+1} =  c_{\ell} - \zeta t \}\\
&=& \{ \alpha_1, \cdots, \alpha_n \} \boxtimes \sum_{\ell=1} ^r \{ y_{n+1} = p_{n+1, \ell} (t) \mod t^{m+1} \} \\
&=& \{ \alpha_1, \cdots, \alpha_n \} \boxtimes \{ y_{n+1} = \prod_{\ell=1} ^r p_{n+1, \ell} (t) \mod t^{m+1} \} \\
&=&\{ \alpha_1, \cdots, \alpha_n \}  \boxtimes \{ y_{n+1} = \tilde{p}_{n+1} (t) \mod t^{m+1} \} \\
&=&\{ \alpha_1, \cdots, \alpha_n \}  \boxtimes \{y_{n+1} = \alpha_{n+1} \} = \{ \alpha_1, \cdots, \alpha_n , \alpha_{n+1} \}.
\end{eqnarray*}
This proves the desired surjectivity.
\end{proof}

\begin{lem}\label{lem:strong surj}
If $\varphi_n$ is strongly surjective, then so is $\varphi_{n+1}$.
\end{lem}

\begin{proof}
For each $\alpha:=\{ \alpha_1, \cdots, \alpha_{n+1}\}$ with a chosen set of irreducible polynomial representatives $p_1 (t), \cdots, p_n (t), p_{n+1} (t)$, in the proof of Lemma \ref{lem:vertical surj}, we saw that we can find a $0$-cycle $C\in \TZ^n (k, n;m)$ with a single $\mathbb{A}^1$-coordinate and a cycle $Z \in z_{{\rm d, nc}} ^1 (X, 1)$ such that $G_{n,k} ( C \otimes Z) = \alpha$. 

By the commutativity of the diagram \eqref{eqn:J comm 0}, we have $\varphi_{n+1} (J_{n,k} ( C \otimes Z)) = \alpha$. However, by construction $J_{n,k} ( C \otimes Z))$ is a $0$-cycle with a single $\mathbb{A}^1$-coordinate. Hence $\varphi_{n+1}$ is strongly surjective.
\end{proof}

We finally get to:

\begin{prop}\label{prop:inverse Bloch surj}
For each integer $n \geq 1$, the inverse Bloch map $\varphi_n$ in \eqref{eqn:inverse Bloch cy} is strongly surjective. Hence it is surjective.
\end{prop}

\begin{proof}
When $n=1$, we already know that $\varphi_1$ is strongly surjective by the argument of the proof of Theorem \ref{thm:inverse Bloch n=1} (see also Remark \ref{remk:strong surj n=1}).

\medskip

Suppose $\varphi_n$ is strongly surjective for some $n \geq 1$. By Lemma \ref{lem:strong surj}, the map $\varphi_{n+1}$ is also strongly surjective. Hence the proposition holds by induction. 
\end{proof}

\subsection{Proof of the main theorem}

\begin{cor}\label{cor:inverse Bloch surj 3}
For all $ m,n \geq 1$, the inverse Bloch map $\varphi$ of \eqref{eqn:inverse Bloch cy} is an isomorphism. 
\end{cor}

\begin{proof}
It is injective by Lemma \ref{lem:inverse Bloch inj} and surjective by Proposition \ref{prop:inverse Bloch surj}.
\end{proof}

Theorem \ref{thm:intro main vanishing} now follows. We rephrase the essence in the following form:

\begin{thm}\label{thm:main vanishing}
Let $k$ be an arbitrary field. Let $m, n \geq 1$ be integers. Then we have an isomorphism
$$
\mathbb{W}_m \Omega_k ^{n-1} \simeq \widehat{K}_n ^M (k_{m+1}, (t)).
$$

In particular, we have the a split short exact sequence
$$ 
0 \to \mathbb{W}_m \Omega_k ^{n-1} \to \widehat{K}_n ^M (k_{m+1}) \to K_n ^M (k) \to 0.
$$
\end{thm}

\begin{proof}
We have a sequence of isomorphisms
$$
\mathbb{W}_m \Omega_k ^{n-1} \simeq^{\dagger} \TCH^n (k, n;m)' \simeq^{\ddagger} \CH_{{\rm v}} ^n (X/ (m+1), n) \simeq ^{ (1)}  \widehat{K}_n ^M (k_{m+1}, (t)),
$$
where $\dagger$ is the isomorphism $gr_k$ of R\"ulling \cite{R} (Theorem \ref{thm:Rulling2}), $\ddagger$ is the isomorphism $\varphi$ of Corollary \ref{cor:inverse Bloch surj 3} with $X= \Spec (k[[t]])$, and $\simeq^{(1)}$ is the isomorphism of Theorem \ref{thm:graph final}.

The second part of the theorem follows from the first part together with the definition of $\widehat{K}_n ^M (k_{m+1}, (t))= \ker ({\rm ev}_{t=0}: \widehat{K}_n ^M (k_{m+1}) \to K_n ^M (k))$. This gives a split short exact sequence.
\end{proof}

Here is an algebraic analogue of the deconcatenation map:

\begin{cor}
Let $k$ be a field, and let $m, n \geq 1$ be integers. Then we have the deconcatenation map on the big de Rham-Witt forms
$$
{\rm Dec} : \mathbb{W}_m \Omega_k ^{n-1} \to \mathbb{W}_m (k) \otimes_{\mathbb{Z}} K_{n-1} ^M (k).
$$
\end{cor}

\begin{proof}
It follows from Corollary \ref{cor:decon mod m} and Theorem \ref{thm:main vanishing}.
\end{proof}

Recall from Theorem \ref{thm:KPCM7.4} that we have the concatenation map
$$
{\rm Con}: \mathbb{W}_m (k) \otimes_{\mathbb{Z}} K_{n-1} ^M (k) \to \mathbb{W}_m \Omega_k ^{n-1}
$$
in the opposite direction. Of course, these maps ${\rm Con}$ and ${\rm Dec}$ are far from being inverse to each other.

\subsection{An example and a remark}

\begin{exm}
The arguments given in the proofs of Lemmas \ref{lem:KS improved} and \ref{lem:vertical surj} offer some hints on how one may construct concrete cycles that map to given Milnor symbols in $\widehat{K}_n ^M (k_{m+1}, (t))$. As an illustration, let us consider an example when $n=2$ and $m=2$. Here, by Proposition \ref{prop:Witt elements} and Lemma \ref{lem:KS}, the group $\widehat{K}_2 ^M (k_3, (t))$ is generated by the symbols of the form
$$
 \{ 1- a t, c\}, \ \{ 1- a t^2, c \}, \ \{ 1- a t, 1- bt \}, \ \{ 1- a t, 1- bt^2 \}, \ \{ 1- at^2, 1- b t^2 \},
 $$
for some $a, b, c \in k ^{\times}$, with $c \not = 1$. 

Let's consider $\{ 1- at, c \} = \{ 1 - at , c - a^3 t^3\}$, where $c \equiv c- a^3 t^3 \mod t^3$. If there is no $d \in k^{\times}$ such that $c= d^3$, then $c-a^3 t^3$ is irreducible. Consider the closed point $\mathfrak{p} =\left( \frac{1}{a}, \gamma \right)$ in $\mathbb{A}^1 \times \square^1$, where $\gamma \in \square^1$ is a closed point whose irreducible polynomial is $c- t^3$. Then, $\varphi (\mathfrak{p}) = \{ 1- at, c- a ^3 t^3 \} = \{ 1- at, c\}. $

If $c = d^3$ for some $d \in k^{\times}$, then 
$$
\{ 1- at, c- a^3 t^3 \} = \{ 1 - at, (d- a t) (d^2 + d a t + a^2 t^2) \}
$$
$$
 = \{ 1-at, d- a t \} + \{ 1-at, d^2 + d a t + a^2 t^2 \},
 $$
and the first symbol $\{ 1- at, d-at \}$ is $\varphi \left( ( \frac{1}{a}, d) \right)$. Depending on whether $d^2 + dat +a^2 t^2$ factorizes over $k$ or not (the latter is far more likely in general unless $k= \bar{k}$), we can also express the second symbol either as the sum of the $\varphi$ of two $k$-rational points of $\mathbb{A}^1 \times \square^1$, or the $\varphi$ of a degree $2$ closed point.

Other examples of the above generators can be computed similarly, while we may need to take into account that $\{ \alpha , -1 \} = 0$ in the Milnor $K$-groups.
\qed
\end{exm}

\begin{remk}\label{remk:GT comparison}
Those who are familiar with some relevant earlier works in the literature, such as Gorchinskiy-Tyurin \cite{GT}, may recall that a few of them tried to transform the generating Milnor symbols in $\widehat{K}_n ^M (k_{m+1}, (t))$ into sums of symbols of the form 
$$ 
N \{ 1- a t^i, c_1, \cdots, c_{n-1}\},
$$
for which one can prove that they belong to the image of the Bloch map, where $N \in \mathbb{N}$ are some large integers and $ a, c_1, \cdots, c_{n-1} \in k ^{\times}$. Then, they argued that using an extra assumption that a large number of integers, including $N$, are invertible in $k$, they were able to deduce that all symbols were in the image of the Bloch map.

As the reader may notice, this article actually does the \emph{opposite} of what those in the literature tried: instead of transforming the generating symbols into large integer multiples of symbols whose last $(n-1)$ entries are constant in the base field $k$, we represent them by \emph{non-constant} polynomials in $t$ of the degrees possibly $\geq m+1$. This eventually allowed us to give certain geometric descriptions in terms of suitable closed points of degrees $d \geq 1$, defined by their irreducible factors. This is how the Milnor symbols are related to the geometry of $0$-cycles.
\qed
\end{remk}

\section{Addendum: another potential approach}\label{sec:addendum}

On the main question of this article, those who are familiar with the $p$-typical de Rham-Witt forms as well as the article of Bloch-Kato \cite{BK} and R\"ulling-Saito \cite{RS} might have tried the following alternative approach, somewhat more pertinent to the usual de Rham-Witt business: to the system of the relative Milnor $K$-groups $\{ \widehat{K}^M_n (k_{m+1}, (t)) \}_{m, n \geq 1 }$, give certain extra structures, namely the restrictions $\{ R : \widehat{K}^M_n (k_{m+1}, (t)) \to \widehat{K}^M_n (k_m, (t)) \}_{m \geq 2}$, the Verschiebung maps $\{ V_r \}_{r \geq 1}$ and the Frobenius maps $\{ F_r \}_{r \geq 1}$ as well as the differentials $\{ d: \widehat{K}^M_n (k_{m+1}, (t)) \to \widehat{K}_{n+1} ^M (k_{m+1}, (t)) \}_{m, n \geq 1}$, and prove that they satisfy the axioms of Witt-complexes, compatible with the corresponding operators of the big de Rham-Witt forms. 

It was the initial approach followed by the author, although he eventually chose to shift the track to the present version, which was conceptually and technically simpler and more geometric. Those who wish to insist may still follow the above Witt-friendlier method: after defining the extra structures, they may encounter roughly two major obstacles: (i) to show that the level-wise homomorphisms exist (which requires the reader to check all the axioms of the Witt-complex structure) and (ii) to show that these are all isomorphisms. Here is a possible road map for those who might want to try this path assuming (i). Of course, the uninterested reader can simply skip \S \ref{sec:addendum} entirely.
\medskip

\subsection{Inverse Cartier operators}

Recall some basics on inverse Cartier operators. 
Let $R$ be a commutative ring with unity of characteristic $p>0$ for a prime $p$. Consider the absolute de Rham complex $\Omega_R ^{\bullet} = \Omega_{R/\mathbb{Z}} ^{\bullet}$ of $R$ over $\mathbb{Z}$, where we let $\Omega_R ^0 = R$. This is a complex of abelian groups, but not that of $R$-modules. We will see below that it becomes so for a different $R$-module structure, which is a kind of Frobenius twisted $R$-action.

\medskip

Let $M$ be an $R$-module in the usual sense. Its Frobenius twisted $R$-module structure (call it $R^{(p)}$-module) on $M$ is defined by
$$
 R \times M \overset{{\rm Fr}\times {\rm Id}_M}{\longrightarrow} R \times M \longrightarrow M,
 $$
which maps $(r,x)$ to $r^p x$, where $r \mapsto r^p$ is the Frobenius ring homomorphism. Let's denote this action by $r \cdot x$.

Using the above, regard each $\Omega_R ^i$ as an $R^{(p)}$-module. One checks directly that $d: \Omega_R ^i \to \Omega_R ^{i+1}$ is an $R^{(p)}$-module homomorphism, although it is not an $R$-module homomorphism. We let
$$ 
Z^i:= \ker ( d: \Omega_R ^i \to \Omega_R ^{i+1}), \ \ B^i:= {\rm im} (d: \Omega_R ^{i-1} \to \Omega_R ^i),
$$
so that $H^i = H^i (\Omega_R ^{\bullet}) = Z^i / B^i$ by definition. All of them are $R^{(p)}$-modules. For each $r \in R$, define $C^{-1} (r):= r^{p-1} dr \in \Omega_R^1$. Using that $dr \wedge dr = 0$, one checks that $C^{-1} (r) \in Z^1$. It is straightforward to check:

\begin{lem}\label{lem:inverse Cartier}
Let $r_1, r_2 \in R$. Then
\begin{enumerate}
\item $ C^{-1} (r_1 + r_2) - C^{-1} (r_1) - C^{-1} (r_2) = d P(r_1, r_2)$ for a polynomial $P (X, Y) \in \mathbb{Z}[X, Y]$. More specifically, we take
$$
P (X,Y) := \sum_{i=1} ^{p-1} \frac{1}{p} \begin{pmatrix} p \\ i \end{pmatrix} X^{p-i} Y ^i,
$$
where each $ \begin{pmatrix} p \\ i \end{pmatrix}$ is divisible by $p$ so that $ \frac{1}{p} \begin{pmatrix} p \\ i \end{pmatrix} \in \mathbb{Z}$.
\item $C^{-1} (r_1 r_2) = r_2 \cdot C^{-1} (r_1) + r_1 \cdot  C^{-1} (r_2).$
\end{enumerate}
\end{lem}
Lemma \ref{lem:inverse Cartier} shows that $C^{-1}: R \to H^1$ is additive and a derivation on $R$ over $\mathbb{Z}$, where $H^1$ is given the $R^{(p)}$-module structure. In particular, by the universal property of the K\"ahler differentials, there exists a unique $R$-module homomorphism $\overline{C^{-1}} : \Omega_R ^1 \to H^1$ (where $\Omega_R ^1$ is the usual $R$-module and $H^1$ is the $R^{(p)}$-module) such that
$$
\xymatrix{
R \ar[rr] ^{C^{-1}} \ar[dr] _d  & & H^1 \\
& \Omega_R ^1 \ar[ru] _{\overline{C^{-1}} } &}
$$
commutes. For simplicity, we denote $\overline{C^{-1}}$ also by $C^{-1}$, and call it the inverse Cartier operator. We can immediately extend it to a group homomorphism 
$$ 
C^{-1}: \Omega_R ^i \to H^i (\Omega_R ^{\bullet}) = Z^i / B^i,
$$
for all $i \geq 1$, still denoted by the same symbol, and also called the inverse Cartier operator. It is known to be injective. So, after letting $Z_1 ^i := Z^i$ and $B_1 ^i:= B^i$, we can define a chain of subgroups (see Bloch-Kato \cite[(1.3)]{BK})
$$
0=:B_0 ^i  \subset B_1 ^i \subset \cdots \subset B_n ^i  \subset \cdots \subset Z_n^i \subset \cdots \subset Z_1 ^i \subset Z_0 ^i:= \Omega_R^i,$$
such that inductively we have $C^{-1} (Z_n ^i) = Z_{n+1} ^i / B_1 ^i$ and $C^{-1} (B_n ^i) = B_{n+1} ^i / B_1 ^i$. We can iterate the inverse Cartier operators to define for all $s \geq 1$ the composite
$$
C^{-s}: \Omega_R ^i \to Z_1 ^i/ B_1 ^i \to Z_2 ^i / B_2 ^i \to \cdots \to Z_s ^i / B_s ^i \subset \Omega_R ^i / B_s ^i.
$$
 For consistency, let $C^{-0}:= {\rm Id}$.

\medskip

Now let $q , m \geq 1$ be integers and write $m = m'  p^s$ for some $s \geq 0$ such that $(m',p)= 1$. Define (R\"ulling-Saito \cite[4.1.3]{RS}, cf. Bloch-Kato \cite[(4.7)]{BK})
$$
{\rm gr}_m ^q (R):= {\rm coker} \left( \theta: \Omega_R ^{q-1} \to \frac{ \Omega_R ^q}{B_s ^q} \oplus \frac{ \Omega_R ^{q-1}}{B_s ^{q-1}} \right),
$$
where $\theta (\alpha) = \left( C^{-s} (d \alpha), (-1)^{q-1} m' C^{-s} (\alpha) \right).$

\medskip

\subsection{A theorem of R\"ulling-Saito}

Recall the following result from R\"ulling-Saito \cite[Remark 4.2, Proposition 4.3]{RS}, part of which goes back to L. Illusie \cite[I. Corollary 3.9]{Illusie}, Bloch-Kato \cite[(1.3), (4.7)]{BK}, and Hyodo-Kato \cite[Theorem 4.4]{HK}. In \cite{RS}, the authors say a field $k$ is perfect at the beginning of \cite[\S 4.1.3]{RS}, but this assumption is superfluous:

\begin{thm}[R\"ulling-Saito]\label{thm:RS induction}
Let $k$ be a field of characteristic $p>0$. Let $q \geq 1, m \geq 2$ be integers. Let $R$ be an essentially smooth $k$-algebra.
Write $m= m' p^s$ for some $s\geq 0$ with $(m', p) = 1$. Then there exists a short exact sequence of groups
\begin{equation}\label{eqn:RS induction p}
0 \to {\rm gr}_m ^q (R) \to \mathbb{W} _m \Omega_R ^q \to \mathbb{W}_{m-1}  \Omega_R ^{q} \to 0,
\end{equation}
where the left arrow is induced from the map
$$
 \Omega_R ^q \oplus \Omega_R ^{q-1} \to \mathbb{W}_m \Omega_R ^q
$$
that sends $(\alpha, \beta)$ to $V_m (\alpha) + (-1) ^q d V_m (\beta)$.

In particular, we have an equivalent short exact sequence
\begin{equation}\label{eqn:RS induction p 2}
0 \to V_m \mathbb{W}_1 \Omega_R ^{q} + d V_m \mathbb{W}_1 \Omega_R ^{q-1} \to \mathbb{W} _m \Omega_R ^q \to \mathbb{W}_{m-1}  \Omega_R ^{q} \to 0,
\end{equation}
where the first arrow is the inclusion, where one notes $\mathbb{W}_1 \Omega_R ^q = \Omega_R ^q$.
\end{thm}

\subsection{A conditional double induction argument}

Theorem \ref{thm:RS induction} suggests an induction argument to be given in Lemma \ref{lem:inverse Bloch surj}, provided that one could prove the following rather nontrivial assertion:

\medskip

\textbf{Assumption ($\star$):} \emph{Let $k$ be a field of characteristic $p>0$. Let $m \geq 1, n \geq 2$ be integers. Suppose there is a Witt-complex structure on $\{ \widehat{K}_n ^M (k_{m+1}, (t)) \}$, equipped with $\{ R, F_r, V_r, d, R\}$. In particular, we have the maps $\varphi_{m,n} : \mathbb{W}_m \Omega_k ^{n-1} \to \widehat{K}^M_n (k_{m+1}, (t))$. 
Furthermore, suppose that inside the group $\widehat{K}_n ^M (k_{m+1}, (t))$, we can express the subgroup $\widehat{K}_n ^M (k_{m+1}, (t^m))$ as
$$
\widehat{K}_n ^M (k_{m+1}, (t^m)) = V_m \widehat{K}_n ^M (k_2, (t)) + d V_m \widehat{K} _{n-1} ^M (k_2, (t)).
$$}

\bigskip

\begin{cor}\label{cor:RS induction}
Let $k$ be a field of characteristic $p>0$. Let $n , m\geq 2$ be integers. Suppose {\rm \textbf{Assumption $(\star)$}} holds. Then the diagram
\begin{small}
$$
\xymatrix{
0 \ar[r] & { \begin{matrix} V_m \Omega_k ^{n-1} \\ + \delta V _m \Omega_k ^{n-2} \end{matrix}}  \ar[d] ^{\alpha_{m, n} } \ar[r] & \mathbb{W}_m \Omega_k ^{n-1} \ar[d] ^{\varphi_{m, n}} \ar[r] & \mathbb{W}_{m-1} \Omega_k ^{n-1} \ar[d] ^{\varphi_{m-1, n}} \ar[r] & 0 \\
0 \ar[r] & { \begin{matrix} V_m \widehat{K}_n ^M (k_2, (t)) \\ + d V_m \widehat{K} _{n-1} ^M (k_2, (t)) \end{matrix}}  \ar[r] & \widehat{K}_n ^M (k_{m+1}, (t)) \ar[r] & \widehat{K}_n ^M (k_m, (t)) \ar[r] & 0}
$$
\end{small}
commutes, where $\alpha_{m,n}  (V_m x + d V_m y) = V_m \varphi_{1,n} (x) + d V_m \varphi_{1, n-1} (y)$ for $x\in \Omega_k ^{n-1} $ and $y \in \Omega_k ^{n-2}$.

\end{cor}

\begin{lem}\label{lem:inverse Bloch surj}
Suppose {\rm \textbf{Assumption ($\star$)}} holds, with $p > 2$. Then the maps $\varphi_{m,n} : \mathbb{W}_m \Omega_k ^{n-1} \to \widehat{K}_n ^M (k_{m+1}, (t))$ are isomorphisms for all pairs $(m,n) \in \mathbb{N}^2$.
\end{lem}

\begin{proof}
We will use a double induction argument. 

When $n=1$, we already know algebraically that $\varphi_{m,1}$ are isomorphisms unconditionally for all $ m \geq 1$. 

When $m=1$, by Gorchinskiy-Osipov \cite{GO} we know that $\varphi_{1, n}$ are isomorphisms for all $n \geq 1$, except for the case $p=2$ (so we supposed $p>2$).

Now we move on to the double induction argument. Suppose $m, n \geq 2$, and suppose that $\varphi_{m', n'}$ are isomorphisms for all pairs $(m', n')$ with $1 \leq m' \leq m$ and $1 \leq n' \leq n$, except $(m', n') = (m,n)$. Under these assumptions, in the diagram of Corollary \ref{cor:RS induction}, the left vertical map $\alpha_{m,n}$ and the right vertical map $\varphi_{m-1, n}$ are both isomorphisms. The snake lemma now implies then that $\varphi_{m,n}$ is an isomorphism as well.
\end{proof}

The author does not have a good argument for {\rm \textbf{Assumption ($\star$)}} without resorting to the cycle-theoretic results proven in this article.

\section{Applications}\label{sec:applications}

We leave some applications, questions, and remarks.

\subsection{Logarithm and logarithmic derivative over $k_{m+1}$}\label{sec:dlog}
One immediate application of the identification $\mathbb{W}_m \Omega_k ^{n-1} \simeq \widehat{K}_n ^M (k_{m+1}, (t))$ stated in Theorem \ref{thm:main vanishing}, is the split short exact sequence for arbitrary field $k$
$$
0 \to \mathbb{W}_m \Omega_k ^{n-1} \to \widehat{K}_n ^M (k_{m+1})) \to K_n ^M (k) \to 0.
$$
Using this, we have another application regarding the logarithm, especially including the case when ${\rm char} (k) > 0$.

The isomorphism $\mathbb{W}_m\Omega_k ^{n-1} \overset{\simeq}{\rightarrow} \widehat{K}_n ^M (k_{m+1}, (t))$ of Theorem \ref{thm:main vanishing} turns the addition of the de Rham-Witt forms into the summation of Milnor symbols, which comes from the multiplication of members of $k_{m+1}^{\times}$. So, we regard the inclusion $\mathbb{W}_m \Omega_k ^{n-1} \hookrightarrow \widehat{K}_n ^M (k_{m+1})$ given by the inverse Bloch map as a sort of exponential map, and write it as $ \exp_n$, the $n$-th exponential. Its image is precisely $\widehat{K}_n ^M (k_{m+1}, (t))$. So, its inverse is:

\begin{defn}[Bloch map as the $n$-th logarithm]\label{defn:log log^0}
 Define the \emph{$n$-th logarithm}, denoted by
$$
\log= \log_n: \widehat{K}_n ^M (k_{m+1}, (t)) \to \mathbb{W}_m \Omega_k ^{n-1}
$$
to be the inverse of $\exp_n$, i.e. the Bloch map $B= \varphi^{-1}$. On the whole group ${K}_n ^M (k_{m+1})$, using the decomposition
\begin{equation}\label{eqn:K Milnor decomp}
\widehat{K}_n ^M (k_{m+1}) = \widehat{K}_n ^M (k_{m+1}, (t)) \oplus K_n ^M (k),
\end{equation}
we define the \emph{normalized $n$-th logarithm} to be
$$
\log^0= \log_n ^0 :  {K}_n ^M (k_{m+1})  \to^{\dagger} \widehat{K}_n ^M (k_{m+1}) \overset{pr_{rel}}{\longrightarrow} \widehat{K}_n ^M (k_{m+1}, (t) )   \overset{ \log_n}{\longrightarrow} \mathbb{W}_m \Omega_k ^{n-1},
$$
where $\dagger$ is the natural surjection in Theorem \ref{thm:Khat_univ} and $pr_{rel}$ is the projection to the first component in \eqref{eqn:K Milnor decomp}.\qed
\end{defn}

\begin{remk}
For instance when $n=1$, we have $k[[t]]^{\times} =  (1 + tk[[t]])^{\times} \oplus k^{\times} $, and 
$$
\log^0: k[[t]]^{\times} \to \mathbb{W} (k)= (1+ tk[[t]])^{\times}
$$
is given by sending a formal power series $f \in k[[t]]^{\times}$ to its scaling $f/ f(0)$. 

In case ${\rm char} (k) = 0$, we can further consider the usual formal logarithm
$$
 \Log : (1+ tk[[t]])^{\times} \overset{\sim}{ \to} tk[[t]]
 $$
by writing 
\begin{equation}\label{eqn:Log}
\Log(1 + t g) = \sum_{i=1} ^{\infty} \frac{1}{i} (-tg)^i= -tg + \frac{1}{2} (-tg)^2 + \frac{1}{3} (-tg)^3 + \cdots,
\end{equation}
and define $\Log^0 (f):= \Log (f / f(0))$ for $f \in k[[t]] ^{\times}$. This $\Log^0$ coincides with the function $\log^0$ considered in S. \"Unver \cite{Unver JAG} in ${\rm char} (k) = 0$.

The $\log^0$ defined in Definition \ref{defn:log log^0} is a natural generalization to all $n\geq 1$ and all characteristic, of the above $\Log^0$ as well as $\log^0$ of S. \"Unver. \qed
\end{remk}

Recall that we had the $d\log$ map (see Krishna-Park \cite[Corollary 7.5]{KP crys}, cf. Geisser-Hesselholt \cite[Proposition B.1.1]{GH})
\begin{equation}\label{eqn:dlog int}
 d\log= d\log_k : K_n ^M (k) \to \mathbb{W}_m \Omega_k ^n.
 \end{equation}
which sends the Milnor symbol $\{ a_1, \cdots, a_n \}$ for $a_i \in k^{\times}$ to $\frac{ d [a_1]}{[a_1]} \wedge \cdots \wedge \frac{ d [a_n]}{[a_n]}$, where $[c]$ for $c \in k^{\times}$ is the Teichm\"uller lift of $c$ in $\mathbb{W}_m (k)$. 

\medskip

Let $ {\rm ev}_{t=0}: K_n ^M (k_{m+1}) \to K_n ^M (k)$ denote the reduction mod $t$. We can now generalize the $d\log_k$ in \eqref{eqn:dlog int} to:

\begin{defn}[Logarithmic derivative] 
Define the \emph{extended $n$-th $d\log$ of modulus $m \geq 1$}
$$
d \log_t: K_n ^M (k_{m+1}) \to \mathbb{W}_m \Omega_k ^n
$$
by the formula
\begin{equation}\label{eqn:dlog_t}
d \log_t (\alpha) := d \left( \log^0 _n (\alpha) \right) + d \log_k ( {\rm ev}_{t=0} (\alpha))
\end{equation}
 for $\alpha \in K_n ^M (k_{m+1})$, where $\log^0_n$ is the normalized logarithm in Definition \ref{defn:log log^0}, and $d$ is the exterior derivation
 $$
 d: \mathbb{W}_{m} \Omega_k ^{n-1} \to \mathbb{W}_m \Omega_k ^n
 $$
 of the big de Rham-Witt forms, while $d\log_k$ is from \eqref{eqn:dlog int}.
\qed
\end{defn}

Notice that the map $\log_n$ is not defined on the whole group $K_n ^M (k_{m+1})$, but on the subgroup $K_n ^M (k_{m+1}, (t))$. Restricted on this subgroup, the map  $d\log_t$ of \eqref{eqn:dlog_t} is indeed related to $d$ of $\log_n$, i.e. the $d$ of the Bloch map $B$:

\begin{cor}
For $\alpha \in K_n ^M (k_{m+1}, (t))$, we have
$$
d \left( \log_n (\alpha) \right) = d\log_t (\alpha).
$$
\end{cor}

\begin{proof}
Since $\alpha \in K_n ^M (k_{m+1}, (t))$, by definition we have $\log^0_n (\alpha) = \log _n(\alpha)$. On the other hand, ${\rm ev}_{t=0} (\alpha) $ is $0$ in $K_n ^M (k)$, thus $d\log_k ( {\rm ev}_{t=0} (\alpha) ) = 0$. Hence by the formula \eqref{eqn:dlog_t}, we have
$$
d \log_t (\alpha)=  d \left( \log _n(\alpha) \right) + 0 = d \left( \log_n (\alpha) \right),
$$
as desired.
\end{proof}


\subsection{Some questions related to $d\log_t$}
Suppose ${\rm char} (k) = p>0$. Recall for instance from M. Morrow \cite[Theorems 1.2, 5.1]{Morrow ENS} that the $d\log_k$ in \eqref{eqn:dlog int} induces $d\log: \widehat{K}_n ^M (k) \to W_i \Omega_k ^n$ for $i \geq 1$ to the $p$-typical de Rham-Witt forms, which in turn induces $d\log_k: \widehat{K}^M_n (k) / p^i \to W_i \Omega_k ^n$, whose image is the logarithmic de Rham-Witt forms $W_i \Omega_{k, \log} ^n$. In \emph{loc.cit.}, it is proven for more general rings. Combined with Geisser-Levine \cite{GL}, we have a description of the $p$-adic motivic cohomology of smooth $k$-schemes in terms of the logarithmic de Rham-Witt forms.

Our new $n$-th $d\log_t$ of modulus $m \geq 1$ in \eqref{eqn:dlog_t} raises a few questions.

\begin{ques}
Is $d\log_t: \widehat{K}_n ^M (k_{m+1}) \to \mathbb{W}_m \Omega_k ^n$ surjective? 
\end{ques}

For instance R\"ulling-Saito \cite[Proposition 4.7]{RS} gives a surjective map
$$
 \mathbb{W}_m (k) \otimes (K_n ^M (k) \oplus K_{n-1} ^M (k)) \to \mathbb{W}_m\Omega_k ^n,
 $$
but it is not clear to the author if it answers the above question.

\begin{ques}
Suppose $m, i \geq 1$ are integers such that $\{ 1, p, \cdots, p^{i} \} \subset \{ 1, 2, \cdots, m \}$.
\begin{enumerate}
\item Is the induced map
\begin{equation}\label{eqn:dlog_t p}
d \log_t: \widehat{K}_n ^M (k_{m+1}) \to \mathbb{W}_m \Omega_k ^n \to W_i \Omega_k ^n
\end{equation}
surjective? Here the second map is the projection from the big de Rham-Witt forms to the $p$-typical de Rham-Witt forms, namely the $p$-typicalization map.

\item If the map in \eqref{eqn:dlog_t p} is not surjective, then what is the image of this map?

\end{enumerate}
\end{ques}

Here is one rather outlandish question, in the spirit of the $p$-adic motivic cohomology as seen from Geisser-Levine \cite{GL}:

\begin{ques}
Suppose $m, i \geq 1$ are integers such that $\{ 1, p, \cdots, p^{i} \} \subset \{ 1, 2, \cdots, m \}$.
\begin{enumerate}
\item Does the map \eqref{eqn:dlog_t p} induce a map
$$
\widehat{K}_n ^M (k_{m+1})/ p^i \to W_i \Omega_k ^n?
$$
\item Is it an isomorphism? If so, is it reasonable to define the notion of ``$p$-adic motivic cohomology" of the non-reduced scheme $\Spec (k_{m+1})$ using it?
\end{enumerate}
\end{ques}

\bigskip

\noindent\textbf{Acknowledgments.}  A good part of this work was done while the author was on his sabbatical leave at the Center for Complex Geometry (CCG) of the Institute for Basic Science of South Korea in the years 2022-2023, and it took a few more years to finish writing it up. The author thanks the Director Jun-Muk Hwang and the staff members of CCG for their warm hospitality and the peaceful working environment. 

The author thanks Ahmed Abbes, Bruno Kahn, Shane Kelly, Shuji Saito, Luc Illusie for kindly taking time to listen to him in various opportunities, and offering comments and / or suggestions. The author particularly feels obliged to acknowledge with gratitude that Matthew Morrow taught him that the main question had been open in ${\rm char} (k) = p>0$ for a long time.

This work was supported by Samsung Science and Technology Foundation under Project Number SSTF-BA2102-03.


\begin{thebibliography}{99}






\bibitem{BS} F. Binda and S. Saito, {\sl Relative cycles with moduli and regulator maps\/}, J. Inst. Math. Jussieu, \textbf{18}, (2019), no. 6, 1233--1293.


\bibitem{Bloch 1973} S. Bloch, {\sl On the tangent space to Quillen $K$-theory\/}, Lect. Notes in Math., \textbf{341}, Springer-Verlag, Berlin-New York, 1973, pp. 205--210.

\bibitem{Bloch 1975} S. Bloch, {\sl $K_2$ of Artinian $\mathbb{Q}$-algebras, with application to algebraic cycles\/}, Comm. Algebra, \textbf{3}, (1975), 405--428.

\bibitem{Bloch crys} S. Bloch, {\sl Algebraic $K$-theory and crystalline cohomology\/}, Publ. Math. l'Inst. Hautes \'Etudes Sci., \textbf{47}, (1977), 187--268.

\bibitem{Bloch HC} S. Bloch, {\sl Algebraic cycles and higher $K$-theory\/}, Adv. Math., \textbf{61}, (1986), no. 3, 267--304.

\bibitem{Bloch notes} S. Bloch, {\sl Some notes on elementary properties of higher chow groups, including functoriality properties and cubical chow groups\/}, available on Spencer Bloch's website at the University of Chicago.


\bibitem{BE2} S. Bloch and H. Esnault, {\sl The additive dilogarithm\/}, Doc. Math. \textbf{Extra Vol}. for Kazuya Kato's fiftieth birthday conference, (2003), 131--155.

\bibitem{BK} S. Bloch and K. Kato, {\sl $p$-adic \'etale cohomology\/}, Publ. Math. d'Inst. Hautes \'Etudes Sci., \textbf{63}, (1986), 107--152.



\bibitem{Dribus} B. F. Dribus, {\sl A Goodwillie-type theorem for Milnor $K$-theory\/}, preprint 2014, arXiv:1402.2222v1.



\bibitem{EM} E. Elmanto and M. Morrow, {\sl Motivic cohomology of equicharacteristic schemes\/}, arXiv:2309.08463v1 [math.KT], 15 Sep 2023.


\bibitem{GH} T. Geisser and L. Hesselholt, {\sl The de Rham-Witt complex and $p$-adic vanishing cycles,\/} J. Amer. Math. Soc., \textbf{19}, (2005), 1--36.


\bibitem{GL} T. Geisser and M. Levine, {\sl The $K$-theory of fields in characteristic $p$\/}, Invent. Math., \textbf{139}, (2000), no. 3, 459--493.



\bibitem{Goodwillie} T. Goodwillie, {\sl Relative algebraic $K$-theory and cyclic homology\/}, Ann. Math., \textbf{124}, (1986), no. 2, 347--402.

\bibitem{GO} S. O. Gorchinskiy and D. V. Osipov, {\sl Tangent space to Milnor $K$-groups of rings\/}, Proc. Steklov Inst. Math., \textbf{290}, (2015), 26--34.

\bibitem{GT} S. O. Gorchinskiy and D. N. Tyurin, {\sl Relative Milnor $K$-groups and differential forms of split nilpotent extensions\/}, Izvestiya Math., \textbf{82}, (2018), no. 5, 880--913.










\bibitem{Gupta-Krishna1} R. Gupta and A. Krishna, {\sl Zero-cycles with modulus and relative $K$-theory\/}, Ann. $K$-theory, \textbf{5}, (2020), no. 4, 757--819.

\bibitem{Gupta-Krishna2} R. Gupta and A. Krishna, {\sl Relative $K$-theory via $0$-cycles in finite characteristic\/}, Ann. $K$-theory, \textbf{6}, (2021), no. 4, 673--712.





\bibitem{HeMa} L. Hesselholt and I. Madsen, {\sl On the K-theory of nilpotent endomorphisms\/}, in Homotopy methods in Algebraic Topology (Boulder, CO, 1999), edited by Greenlees, J. P. C. et al., Contempo. Math. Amer. Math. Soc. Providence, RI, \textbf{271},  2001, pp. 127--140.


\bibitem{Hesselholt} L. Hesselholt, {\sl The big de Rham-Witt complex\/}, Acta Math., \textbf{214}, (2015), 135--207.

\bibitem{HK} O. Hyodo and K. Kato, {\sl Expos\'e V: Semi-stable reduction and crystalline cohomology with logarithmic poles, S\'eminaire Bourbaki\/}, Ast\'erisque, \textbf{223}, (1994), 221--268.


\bibitem{Illusie} L. Illusie, {\sl Complexe de de Rham-Witt et cohomologie cristalline\/}, Ann. Sci. \'Ec. Norm. Sup\'er. S\'er. 4, \textbf{12}, (1979), 501--661.

\bibitem{Iwasa-Kai} R. Iwasa and W. Kai, {\sl Chern classes with modulus\/}, Nagoya Math. J., \textbf{236}, (2019), 84--1332.

\bibitem{KMSY3} B. Kahn, H. Miyazaki, S. Saito and T. Yamazaki, {\sl Motives with modulus, III, The categories of motives\/}, Ann. $K$-theory, \textbf{7}, (2022), no. 1, 119-178.

\bibitem{KS} K. Kato and S. Saito, {\sl Global class field theory of arithmetic schemes\/}, in Applications of Algebraic $K$-theory to Algebraic Geometry and Number Theory, Proceedings of a Summer Research Conference, June 12--18, 1983, Eds. S. J. Bloch et. al., Contempo. Math., Amer. Math. Soc. Providence, RI, \textbf{55}, part I, pp. 255--331.



\bibitem{Kerz finite} M. Kerz, {\sl Milnor $K$-theory of local rings with finite residue fields\/}, J. Algebr. Geom., \textbf{19}, (2010), 173--191.

\bibitem{KM} J. Koizumi and H. Miyazaki, {\sl A motivic constructin of the de Rham-Witt complex\/}, J. Pure Appl. Algebra, \textbf{228}, (2024), 107602


\bibitem{KP moving} A. Krishna and J. Park, {\sl Moving lemma for additive higher Chow groups\/}, Algebra \& Number Theory, \textbf{6}, (2012), no. 2, 293--326.

\bibitem{KP Jussieu} A. Krishna and J. Park, {\sl Mixed motives over $k[t]/(t^{m+1})$\/}, J. Inst. Math. Jussieu, \textbf{11}, (2012), no. 3, 611--657.

\bibitem{KP DGA} A. Krishna and J. Park, {\sl DGA-structure on additive higher Chow groups\/}, Internat. Math. Res. Notices, Vol 2015 (2015), no. 1, 1--54.

\bibitem{KP DM} A. Krishna and J. Park, {\sl On additive higher Chow groups of affine schemes\/}, Doc. Math., \textbf{21}, (2016), 49--89.

\bibitem{KP sfs} A. Krishna and J. Park, {\sl A moving lemma for relative $0$-cycles\/}, Algebra \& Number Theory, \textbf{14}, (2020), no. 4, 991--1054.

\bibitem{KP crys} A. Krishna and J. Park, {\sl De Rham-Witt sheaves via algebraic cycles} (with an appendix by Kay R\"ulling), Compositio Math., \textbf{158}, (2021), no. 10, 2089--2132.





\bibitem{Milnor IM} J. Milnor, {\sl Algebraic $K$-theory and quadratic forms\/}, Invent. Math., \textbf{9}, (1970), 318--344.


\bibitem{Morrow} M. Morrow, {\sl $K_2$ of localizations of local rings\/}, J. Algebra, \textbf{399}, (2014), 190--204.

\bibitem{Morrow ENS} M. Morrow, {\sl $K$-theory and logarithmic Hodge-Witt sheaves of formal schemes in characteristic $p$,\/}, Ann. Sci. \'Ecole Norm. Sup\'er., \textbf{52}, (2019), 1537--1601.


\bibitem{NS} Yu. P. Nesterenko and A. A. Suslin, {\sl Homology of the general linear group over a local ring, and Milnor's $K$-theory\/}, (Russian) Izv. Akad. Nauk SSSR Ser. Mat., \textbf{53}, (1989), no. 1, 121--146; English translation: Math. USSR-Izv., \textbf{34}, (1990), no. 1, 121--145.



\bibitem{P2} J. Park, {\sl Regulators on additive higher Chow groups\/}, Amer. J. Math., \textbf{131}, (2009), no. 1, 257--276.


\bibitem{Park localization} J. Park, {\sl On localization for cubical higher Chow groups\/}, Tohoku Math. J., \textbf{75}, (2023), 251--281.


\bibitem{Park presentation} J. Park, {\sl Geometric presentations of Milnor $K$-groups of certain Artin algebras and Bass-Tate-Kato norms\/}, preprint 2025, arXiv:2511.22489v1 [math.AG]


\bibitem{Park AMS} J. Park, {\sl Calculus of absolute K\"ahler differential forms and Milnor $K$-theory\/}, in $K$-theory and algebra, analysis and topology, Proceedings of the ICM2018 Satellite conference on $K$-theory, Contempo. Math., American Math. Soc. Providence, RI, \textbf{749}, (2020), pp. 327--351.


\bibitem{Park MZ} J. Park, {\sl Motivic cohomology of fat points in Milnor range via formal and rigid geometries\/}, Math. Z., \textbf{302}, (2022), 1679--1719.






\bibitem{PU Milnor} J. Park and S. \"Unver, {\sl Motivic cohomology of fat points in Milnor range\/}, Doc. Math., \textbf{23}, (2018), 759--798.

\bibitem{R} K. R\"ulling, {\sl The generalized de Rham-Witt complex over a field is a complex of zero-cycles\/}, J. Algebr. Geom., \textbf{16}, (2007), no. 1, 109--169.

\bibitem{RS} K. R\"ulling and S. Saito, {\sl Higher Chow groups with modulus and relative Milnor $K$-theory\/}, Trans. Amer. Math. Soc., \textbf{370}, (2018), no. 2, 987--1043.





\bibitem{Totaro} B. Totaro, {\sl Milnor $K$-theory is the simplest part of algebraic K-theory\/}, $K$-Theory, \textbf{6}, (1992), no. 2, 177--189. 

\bibitem{Unver JAG} S. \"Unver, {\sl Infinitesimal Chow dilogarithm\/}, J. Algebr. Geom., \textbf{30}, (2021), no. 3, 529--571.

\bibitem{vdK} W. van der Kallen, {\sl Le $K_2$ des nombres duaux\/}, C. R. Acad. Sci. Paris, S\'er. A, \textbf{273}, (1971), 1204--1207.

\bibitem{vdK2} W. van der Kallen, {\sl The $K_2$ of rings with many units\/}, Ann. Sci. \'Ec. Norm. Sup\'er. S\'er. 4, \textbf{10}, (1977), no. 4, 473--515.

\bibitem{vdK ICM} W. van der Kallen, {\sl Generators and relations in algebraic $K$-theory\/} in Proceedings of Internat. Congress Math. (Helsinki, 1978), Vol I, 305--310, Academia Scientiarum Fennica, Helsinki, 1980.


\end{thebibliography}
\end{document}